\documentclass[11pt, reqno, oneside]{amsart}

\newcommand{\hkra}{\hookrightarrow}

\newcommand{\ra}{\rightarrow}

\newcommand{\rimp}{\Rightarrow}

\newcommand{\set}[1]{\left\{ #1 \right\}}

\newcommand{\cc}[1]{\overline{#1}}

\newcommand{\union}[2]{\bigcup\limits_{#1}{#2}}%union
\newcommand{\inter}[2]{\bigcap\limits_{#1}{#2}}%intersection
\newcommand{\sub}{\subset}%subset
\newcommand{\super}{\supseteq}

\newcommand{\sm}{\ensuremath{\setminus}}%set difference
\newcommand{\s}[2]{\sum\limits_{#1}{#2} }
\newcommand{\p}[2]{\prod\limits_{#1}{#2} }
\newcommand{\g}{\circ}%group operation
\newcommand{\units}{^\times}%units
\newcommand{\inv}{^{-1}}
\newcommand{\dS}[2]{\bigoplus\limits_{#1}{#2}}

\newcommand{\concat}{\ast}

\newcommand{\cl}{\colon}

\newcommand{\djun}[1]{\bigsqcup\limits_{#1}}

\newcommand{\lbr}[1]{\Bigl(#1\Bigr)}

\newcommand{\emst}{\emptyset}

\newcommand{\xra}{\xrightarrow}

\newcommand{\gk}{\mathfrak{g}}
\newcommand{\inpr}[2]{\langle {#1},{#2} \rangle}
\newcommand{\maxi}[2]{\max\limits_{#1}{#2}}

\newcommand{\pd}{\partial}

\newcommand{\eva}{\normalfont\text{ev}}

\newcommand{\lspan}[1]{\langle {#1}\rangle}
\newcommand{\pf}{\,\pitchfork\,}

\newcommand{\conn}{\nabla}

\newcommand{\hpd}{\cc{\partial}}

\newcommand{\chml}{\check{H}}

\newcommand{\inn}{^\mathrm{o}}

\newcommand{\dsm}{\oplus}

\newcommand{\fcolim}{\varinjlim\limits}

\newcommand{\lc}{\underline}

\newcommand{\ucc}{^{\bullet}}

\newcommand{\dul}{^\vee}

\newcommand{\vfc}[1]{[#1]^{\virt}}

\newcommand{\fcl}[1]{[#1]}

\newcommand{\rht}{\;|\;}

\newcommand{\ide}{\text{id}}

\newcommand{\Pt}{\normalfont\text{pt}}
\newcommand{\im}{\normalfont\text{im}}
\newcommand{\reg}{\normalfont\text{reg}}

\newcommand{\PGL}{\normalfont\text{PGL}}
\newcommand{\PU}{\normalfont\text{PU}}

\newcommand{\coker}{\normalfont\text{coker}}
\newcommand{\pr}{\normalfont\text{pr}}

\newcommand{\Mbar}{\overline{\mathcal M}}

\newcommand{\vdim}{\text{vdim}}
\newcommand{\delbar}{\bar\partial}
\newcommand{\del}{\partial}
\newcommand{\Hom}{\normalfont\text{Hom}}

\newcommand{\Aut}{\normalfont\text{Aut}}
\newcommand{\bC}{\mathbb{C}}
\newcommand{\bD}{\mathbb{D}}
\newcommand{\bE}{\mathbb{E}}

\newcommand{\bL}{\mathbb{L}}

\newcommand{\bN}{\mathbb{N}}

\newcommand{\bP}{\mathbb{P}}
\newcommand{\bQ}{\mathbb{Q}}
\newcommand{\bR}{\mathbb{R}}

\newcommand{\bT}{\mathbb{T}}

\newcommand{\bZ}{\mathbb{Z}}
\newcommand{\cB}{\mathcal{B}}
\newcommand{\cC}{\mathcal{C}}
\newcommand{\cD}{\mathcal{D}}
\newcommand{\cE}{\mathcal{E}}
\newcommand{\cF}{\mathcal{F}}
\newcommand{\cG}{\mathcal{G}}
\newcommand{\cH}{\mathcal{H}}

\newcommand{\cJ}{\mathcal{J}}
\newcommand{\cK}{\mathcal{K}}
\newcommand{\cL}{\mathcal{L}}
\newcommand{\cM}{\mathcal{M}}
\newcommand{\cN}{\mathcal{N}}
\newcommand{\cO}{\mathcal{O}}
\newcommand{\cP}{\mathcal{P}}

\newcommand{\cS}{\mathcal{S}}
\newcommand{\cT}{\mathcal{T}}
\newcommand{\cU}{\mathcal{U}}
\newcommand{\cV}{\mathcal{V}}
\newcommand{\cW}{\mathcal{W}}

\newcommand{\fE}{\mathfrak{E}}

\newcommand{\fL}{\mathfrak{L}}
\newcommand{\fM}{\mathfrak{M}}

\newcommand{\fW}{\mathfrak{W}}
\newcommand{\fX}{\mathfrak{X}}

\newcommand{\fd}{\mathfrak{d}}

\newcommand{\ff}{\mathfrak{f}}
\newcommand{\fg}{\mathfrak{g}}

\newcommand{\fj}{\mathfrak{j}}

\newcommand{\fl}{\mathfrak{l}}

\newcommand{\fs}{\mathfrak{s}}
\newcommand{\ft}{\mathfrak{t}}
\newcommand{\fu}{\mathfrak{u}}

\newcommand{\codim}{\normalfont{\text{codim}}}

\newcommand{\virt}{\normalfont\text{vir}}
\newcommand{\rank}{\normalfont\text{rank}}

\newcommand{\wt}{\widetilde}
\newcommand{\wh}{\widehat}

\newcommand{\pcd}{\normalfont\text{PD}}

\usepackage[utf8]{inputenc}
\usepackage[english]{babel}
\usepackage{amsmath}
\usepackage{amssymb}
\usepackage{mathabx}
\usepackage{amsfonts}
\usepackage{amsthm}
\usepackage[shortlabels]{enumitem}
\usepackage[dvipsnames]{xcolor}
\usepackage{tikz-cd}
\usepackage{subfiles}
\usepackage[scr]{rsfso}
\usepackage{extarrows}
\usepackage[pagebackref=true]{hyperref}
\usepackage{longtable}
\usepackage{appendix}
\usepackage{bm}
\hypersetup{
    colorlinks=true,
    linkcolor=Blue,
    filecolor=red,      
    urlcolor=teal,
    citecolor=RoyalBlue,}

\newtheorem{theorem}{Theorem}[section]
\newtheorem{lemma}[theorem]{Lemma}
\newtheorem{corollary}[theorem]{Corollary}
\newtheorem{proposition}[theorem]{Proposition}

\theoremstyle{definition}
\newtheorem{definition}[theorem]{Definition}
\newtheorem{construction}[theorem]{Construction}
\theoremstyle{remark}
\newtheorem{remark}[theorem]{Remark}

\newtheorem{example}[theorem]{Example}
\newtheorem*{notation*}{Notation}

\numberwithin{equation}{subsection}

\makeatletter
\@addtoreset{equation}{section}
\@addtoreset{theorem}{section}
\makeatother

\DeclareMathOperator{\Pic}{Pic}

\newcommand{\Addresses}{{% additional braces for segregating \footnotesize
		\bigskip
		\footnotesize
		\textsc{Université Paris Cité, Sorbonne Université, CNRS, IMJ-PRG, F-75005 Paris, France}\par\nopagebreak
		\text{ORCID}: \texttt{0000-0002-2392-7875}}}

\begin{document}

\title[GW invariants defined via global Kuranishi charts]{Properties of Gromov--Witten Invariants defined via Global Kuranishi Charts}
\author{Amanda Hirschi}

\maketitle
\begin{abstract}
  Using the global Kuranishi charts constructed by Hirschi--Swaminathan, we define gravitational descendants and equivariant Gromov--Witten invariants for general symplectic manifolds. We prove that these invariants satisfy the axioms of Kontsevich and Manin and their generalisations. A virtual localisation formula holds in this setting; we use it to derive an explicit formula for the equivariant Gromov--Witten invariants of Hamiltonian GKM manifolds. In particular, the symplectic Gromov--Witten invariants of smooth toric varieties agree with their algebro-geometric counterpart. In the semipositive case, the invariants studied here recover those of Ruan and Tian.\end{abstract}

\tableofcontents
\section{Introduction}

\subsection{Background}
The celebrated work \cite{Gr85} of Gromov showed that pseudoholomorphic curves provide a powerful tool in symplectic topology. In particular, while symplectic manifolds are a strictly more general class than K\"ahler manifolds, they exhibit many of the rigid properties of the latter. To this day, most symplectic invariants are defined using pseudoholomorphic curves; they figure prominently in Floer theory, \cite{Fl88,Sa99}, as well as in contact topology, e.g., in symplectic field theory \cite{EGH00} or contact dynamics \cite{HZ99}.\par
The \emph{Gromov--Witten invariants} of a closed symplectic manifold $(X,\omega)$ are one of the first symplectic pseudoholomorphic curve invariants to be defined, see \cite{RT95,RT97}. They are based on counts of $J$-holomorphic curves for an $\omega$-tame almost complex structure $J$. However, due to transversality issues, these invariants were first only defined for semipositive symplectic manifolds, which constrained certain singularities of the moduli space of pseudoholomorphic curves, \cite[Chapter 6-7]{MS12}. In general, abstract perturbations of the Cauchy--Riemann equation are required to deal with the non-regularity of the moduli space. The first virtual frameworks were developed by \cite{FO99,LT98}, defining Gromov--Witten invariants for general closed symplectic manifolds. Later frameworks were constructed in \cite{IP13,CM,Ger13,P16,MW17b,HWZ17,Joy17,IP19}. The aim of each approach, executed in different ways, is to construct a replacement for the virtual fundamental class for the \emph{moduli space of stable $J$-holomorphic maps} $\Mbar_{g,n}^{\,J}(X,A)$. These moduli spaces are solution spaces of the Cauchy-Riemann equation, an elliptic partial differential equation. Early on, nice local models were constructed. However, the construction of a virtual fundamental class required delicate local-to-global arguments to pass from this local information to a global invariant.\par
Despite the lack of a virtual fundamental class at that time, Kontsevich and Manin conjectured in \cite{KM94} that these invariants define a cohomological field theory on the (rational) cohomology of the target manifold $X$. In genus $0$, this reduces to a deformation of the ordinary cup product, called the quantum cup product. 
This structure stems from the rich geometry of the moduli spaces $\Mbar_{g,n}$ of stable curves, which was originally described in the seminal paper \cite{Wi91}. Kontsevich used the axioms to prove his famous recursion formula, giving counts of rational curves in $\bC P^2$ of any degree. They have also been extensively used for applications in symplectic geometry, such as in \cite{Sei99,LMP99,McD09,Ush11,HW24}. The rich algebraic structure of Gromov--Witten invariants was further developed by Dubrovin \cite{Dub14} and Givental \cite{Giv01,Giv01b,Giv96} among many others.

\subsection{Global Kuranishi charts} In \cite{AMS21}, the authors achieved a breakthrough by constructing a particularly simple representation for the moduli spaces of stable maps in genus $0$: that of a global Kuranishi chart. In \cite{HS22} and independently \cite{AMS23}, global Kuranishi charts were constructed for moduli spaces of stable maps of arbitrary genus. These constructions allow for a straightforward definition of a virtual fundamental class.

\begin{definition}A \emph{global Kuranishi chart} $(G,\cT,\cE,\fs)$ for a space $Z$ consists of a compact Lie group $G$ acting almost freely and locally linearly on the manifold $\cT$ and the vector bundle $\cE\to \cT$, with an equivariant section $\fs\cl\cT\to \cE$ satisfying $\fs\inv(0)/G\cong Z$. If $Z$ is compact, the associated \emph{virtual fundamental class} $\vfc{Z}\in \chml^{\vdim}(Z;\bQ)\dul$ is the composite
\begin{equation*}\label{eq:vfc-de}
    \chml^{\vdim}(Z;\bQ)\xra{\fs^*\tau(\cE/G)} H^{\dim(\cT/G)}_c(\cT/G;\bQ)\xra{\pcd} H_0(\cT;\bQ)\to \bQ.
\end{equation*}
Here $\vdim = \dim(\cT/G)- \rank(\cE)$, while $\tau(\cE/G)$ is the Thom class of the orbibundle $\cE/G$ and $\pcd$ denotes Poincar\'e duality.
\end{definition}
	
To motivate this definition of the virtual fundamental class, note that for any section $s$ of a finite-dimensional vector bundle $V\to M$, the pullback $s^*\tau(V)$ is a natural invariant of $\fs\inv(0)$. In the case where $s$ intersects the zero section transversely, we have $s^*\tau(V) = \pcd(\fs\inv(0))$ and integration over $\fs\inv(0)$ is precisely a special case of this virtual fundamental class.

Thus, we can define the \emph{Gromov--Witten homomorphism} $$\normalfont\text{GW}^{X,\omega}_{g,n,A}\cl H^*(X^n;\bQ)\to H^*(\Mbar_{g,n};\bQ)$$ by 
$$\normalfont\text{GW}^{X,\omega}_{g,n,A}(\alpha) = \pcd(\text{st}_*(\eva^*\alpha\cap \vfc{\Mbar_{g,n}^{\,J}(X,A)}))$$
using either \cite{HS22} or \cite{AMS21}. Here $\eva$ denotes the evaluation at the marked points, while the map $\text{st}\cl \Mbar_{g,n}^{\,J}(X,A)\to \Mbar_{g,n}$ forgets the map to $X$ and stabilises the domain. The Gromov--Witen invariants of $X$ are recovered by evaluating these cohomology classes $\normalfont\text{GW}^{X,\omega}_{g,n,A}(\alpha)$ on $[\Mbar_{g,n}]$.

\subsection{Main results} The purpose of this paper is to set up and prove extensions and properties of the invariants constructed in \cite{HS22}. Any further mention of Gromov--Witten invariants will refer to this construction unless specified otherwise. Given their usefulness for both foundational results as well as applications, we prove that these invariants indeed define a cohomological field theory on $H^*(X;\bQ)$.

\begin{theorem}
	The Gromov--Witten classes of $(X,\omega)$ satisfy the Kontsevich-Manin axioms, listed below.
\end{theorem}
\medskip

\noindent\textit{(Effective)} If $\inpr{[\omega]}{A} < 0$, then $\normalfont\text{GW}^{X,\omega}_{g,n,A} = 0$.

\bigskip

\noindent\textit{(Homology)} $\normalfont\text{GW}^{X,\omega}_{g,n,A}$ is induced by a homology class.

\bigskip

\noindent\textit{(Grading)} $\normalfont\text{GW}^{X,\omega}_{g,n,A}$ has degree $2((\dim_\bC X-3)(1-g) + \inpr{c_1(T_X)}{A} + n)$. 

\bigskip

\noindent\textit{(Symmetry)} $\normalfont\text{GW}^{X,\omega}_{g,n,A}$ is equivariant with respect to the canonical $S_n$-actions given by permuting the factors, respectively the marked points.

\bigskip

\noindent\textit{(Mapping to a point)} If $\bE$ denotes the Hodge bundle over $\Mbar_{g,n}$, then
$$\normalfont\text{GW}^{X,\omega}_{g,n,0}(\alpha_1,\dots,\alpha_n) = \inpr{\alpha_1\cdot \dots \cdot \alpha_n}{c_{\dim_\bC(X)}(T_X)\cap [X]} \; c_g(\bE^*)$$
\indent for all $\alpha_i \in H^*(X;\bQ).$
\bigskip

\noindent\textit{(Fundamental class)} If $1_X$ denotes the unit of $H^*(X;\bQ)$ and $\pi_n$ the map that forgets the $n^{\text{th}}$ marked point, then 
$$\normalfont\text{GW}^{X,\omega}_{g,n,A}(\alpha_1,\dots,\alpha_{n-1}, 1_X) = \pi_n^*\normalfont\text{GW}^{X,\omega}_{g,n-1,A}(\alpha_1,\dots,\alpha_{n-1}).$$
for any $\alpha_1,\dots,\alpha_{n-1}\in H^*(X;\bQ)$.

\bigskip

\noindent\textit{(Divisor)} If $|\alpha_n| = 2$ and ${\pi_n}_! = \pcd\g {\pi_n}_*\g \pcd$ denotes the exceptional pushforward, then 
$${\pi_n}_!\, \normalfont\text{GW}^{X,\omega}_{g,n,A}(\alpha_1,\dots,\alpha_n) = \inpr{\alpha_n}{A} \, \normalfont\text{GW}^{X,\omega}_{g,n-1,A}(\alpha_1,\dots,\alpha_{n-1}).$$
for any $\alpha_1,\dots,\alpha_{n-1} \in H^*(X;\bQ)$. 

\bigskip

\noindent\textit{(Splitting)} Write $\pcd(\Delta_X) = \s{k \in K}{\gamma_k\times\gamma'_k}$ with $\gamma_k,\gamma'_k\in H^*(X;\bQ)$. Given $S\subset \{1,\dots,n\}$, let $$\varphi_S \cl \Mbar_{g_0,n_0+1}\times\Mbar_{g_1,n_1+1}\to \Mbar_{g,n}$$ be the associated clutching map. Then
$$\varphi_S^* \normalfont\text{GW}^{X,\omega}_{g,n,A}(\alpha_1,\dots,\alpha_n) = (-1)^{\epsilon(\alpha;S)}\s{\substack{A_0+A_1 = A\\k\in K}}{\normalfont\text{GW}^{X,\omega}_{g_0,n_0+1,A_0}((\alpha_i)_{i\in S}, \gamma_k)\,\normalfont\text{GW}^{X,\omega}_{g_1,n_1+1,A_1}(\gamma'_k,(\alpha_j)_{j\notin S})}$$
where $\epsilon(\alpha;S) = |\{i > j\rht i \in S,\, j \notin S,\,|\alpha_i|,|\alpha_j| \in 1+ 2\bZ\}|$ and $\alpha_i \in H^*(X;\bQ)$

\bigskip

\noindent\textit{(Genus reduction)} If $\psi \cl \Mbar_{g,n}\to \Mbar_{g+1,n-2}$ denotes the map that creates a non-separating node by gluing the last two marked points, then 
$$\psi^*\normalfont\text{GW}^{X,\omega}_{g+1,n-2,A}(\alpha) = \normalfont\text{GW}^{X,\omega}_{g,n,A}(\alpha,\pcd(\Delta_X))$$ 
for any $\alpha\in  H^*(X^{n-2};\bQ)$.

\bigskip

 See \cite{AMS23} for the analogous results for Gromov--Witten invariants valued in certain complex-oriented generalised cohomology theories. One can generalise the above-defined Gromov--Witten invariants by integrating natural cohomology classes on the moduli space itself over the virtual fundamental class. In \textsection\ref{subsec:grav-desc}, we define the so-called gravitational descendants, discussed in \cite{Wi91,Giv96,RT97,KM98}. We prove that these invariants satisfy the famous String, Dilaton and Divisor equation.\\

In \textsection\ref{sec:localisation}, we define the notion of an \emph{equivariant virtual fundamental class} associated to a global Kuranishi chart equipped with a suitable group action. Subsequently, we apply this to moduli spaces of stable pseudoholomorphic maps whose target $X$ admits a Hamiltonian group action to construct equivariant Gromov-Witten invariants.

\begin{proposition}
    Let $(X,\omega,\mu)$ be a Hamiltonian $K$-manifold, where $K$ is a compact connected Lie group. If $J$ is $K$-invariant, we can construct a global Kuranishi chart for $\Mbar_{g,n}^{\,J}(X,A)$ that carries a compatible $K$-action. The \emph{equivariant Gromov--Witten homomorphisms} of $(X,\omega,\mu)$ are the $H^*_K$-linear maps
    $$\normalfont\text{GW}^{X,\omega,\mu}_{g,n,A} \cl H^*_K(X^n;\bQ)\to H^*(\Mbar_{g,n};\bQ)\otimes H^*_K(\Pt,\bQ)$$
    defined as in the non-equivariant setting. They define a cohomological field theory on $H^*_K(X;\bQ)$.  
\end{proposition}
 %:= (\eva\times \normalfont\text{st})_*\vfc{\Mbar_{g,n}^{\,J}(X,A)}_{K}\cl H^*_K(X^n\times \Mbar_{g,n};\bQ)\to H^*_K(*;\bQ).$$

Kontsevich observed in \cite{Kon95} that one could use localisation in order to compute Gromov--Witten invariants. The extension of Gromov--Witten theory to the equivariant setting was first systematically explored by Givental, \cite{Giv96}, and used to prove mirror symmetry for Calabi-Yau complete intersections. A construction in the algebraic setting of the equivariant invariants for smooth toric varieties was given in \cite{GP99}. The authors generalise the localisation formula of \cite{AB84} to Deligne-Mumford stacks with equivariant perfect obstruction theories; see \cite{Beh99} for an alternative proof. They use it to derive a combinatorial formula for the equivariant Gromov--Witten invariants of projective spaces. It allowed for the computation of Gromov--Witten invariants of toric and compact homogeneous varieties; such as in \cite{Liu13,Spi99}, as well as toric fibrations, \cite{Br14}. Givental, \cite{Giv01b}, used it to obtain conjectural formulas for the higher genus Gromov--Witten invariants in terms of genus $0$ invariants in the case of generically semisimple quantum cohomology; this was shown by Teleman in \cite{Te12}. Meanwhile, \cite{GiKi95} computed the small quantum cohomology of flag varieties using the localisation formula. For a localisation formula in the symplectic setting, see \cite{CT10}.\par 
We prove a virtual localisation formula, Theorem \ref{Virtual Localisation}, for spaces admitting equivariant global Kuranishi charts. Under certain assumptions on the Hamiltonian torus action, we can use it to derive a combinatorial formula for its equivariant Gromov--Witten invariants, generalising \cite[Theorem 78]{Liu13}. Explicitly, we consider Hamiltonian torus manifolds $(X,\omega)$ that satisfy the \emph{GKM condition}, see Definition \ref{de:gkm} for details. In particular, $X$ has finitely many fixed points, and the union of all at most $1$-dimensional orbits is a union of spheres.
 
\begin{theorem}[Theorem \ref{thm:formula-eq-gw}] If $X$ admits a Hamiltonian GKM action by a torus $\bT$, its equivariant Gromov--Witten invariants satisfy
	\begin{equation}\label{eq:formula-equ-gw}
	\normalfont\text{GW}^{X,\omega,\mu}_{g,n,A}(\alpha_1,\dots,\alpha_n)([\Mbar_{g,n}]) = \s{\Gamma\in G_{g,n}(X,A)}{\frac{1}{|\Aut(\Gamma)|}\lspan{j_{\Gamma}^*\eva^*\alpha\cdot\textbf{w}(\Gamma),\fcl{\Mbar_\Gamma}_\bT}}
	\end{equation}
	in the fraction field of $H^*_\bT(\Pt;\bQ)$. Here, $G_{g,n}(X,A)$ is a collection of decorated graphs, $\Mbar_\Gamma$ is a product of moduli spaces of stable curves, and $\textbf{w}(\Gamma)\in H^*_\bT(\Mbar_\Gamma;\bQ)$ is a weight associated to the graph $\Gamma$.
\end{theorem}

The class of Hamiltonian GKM torus manifolds contains not only all symplectic toric manifolds but also many non-K\"ahler examples, \cite{Tol98,Wo98,GKZ20,GKZ23}. Their symplectic geometry has been studied in \cite{GZ00,GZ01,GKZ20b,GKZ22}. To any such manifold, one can associate its \emph{GKM graph}, which combinatorially encodes information about the action and the topology of the manifold. In particular, it allows for a complete determination of the equivariant cohomology, \cite[Proposition 2.30]{GKZ22b}. As the right hand side of \eqref{eq:formula-equ-gw} only relies on the GKM graph of the Hamiltonian torus manifold, so do its equivariant Gromov--Witten invariants.

\begin{theorem}\label{} Suppose $(X,\omega,\mu)$ and $(X',\omega',\mu')$ are two Hamiltonian torus manifolds satisfying the GKM condition. If their GKM graphs are isomorphic, then their equivariant Gromov--Witten invariants agree.
\end{theorem}

In genus $0$, this can be seen as a quantum version of the fact that the GKM graph determines the equivariant cohomology of the manifold, \cite[Proposition~2.30]{GKZ22b}. Nonetheless, it is somewhat unexpected, as Hamiltonian GKM manifolds with the same GKM graph need not even be homotopy equivalent, \cite[Theorem 1.3]{GKZ22}.
Since any symplectic toric manifold is symplectomorphic to a smooth toric variety, \cite[Theorem 2.1]{Del88}, we have the following comparison result.

\begin{corollary}\label{} The symplectic Gromov--Witten invariants of a toric symplectic manifold agree with its algebraic Gromov--Witten invariants.
\end{corollary}

In \cite{LiSh17}, the equivariant Gromov--Witten invariants of algebraic GKM stacks were computed, generalising \cite{Liu13} in a different direction than we do here.\par 
We point out that the equivariant Gromov--Witten invariants considered in this paper do not agree with the Hamiltonian Gromov--Witten invariants defined in \cite{Mir03,CGMS02} or \cite{TX17}, which are defined using the vortex equation instead. A construction of quantum Kirwan maps exists in the setting of projective varieties, \cite{GoWo22} and under certain assumptions in the symplectic setting, \cite{Xu24}. The quantum Kirwan map relates the equivariant quantum cohomology of a toric variety to the quantum cohomology of its symplectic reduction. See also \cite{NWZ14}. It would be interesting to see whether this could be generalised further in the symplectic setting.\par

Finally, we compare the Gromov--Witten invariants considered here with the Gromov--Witten invariants defined by Ruan and Tian in \cite{RT97}. The latter ones are only defined if the symplectic target is \emph{semipositive}, that is, if any $J$-holomorphic sphere $u$ with negative first Chern number satisfies $c_1(u) < 3-\dim_\bC(X)$. This topological property ensures a certain control over the singularities of the moduli space for stable maps and allows for a pseudocycle definition of Gromov--Witten invariants, which will be recalled in \textsection\ref{sec:compare-pseudocycles}. 
 
 \begin{theorem}[Theorem \ref{pseudocycle-comparison}]\label{thm:pseudocycle} If $(X,\omega)$ is semipositive, the Gromov--Witten invariants considered here agree with the Gromov--Witten invariants of Ruan-Tian.
 \end{theorem}

The following immediate consequence of this result is not at all obvious from the global Kuranishi chart construction itself.

\begin{corollary} If $(X,\omega)$ is semipositive, then $\normalfont\text{GW}^{X,\omega}_{0,n,A}$ is integral for $n \geq 3$.
\end{corollary}
 
 In \cite{Schm23}, Schmaltz shows that the polyfold Gromov--Witten invariants of \cite{HWZ17} specialise to the genus $0$ Gromov--Witten invariants defined in \cite{MS12} via pseudocycles. Ionel and Parker show in \cite{IP19b} that the analogue of Theorem \ref{thm:pseudocycle} for their relative virtual fundamental class. For comparison results between algebraic and symplectic Gromov--Witten invariants, see \cite{Sie99,LT99}.

\subsection*{Acknowledgments} I am grateful to Ailsa Keating for her support and feedback, and to Nick Sheridan and Ivan Smith for extensive comments on an earlier draft. I thank Penka Georgieva, Andrin Hirschi, Kai Hugtenburg, Chiu-Chu Melissa Liu, Alexandru Oancea, Noah Porcelli, Oscar Randal-Williams, Dhruv Ranganathan, Mohan Swaminathan, Paul Seidel, and Guangbo Xu for valuable discussions and suggestions. I thank the referee for clarifying comments. My research was supported by EPSRC scholarship No.2434378 and is currently funded by ERC grant No.864919.

\medskip
\section{Review of the global Kuranishi chart construction}\label{sec:background}

Let $(X,\omega)$ be a closed symplectic manifold with $A\in H_2(X;\bZ)$ and fix $J \in \cJ_\tau(X,\omega)$. We recall the construction of a global Kuranishi chart for $\Mbar_{g,n}^{\,J}(X,A)$ described in \cite{HS22}. However, we will not go into too much detail and refer the reader to \cite{HS22} for a more thorough discussion.

\begin{definition}\label{aux-choices-defined}
	An \emph{auxiliary datum} $\alpha = (\nabla^X,\cO_X(1),p,\cU,\lambda,k)$ for $\Mbar_{g,n}^{\,J}(X,A)$ consists of
	\begin{enumerate}[(1),leftmargin=25pt,ref=\arabic*]
		\item a $J$-linear connection $\nabla^X$ on $T_X$;
		\item\label{polarization-on-target} a Hermitian line bundle $(\cO_X(1),\conn)\to X$ with curvature $F^\conn = -2\pi i \Omega$ for a symplectic form $\Omega$ taming $J$. We set $d \coloneq \langle[\Omega],A\rangle$.
		\item\label{framed-maps-etc} $p\gg 1$ is an integer. 
	\end{enumerate}

Given this, we abbreviate
$$m\coloneq p(2g-2+3d)\qquad\qquad  N\coloneq m-g$$
and set
\begin{equation}\label{eq:groups}  \cG\coloneq\PGL(N+1,\bC)\qquad\qquad G\coloneq\PU(N+1).\end{equation}
Here, given a smooth stable map $u \cl C\to X$ of genus $g$ representing $A$, we observe that $m$ is the degree of the $p^{\text{th}}$ power of the line bundle 
$$\fL_u \coloneq \omega_C \otimes u^*\cO_X(1)^{\otimes 3}.$$
By \cite[Lemma~4.1]{HS22}, if $A \neq 0$ or $g\geq 2$, we can choose $p$ sufficiently large so that $\fL_u^{\otimes p}$ is very ample, and has vanishing $H^1$ for any $[u,C]\in\Mbar_{g}^{\,J}(X,A)$. If $A = 0$ and $g < 2$, we have to allow for marked points and replace $\omega_C$ with $\omega_C(x_1+\dots + x_n)$.\par
In either case, $N = \dim_\bC H^0(C,\fL_u^{\otimes p})-1$ by the Riemann-Roch theorem. By definition, any basis $\cF$ of $H^0(C,\fL^{\otimes p}_u)$ induces a holomorphic map 
$$\iota_\cF \cl C\hkra \bP^N,$$ 
with non nontrivial automorphisms. It is regular because $H^1(C,\fL_u^{\otimes p})$ vanishes. We call any such map a \emph{framing}. Denote by 
\begin{equation}\label{eq:base-space} \Mbar_{g,n}^*(\bP^N,m)\sub \Mbar_{g,n}(\bP^N,m)\end{equation} 
the subspace of all regular automorphism-free holomorphic curves of genus $g$ of degree $m$ whose image is not contained in a hyperplane. These moduli spaces are smooth quasi-projective varieties of the expected dimension and will serve as our base spaces.\par
To define the remaining data of the auxiliary datum $\alpha$, let 
$$ \Mbar^{*,st}_{g,3d}(\bP^N,m) \sub  \Mbar_{g,3d}^*(\bP^N,m)$$
be the locus of curves with stable domain. Then,
\begin{enumerate}[leftmargin=25pt,ref=\arabic*]
	\setcounter{enumi}{3}
	\item $\cU$ is a good covering in the sense of \cite[Definition~3.10]{HS22};
	\item $\lambda$ is a $\PGL_\bC(N+1)$-equivariant map 
	$$\Mbar^{*,st}_{g,3d}(\bP^N,m)/S_{3d}\to \cG/G,$$
	where $\cG$ and $G$ were defined in \eqref{eq:groups};	
	\item $k\geq 1$ is an integer. 
\end{enumerate}
\end{definition}

We use perturbations based on the complex vector bundle 
\begin{equation}\label{}E \coloneq E_k \coloneq \cc{\Hom}_\bC(p_1^*T_{\bP^N},p_2^*T_X)\otimes p_1^*\cO_{\bP^N}(k) \otimes \cc{H^0(\bP^N,\cO(k))} \end{equation} 
over $\bP^N\times X$, where $\cc{H^0(\bP^N,\cO(k))}$ is the underlying real vector space of $H^0(\bP^N,\cO(k))$ equipped with the conjugate complex structure. Note that $E$ depends on $k$. Given a map $(\iota,u) \cl C\to \bP^N\times X$, we set $$E_{(\iota,u)} \coloneq H^0(C,(\iota,u)^*E).$$ 

We can now give the construction of the global Kuranishi chart, treating the case of $n = 0$ first and generalising to the other cases in Remark \ref{rem:adding-marked-points}.

\begin{construction}\label{high-level-description-defined}
		Given an unobstructed auxiliary datum $(\nabla^X,\cO_X(1),p,\cU,k)$, the associated global Kuranishi chart $\cK = (G,\cT,\cE,\fs)$ for $\Mbar_g^{\,J}(X,A)$ is defined as follows:
		\begin{itemize}[leftmargin=15pt]
			\item The \emph{thickening} $\mathcal T$ consists of tuples $(u,\iota,C,\eta,\alpha)$, up to reparametrisation, where
			\begin{enumerate}[(a),ref=\alph*]
				\item\label{thickening-condition-1} $u\cl C\to X$ is a framed stable map representing $A$ and $\iota \cl C\hkra \bP^N$ defines an element of $\Mbar_g^*(\bP^N,m)$ so that $(\iota,u)$ lies in the domain of $\lambda_\cU$;
				\item\label{thickening-condition-2} $\eta\in E_{(\iota,u)}$ satisfies 
				\begin{equation}\label{eq:perturbed-cr}
					\delbar_J\tilde u + \langle\eta\rangle\circ d\tilde{\iota} = 0
				\end{equation}
				where $\tilde u$ and $\tilde{\iota}$ denote the pullbacks to the normalisation $\tilde C$ of C;
				\item\label{thickening-condition-3} $\alpha\in H^1(C,\cO_C)$ is such that 
				\begin{equation*}\label{line-bundle-identity}
					[\iota^*\cO_{\bP^N}(1)]\otimes [\fL_u^{\otimes -p}]  = \exp(\alpha) 
				\end{equation*}
			where $\exp \cl H^1(C,\cO_C)\to \Pic^0(C)$ is the universal covering map.
                \item\label{regularity} $H^1(C,(\iota,u)^*E) = 0$ and $D(\delbar_J)_u\oplus (\lspan{\cdot}\g d\tilde{\iota})$ is surjective.
			\end{enumerate}
   
		  $G$ acts on $\cT$ via its standard action on $\bP^N$ making the forgetful map $\pi:\cT\to\Mbar_g^*(\bP^N,m)$ equivariant.
			\item The \emph{obstruction bundle} $\cE\to\cT$ has the fibre
			\begin{align}\label{obstruction-bundle-fibre-def}
				\cE_y = \fs\fu(N+1)\oplus E_{(\iota,u)} \oplus H^1(C,\cO_C)
			\end{align}
			 over a point $y= (u,\iota,C,\eta,\alpha)\in\cT$ and it is endowed with the natural lift of the $G$-action.
			\item The \emph{obstruction section} $\fs:\cT\to\cE$ is given by
			\begin{align}
				\fs(u,\iota,C,\eta,\alpha) = (\exp\inv\lambda_\cU(\iota,u),\eta,\alpha).
			\end{align}
			where $\exp \cl \fs\fu(N+1)\to \cG/G$ is induced by the exponential map and $\lambda_\cU\cl \cT\to \cG/G$ is a $G$-equivariant map defined in \cite[Equation~(3.2.1)]{HS22} using the good covering $\cU$ and $\lambda$. 
		\end{itemize}
\end{construction}

\begin{remark} Condition \eqref{regularity} ensures that all fibres of $\cE$ have the same dimension, while \eqref{eq:perturbed-cr} ensures that the thickening is cut out transversely.
\end{remark}

\begin{remark} The only purpose of the part $\lambda_\cU$ of the obstruction section is to select a class of `unitary framings'. The problem is that the choice of a very ample line bundle gives a full $\PGL$-orbit of framings. However, we only have a $\PU$-action on the thickening and thus need to further reduce the space of framings lying in the zero locus of the obstruction section.
\end{remark}

\begin{definition}
	An auxiliary datum is \emph{unobstructed} if the canonical forgetful map $$\fs\inv(0)\to \Mbar_{g}^J(X,A)$$ 
	is surjective.
\end{definition}	

The existence of unobstructed auxiliary data is shown in \cite[\textsection 4]{HS22}. By Proposition~5.1, $\cK =(G,\cT,\cE,\fs)$ obtained from Construction~\ref{high-level-description-defined} is a global Kuranishi chart for $\Mbar_g^{\,J}(X,A)$ with a canonical orientation of its thickening and obstruction bundle. Moreover, the canonical map $\pi\cl \cT\to \Mbar_g^*(\bP^N,m)$, remembering only the framing, is a topological submersion.
	
%	 following holds for any stable $J$-holomorphic map $u:C\to X$ in $\Mbar_g^{\,J}(X,A)$.
%	\begin{enumerate}[(a),leftmargin=20pt]
%		\item\label{very-ample-acyclic-line-bundle} The line bundle $\fL_u^{\otimes p}\to C$ is very ample and $H^1(C,\fL_u^{\otimes p}) = 0$.
%		\item\label{killing-obstructions} For any basis $\cF$ of $H^0(C,\fL^{\otimes p}_u)$ with $\lambda(\iota_F,u) = 0$, we have
%		\begin{enumerate}[(1),leftmargin=20pt,ref = \arabic*]
%			\item\label{constant-rank-obstruction-space} $H^1(C,E_{(\iota_\cF,u)}) = 0$;
%			\item\label{spanning-cokernel} the perturbed Cauchy-Riemann operator 
%			\begin{equation}\label{eq:linearised-perturbed-operator}
%				D(\delbar_J)_u\oplus(\langle\cdot\rangle\g d\tilde{\iota}_\cF)\cl \Omega^0(C,u^*T_X)\oplus E_{(\iota,u)}\to\Omega^{0,1}(\tilde C,\tilde u^*T_X)
%			\end{equation}
%			
%			is surjective, where $\langle\cdot\rangle$ is induced by the standard Hermitian metric on $\cO(k)$.
%		\end{enumerate}
%	\end{enumerate}

\begin{remark}\label{rem:adding-marked-points}
    We obtain a global Kuranishi chart for $\Mbar_{g,n}^{\,J}(X,A)$ by pulling back $\cK$ along the map
    $$\pi_{\lc{n}}\cl \Mbar_{g,n}^*(\bP^N,m)\coloneq \pi_{\lc{n}}\inv(\Mbar_g^*(\bP^N,m))\to \Mbar_g^*(\bP^N,m)$$ 
    that forgets all marked points.
\end{remark}

By \cite[\textsection6]{HS22}, the resulting virtual fundamental class $\vfc{\Mbar_{g,n}^{\,J}(X,A)}$ is independent of the choice of unobstructed auxiliary datum. We define the \emph{Gromov--Witten class} of $(X,\omega)$ associated to $(A,g,n)$ to be 
\begin{equation}\label{eq:gw-class-defined}
    \text{GW}^{X,\omega}_{g,n,A} \coloneq (\eva\times\text{st})_*\vfc{\Mbar_{g,n}^{\,J}(X,A)}\quad \text{in } H_{*}(X^n\times\Mbar_{g,n};\bQ)
\end{equation}
If $2g-2+ n \leq 0$ and $A\neq 0$, we formally treat $\Mbar_{g,n}$ as a point and use the same definition. We use the notation 
$$\lspan{\alpha_1,\dots,\alpha_n;\sigma}^{X,\omega}_{g,n,A} \coloneq \inpr{\alpha_1\times\dots\times\alpha_n\times\pcd(\sigma)}{\normalfont\text{GW}^{X,\omega}_{g,n,A}}$$
for $\alpha_1,\dots,\alpha_n\in H^*(X;\bQ)$ and $\sigma \in H_*(\Mbar_{g,n};\bQ)$. 

\section{Kontsevich--Manin axioms}\label{sec:axioms}

The Kontsevich-Manin axioms are a catchphrase used to describe the properties in \cite{KM94} that GW invariants are expected to satisfy. In contrast to the introduction, we will consider here the GW classes defined by \eqref{eq:gw-class-defined}. While the axioms are less elegant in terms of the GW classes, their proof is more transparent.

\subsection{The first five axioms}

The Effective, Homology, and Grading axioms follow directly from the construction. 

\begin{lemma}[Symmetry]\label{symmetry-axiom} We have
	\begin{equation*}\label{}\lspan{\alpha_{\tau(1)},\dots,\alpha_{\tau(n)};\tau_*\sigma}^{X,\omega}_{g,n,A} = (-1)^{\epsilon_{\tau,\alpha}}\lspan{\alpha_1,\dots,\alpha_n;\sigma}^{X,\omega}_{g,n,A}\end{equation*} 
	for any permutation $\tau\in S_n$ and classes $\alpha_i \in H^*(X;\bQ)$ and $\sigma\in H_*(\Mbar_{g,n};\bQ)$, where 
	$$\epsilon_{\tau,\alpha} = |\{i > j\mid \tau(i)< \tau(j),\;|\alpha_i|,|\alpha_j|\in 2\bZ +1\}|.$$
\end{lemma}

\begin{proof} Given a global Kuranishi chart $\cK_n$ as in Construction \ref{high-level-description-defined}, the $S_n$-action lifts to a continuous action by orientation-preserving rel--$C^\infty$ diffeomorphisms on the thickening and the obstruction bundle. As the equivariant Thom class of the obstruction bundle is $S_n$-invariant, so is the virtual fundamental class.
\end{proof}

\begin{lemma}[Mapping to a point]\label{mapping-to-point} If $2g-2+ n> 0$, then
	$$\vfc{\Mbar_{g,n}^{\,J}(X,0)} = c_{g\dim_\bC(X)}(T_X\boxtimes \bE\dul)\cap \fcl{X\times \Mbar_{g,n}}$$
	in $H_*(X\times\Mbar_{g,n};\bQ)$, where $\bE$ denotes the Hodge bundle over $\Mbar_{g,n}$. In particular, 
	$$\lspan{\alpha_1,\dots,\alpha_n;[\Pt]}^{X,\omega}_{0,n,0} = \lspan{\alpha_1\cdot \dots \cdot \alpha_n,[X]},$$
\end{lemma}

\begin{proof} Given a constant stable map $u\cl C\to X$ with image $x$, we have $$H^1(C,u^*T_X) = H^1(C,\cO_C)\otimes T_xX.,$$ 
	where the tensor product is taken over $\bC$ as we will always do. Thus, the cokernel of $D\delbar_J(u)$ has (real) dimension $2g\dim(X)$ and the obstruction bundle of $\Mbar_{g,n}^{\,J}(X,0) = X\times\Mbar_{g,n}$ is given by $\normalfont\text{Ob} \coloneq  T_X\boxtimes \bE^*$. Let $s_0$ denote its zero section. Fix a global Kuranishi chart $\cK_n = (G,\cT,\cE,\fs)$ as given by Construction \ref{high-level-description-defined} with base space $\cM\subset \Mbar_{g,n}(\bP^N,m)$. Denote by $\cC\to \cM$ the universal curve and set $\cL \coloneq  R^1\pi_*\cO_\cC$. Then $\cK_{\normalfont\text{ob}} \coloneq  (G,\cT_{\normalfont\text{ob}},\widetilde{\normalfont\text{Ob}},\tilde{\fs})$ with 
	$$\cT_{\normalfont\text{ob}} \coloneq  X\times \set{([\iota,C,x_1,\dots,x_n],\alpha)\in \cL \mid [\iota^*\cO(1)] = p\cdot [\omega_C(x_1+\dots + x_n)]+\alpha}$$
	is a global Kuranishi chart for $\Mbar_{g,n}^{\,J}(X,0)$. The obstruction bundle is $$\widetilde{\normalfont\text{Ob}} = \normalfont\text{Ob}\boxplus \cL\oplus \fs\fu(N+1)$$ and $\tilde{\fs}$ is given by the zero section in the first summand, the obvious map in the second one, and $\exp\inv \lambda$ in the last one. Let $j \cl \cT_{\normalfont\text{ob}}\hkra \cT$ be the inclusion. There exists a natural equivariant morphism $\Phi\cl j^*\cE\to \widetilde{\normalfont\text{Ob}}$ of complex vector bundles; it is given by the identity on $\fs\fu(N+1)$ and $\cL$ and maps the perturbation term $\eta$ to the image of $\lspan{\eta}\g d\iota$ under the quotient map $$\Omega^{0,1}_J(\tilde{C},\tilde{u}^*T_X)\to H^1(C,u^*T_X).$$ 
	By the construction of $\cK_n$, the map $\Phi$ is surjective. Moreover, its kernel agrees with the normal bundle $ N_{X\times \cM/\cT}$ of $X\times\cM$ in $\cT$ (as rel--$C^\infty$ manifolds over $\cM)$. Fixing a splitting $L \cl \widetilde{\normalfont\text{Ob}}\to \cE|_{\cT'}$ of $\Phi$ we obtain that the two-term complexes associated to $\cK_{\normalfont\text{ob}}$ and $\cK_n$ respectively, are quasi-isomorphic in the sense of Lemma \ref{quasi-isomorphic-global-charts}, whence the claim follows.
\end{proof}

\begin{remark}\label{vfc-if-obstruction-bundle} This argument can be applied in any situation where $\Mbar_{g,n}^{\,J}(X,A)$ is smooth with obstruction bundle $\text{Ob}$ to see that $$\vfc{\Mbar_{g,n}^{\,J}(X,A)} = e(\text{Ob}) \cap \fcl{\Mbar_{g,n}^{\,J}(X,A)}$$ 
	under the identification of the dual of \v{C}ech cohomology with singular homology.\end{remark}

We observe the following vanishing statement, alluded to in \cite{KM94}. See \cite[Proposition 2.14(3)]{RT97} for the same statement in their setting.

\begin{lemma}\label{lem:vanishing-dimension} If $(1-g)(\dim_\bC(X)-3) + 2\inpr{c_1(T_X)}{A} < 0$, then $\normalfont\text{GW}^{X,\omega}_{g,n,A} = 0$ for any $n \geq 0$.
\end{lemma}

\begin{proof} Let $(G,\cT,\cE,\fs)$ be a global Kuranishi chart for $\Mbar_g^{\,J}(X,A)$ and let $(G,\cT_n,\cE_n,\fs_n)$ be the induced global Kuranishi chart for $\Mbar_g^{\,J}(X,A)$. By construction, there exist equivariant maps $\pi_n\cl \cT_n \to \cT$ and $\tilde{\pi}_n \cl \cE_n = \pi_n^*\cE \to \cE$ satisfying $\tilde{\pi}_n\fs_n = \fs_n\pi_n$. As $|\fs^*\tau_{\cE/G}| > \dim(\cT/G)$, it follows that $\fs_n^*\tau_{\cE_n/G} = \pi_n^*\fs^*\tau_{\cE/G} = 0$.   \end{proof}

\subsection{Fundamental class axiom}
Following \cite{KM94}, we call the pair $(g,n)$ of nonnegative integers \emph{basic} if it is an element of $\{(0,3),(0,1),(0,2)\}$.

\begin{proposition}[Fundamental class]\label{fundamental-class-axiom}  Suppose $n\geq 1$ and $(g,n)$ is not basic. Then, for $\alpha_1,\dots,\alpha_{n-1}\in H^*(X;\bQ)$ and $\sigma \in H_*(\Mbar_{g,n};\bQ)$, we have  
	$$\lspan{\alpha_1,\dots,\alpha_{n-1},1_X;\sigma}^{X,\omega}_{g,n,A} = \lspan{\alpha_1,\dots,\alpha_{n-1};{\pi_n}_*\sigma}^{X,\omega}_{g,n-1,A},$$
 where $1_X$ is the unit of $H^*(X;\bQ)$.
\end{proposition}

\begin{proof}  Let $\cK$ be a global Kuranishi chart for $\Mbar_g^{\,J}(X,A)$ as given by Construction \ref{high-level-description-defined}. Let $\cK_n$ and $\cK_{n-1}$ be the induced global Kuranishi charts for $\Mbar_{g,n}^{\,J}(X,A)$ and $\Mbar_{g,n-1}^{\,J}(X,A)$. Denote by $\cM_n$ and $\cM_{n-1}$ the respective base space. Let $\pi_n \cl \cT_n \to \cT_{n-1}$ be the forgetful map between the thickenings. By construction, $\cE_n = \pi_n^*\cE_{n-1}$ and $\fs_n = \pi_n^*\fs_{n-1}$.
We will construct principal $H$-bundles $\wt\cT_{n-1}\to \cT_{n-1}$ and $\wt\cT_n\to \cT_n$ for some compact Lie group $H$ so that the induced forgetful map $\wt\pi_n \cl \wt\cT_n \to \wt\cT_{n-1}$ satisfies
\begin{equation}\label{eq:fund-class-basic} (\wt\pi_n)_!\,\normalfont\text{st}^* = \normalfont\text{st}^*\,{\pi_n}_!.\end{equation}
This will prove the claim since the global Kuranishi chart $\wt\cK_\ell$ determined by the principal bundle $\wt\cT_\ell\to \cT_\ell$ is equivalent to $\cK_\ell$ via group enlargement. To this end, let $P_{n-1}/H$ be a presentation of $\Mbar_{g,n-1}$ as a global quotient and let $P_n$ be the pullback of $P_{n-1}$ along the representable morphism $\pi_n$. Then $P_n/H$ is a presentation of $\Mbar_{g,n}$ and $\pi_n$ pulls back to a smooth map $P_n \to P_{n-1}$. Set 
$$\wt\cT_\ell \coloneq  \cT_\ell\times_{\Mbar_{g,\ell}} P_\ell$$ for $\ell\in \{n-1,n\}$. By Lemma \ref{resolution-fibre-product}, $\wt\cT_\ell\to \cT_\ell$ is a principal $H$-bundle and determines the global Kuranishi chart
$$\wt\cK_\ell \coloneq  (G\times H,\wt\cT_\ell,\wt\cT_\ell\times_{\cT_\ell}\cE_\ell,\ide\times\fs_\ell)$$
for $\Mbar_{g,\ell}J(X,A)$. By construction, the square 
\begin{center}\begin{tikzcd}
		\wt\cT_{n} \arrow[r,"\wt\pi_{n}"] \arrow[d,"p_n"]&\wt\cT_{n-1} 
		\arrow[d,"p_{n-1}"]\\ 
		P_n \arrow[r,""] & P_{n-1} \end{tikzcd} \end{center}
is cartesian. As $p_{n-1}$ is a submersion away from a subset of codimension at least $2$, we can apply Lemma \ref{projection-formula-up-codimension} to deduce the equality in \eqref{eq:fund-class-basic}.
\end{proof}

\begin{remark}\label{} In the case where $g = 0$ and $n = 3$, we can apply the same argument to $\Mbar_{0,3}$ and $\Mbar_{0,2}$, the second of which we formally replace by a point, to obtain 
	\begin{equation}\label{} \lspan{\alpha_1,\alpha_2, 1_X}^{X,\omega}_{0,3,A} = 
		\begin{cases}
			0 \quad & A \neq 0\\
			\inpr{\alpha_1\cdot \alpha_2}{[X]}\quad & A = 0
	\end{cases}\end{equation}
using Lemma \ref{mapping-to-point}.
\end{remark}

\subsection{Divisor axiom} The Divisor axiom does not rely on the geometry of the moduli space of stable curves but is due to the fact that generically a stable map representing $A$ intersects a divisor $Y\sub X$ in $\lspan{\pcd(Y),A}$ many points.

\begin{proposition}[Divisor]\label{divisor-axiom} Suppose $(g,n)$ is not basic and $n \geq 1$. For $\alpha_1,\dots,\alpha_{n}\in H^*(X;\bQ)$ with $|\alpha_n| = 2$ and $\sigma \in H_*(\Mbar_{g,n-1};\bQ)$ we have 
	$$\lspan{\alpha_1, \dots ,\alpha_n;\pi_n^!\sigma}^{X,\omega}_{g,n,A} = \inpr{\alpha_n}{A}\,\lspan{\alpha_1, \dots ,\alpha_{n-1};\sigma}^{X,\omega}_{g,n-1,A}.$$
\end{proposition}

The first observation is that for suitable auxiliary data, the evaluation maps become (non-surjective) relative submersions from the thickening to $X$. Let us introduce the following definition first.

\begin{definition}\label{} An unobstructed auxiliary datum $\alpha = (\conn^X,\cO_X(1),p,\cU,\lambda,k)$ \emph{$n$-unobstructed} if for any $(u,\iota,C,x_1,\dots,x_n)\in \fs_{\alpha,n}\inv(0)$ and the perturbed operator \eqref{eq:perturbed-cr} is surjective when restricted to 
	\begin{equation}\label{vanishing}\{\xi\in C^\infty(C,{u}^*T_X)\mid  \xi(x_i) = 0\}\oplus H^0(C,(\iota,u)^*E) \end{equation}
	for each $1 \leq i \leq n$.
\end{definition}

\begin{lemma}\label{strongly-unobstructed-aux-exist} If $\alpha =(\conn^X,\cO_X(1),p,\cU,\lambda,k)$ is an unobstructed auxiliary datum and $n \geq 1$, we can find $\wt k \geq 0$ so that $\wt\alpha \coloneq  (\conn^X,\cO_X(1),p,\cU,\lambda,\wt k)$ is $n$-unobstructed.
\end{lemma}

\begin{proof} Observe first that the zero locus $\fs_{\alpha,n}\inv(0)$ does not depend on the choice of $k$. Now given a point $(u,\iota,C,x_1,\dots,x_n,0,0)$ in the zero locus of the global Kuranishi chart $\cK_{\alpha,n}$, set 
	$$V_i \coloneq  \set{\xi \in C^\infty(C,{u'}^*T_X)\mid  \xi(\kappa(x_i)) = 0}$$
	and let $D_i$ be the restriction of $D\coloneq  D\delbar_J(u')$ to $V_i$. Since $V_i$ is of finite codimension, $D_i$ extends to a Fredholm operator
	$$D_i\cl \cc{V}^{{\ell,2}}\sub W^{\ell,2}(C',{u'}^*T_X) \to W^{\ell-1,2}(\wt C',\Omega^{0,1}_{\wt C'}\otimes_\bC {u'}^*T_X).$$
	 on the Sobolev completion of $V_i$ for $\ell $ sufficiently large, whose formal adjoint is given by $D^*$ post-composed with the orthogonal projection onto $\cc{V_i}^{{\ell-1,2}}$.
	 The proof of the existence of $k$ so that the map
	 $$\Phi_k\cl H^0(C',(\iota',u')^*E_k)\to \Omega^{0,1}(\wt C'.{u'}^*T_X): \eta \mapsto \lspan{\eta}\g d\iota$$
	 surjects onto the cokernel of $D_{u'}$ relies on \cite[Proposition 6.26]{AMS21}. Explicitly, it shows that for any element $\zeta \in \ker(D_{u'})$ we can find $k_\zeta \gg 1$ and $\eta \in H^0(C',(\iota',u')^*E_k)$ so that $$\lspan{\Phi_{k}(\eta),\zeta}_{L^2}\neq 0$$ 
	 for $k \geq k_\zeta$.
	 Taking a (necessarily finite) basis of $\ker(D^*)$, one obtains $k_j$ by taking the maximum of the associated integers $k_\zeta$. We claim that both steps also work with $D_i$ instead of $D_{u'}$. Indeed, the first step only uses the continuity of $\zeta$. Choosing $\ell\geq 4$, this is given for any element of the domain of $D_{u'}^*$. Since $\ker(D_i^*)$ is finite-dimensional as well, we may thus argue as in \cite[Proposition~6.26]{AMS21} to find $k_{u,i}$ so that $\Phi_{k'}$ surjects onto the cokernel of $D_i$ for any $k' \geq k_i{u,i}$. Using the same argument as in \cite[Lemma~4.19]{HS22} together with the fact that the codimension of \eqref{vanishing} in $C^\infty(C,{u}^*T_X)\oplus H^0(C,(\iota,u)^*E)$ is constant and finite, the existence of $k_i\geq k$ with $ k_{i}\geq  k_{u,i}$ for any $u$ now follows from the compactness of $\fs_{\alpha,n}\inv(0)$. Taking $\wt k = \maxi{i}{k_i}$, the auxiliary datum $\wt \alpha \coloneq  (\conn^X,\cO_X(1),p,\cU,\wt k)$ has the desired property.
\end{proof}

\begin{lemma}\label{evaluation-on-thickening} If $\alpha$ is an $n$-unobstructed auxiliary datum, the evaluation map $\eva_j \cl \cT_n \to X$ is a relative submersion for $1 \leq j \leq n$, possibly after shrinking $\cT_n$.
\end{lemma}

\begin{proof} Given $y = (u,\iota,C',x_1,\dots,x_n)\in \fs_{\alpha,n}\inv(0)$ with image $(u',\iota',C')\in \fs_{\alpha,0}\inv(0)$, let $\kappa \cl C\to C'$ be the canonical contraction map. Let $v \in T_{u(x_j)}X = T_{u(\kappa(x_j))}X$ be arbitrary. Extend it to a vector field $\xi \in C^\infty(C',{u'}^*T_X)$ and let $(\xi_1,\eta_1)\in \ker(D_{u'} +\lspan{\cdot}\g d\iota')$ be such that $\xi_1(\kappa(x_j)) = 0$ and $D_{u'}\xi_1 + \lspan{\eta_1}\g d\iota = -D_{u'}\xi$. Let $\wt\xi_1, \wt\xi$ and $\wt\eta$ be the pullback to an element of $C^\infty(C,u^*T_X)$, respectively $H^0(C,(\iota,u)^*E)$. Since $\iota$ is constant on the irreducible components of $C$ contracted by $\kappa$, we have
	$$(\wt\xi_1+\wt\xi,\wt\eta_1)\in \ker(D_u+\lspan{\cdot}\g d\iota) = T_{y}\cT_{\alpha,n}$$
	and $$ d\eva(y)(\xi_1+\xi,\eta_1) = \xi_1(\kappa(x_j)) = v.$$ 
\end{proof}

 Let now $Y\subset X$ be a smooth hypersurface Poincar\'e dual to a class $\gamma\in H^2(X,\bQ)$, possibly first multiplying $\gamma$ by a large integer. Let $\cK_n$ be a global Kuranishi chart for $\Mbar_{g,n}^{\,J}(X,A)$ satisfying the conclusion of Lemma \ref{evaluation-on-thickening}. Let $\cK_Y$ be the global Kuranishi chart with thickening $\cT_Y \coloneq  \eva_n\inv(Y)$ and all other data given by restriction. Set $\Mbar_Y \coloneq  \fs_Y\inv(0)/G$ and let $j \cl \cT_Y/G\hkra \cT_n/G$ be the inclusion.
By Proposition \ref{vfc-embedded-chart},
\begin{equation*}
    j_*\vfc{\Mbar_Y} = \eva_n^*\gamma\cap \vfc{\Mbar_{g,n}^{\,J}(X,A)}.
\end{equation*}

The forgetful map $\pi_n$ restricts to a proper map $\cT_Y\to \cT_{n-1}$ of manifolds of the same dimension.

\begin{lemma}\label{degree-divisor-axiom} ${\pi_n}_*\vfc{\Mbar_Y} = \inpr{\gamma}{A}\,\vfc{\Mbar_{g,n-1}^{\,J}(X,A)}$.
\end{lemma}

\begin{proof} We show that the restriction $\pi\cl\cT_Y\to \cT_{n-1}$ has degree $\inpr{\gamma}{A}$. It suffices to check the claim for a generic point in each connected component of $\cT_{n-1}$. Since $\cT_{n-1}$ is a covering of a space of regular stable maps, we may choose each such point to have smooth domain. Fix $y = [u,\iota,C,x_1,\dots,x_{n-1},\alpha,\eta]$ in $\cT_{n-1}$ with $u \pf Y$. By the definition of $\cT_n$ and $\cT_{n-1}$ we can find a neighbourhood $U\sub \cT_{n-1}$ of $y$ so that any curve in $U$ intersects $Y$ transversely and
	$$U\cong V\times B$$ 
	where $V\sub \cM_{n-1}$ and $B \sub T_{\cT_{n-1}/\cM_{n-1},y}$ are open neighbourhoods of the image of $y$, respectively, the origin. Recall that $\pi \cl \cT_n \to \cT_{n-1}$ fits into the cartesian square
	\begin{center}\begin{tikzcd}
			\cT_{n} \arrow[r,"\pi_n"] \arrow[d,""]&\cT_{n-1} \arrow[d,""]\\ \cM_{n} \arrow[r,"\pi_n"] & \cM_{n-1} \end{tikzcd} \end{center}
	where the vertical maps are topological submersions and the horizontal maps are submersions over a neighbourhood of $y$, respectively, of the image of $y$ in $\cM_{n-1}$. As the domain $C$ of $y$ is smooth, the relative smooth structure on $\cT_{n-1}$, and thus $\cT_{n}$, extends to a smooth structure. However, we will only use that $\cT_{n-1}$ has a relative smooth structure over $\cM_{n-1}$.\par
	 After shrinking $V$, we have
	$$\pi_n\inv(U) = \cT_n \times_{\cM_n} \pi_n\inv(V) \cong V_{n-1}\times C\times B$$
	and the (relative) differential of $\eva_n$ at $x \in \pi_n\inv(\{y\})\cong C$ is given by 
	$$T_{x} C \times B \to T_{u(x)}X : (v,\xi) \mapsto du(x)v + \xi(x).$$
%	Therefore,
%	$$T_{\cT_Y/\cM_{n-1},x} = \{(v,\xi)\in T_{x}C\times B \mid du(x)v + \xi(x)\in T_{u(x)}Y\}.$$
%	for $x \in u\inv(Y)$. 
	The orientation of $T_{\cT_Y/\cM_n,x}$ is induced by it being the kernel of $T_{\cT_n/\cM_n,x}\to T_{u(x)}X/T_{u(x)}Y$. As $du(x)$ induces a complex linear isomorphism to $T_{u(x)}X/T_{u(x)}Y$, we are reduced to showing that, given an exact sequence 
	$$0 \to W\to V\oplus \bC\xra{(v,z)\mapsto f(v) + az} \bC\to 0$$
	for a real-linear map $f$ and $a \in \bC\units$, the projection $W\to V$ is an orientation-preserving isomorphism. This follows from the fact that $z \mapsto -az$ is complex linear and thus orientation preserving. Hence
	$$\deg(\pi;y) =\s{x\in \pi\inv(\{y\})}{\normalfont\text{sign}(d\pi(x))} = \s{x\in \pi \inv(\{y\})}{\normalfont\text{ind}(u,Y,x)} = \lspan{\gamma,A}.$$ 
\end{proof}

The Divisor axiom is an immediate consequence.

\subsection{Splitting axiom}

Fix $g = g_0+ g_1$ and $n = n_0+n_1$ with $n_i \geq 0$ and $2g_i-2 + n_i+1 > 0$. Let $S\subset \{1,\dots,n\}$ be a subset with $|S| = n_0$. The clutching map 
$$\varphi_S \cl \Mbar_{g_0,n_0+1}\times\Mbar_{g_0,n_0+1}\to \Mbar_{g,n}$$ 
given by gluing two curves together at the $(n_0+1)^{\normalfont\text{th}}$ and first marked point and renumbering according to the partition induced by $S$ is a closed local immersion. This map lifts to maps
\begin{equation*}\label{splitting-spaces}\varphi_{S,X}\cl \Mbar_{g_0,n_0+1}^{\,J}(X,A_0)\times_{X}\Mbar_{g_1,n_1+1}^{\,J}(X,A_1)\to \Mbar_{g,n}^{\,J}(X,A_0+A_1).\end{equation*}
Together with the clutching maps described in the next subsection, the images of the maps $\varphi_S$ form the boundary divisor of $\Mbar_{g,n}$. The same is true for moduli spaces of stable maps, where one has an additional choice of how to ``split" the homology class. The Splitting axiom is an algebraic reflection thereof. 

\begin{proposition}[Splitting]\label{splitting-axiom} Write $\pcd(X) = \s{i\in I}{\gamma_i\times \gamma_i'}$ for $\gamma_i,\gamma_i'\in H^*(X;\bQ)$. We have
	\begin{multline*}\lspan{\alpha_1,\dots,\alpha_{n};{\varphi_S}_*(\sigma_0\otimes\sigma_1)}^{X,\omega}_{g,n,A} \\
		\quad=  (-1)^{\epsilon_{\alpha,S}}\s{A_0+A_1 = A}{\s{i}{}} \lspan{\alpha_{f_0(1)},\dots,\alpha_{f_0(n_0)}, \gamma_i;\sigma_0}^{X,\omega}_{g_0,n_0+1,A_0} \lspan{\gamma'_i,\alpha_{f_1(2)},\dots,\alpha_{f_1(n_1+1)};\sigma_1}^{X,\omega}_{g_1,n_1+1,A_1}\end{multline*}
	 for any $\alpha_1,\dots,\alpha_{n}\in H^*(X;\bQ)$ and $\sigma_i \in H_*(\Mbar_{g_i,n_i+1};\bQ)$, where $$\epsilon_{\alpha,S} \coloneq  |\{i < j \mid j \in S,\;i \notin S,\;|\alpha_i|,|\alpha_j|\in 2\bZ+1\}|.$$
\end{proposition}

By Lemma \ref{symmetry-axiom}, we may assume $S = \{1,\dots,n_0\}$ and omit it from the notation. The domain of $\varphi_X$ admits a global Kuranishi chart $\cK_{g_0,n_0+1,A_0}\times_X \cK_{g_1,n_1+1,A_1}$ of the expected virtual dimension by Lemma \ref{evaluation-on-thickening}, which embeds into $\cK_{g_0,n_0+1,A_0}\times \cK_{g_1,n_1+1,A_1}$. For the sake of brevity, we denote the base space of $\cK_{g,n,A}$ by $\cM$.
For $0 \leq m_0 \leq m$, let $\widehat{\cM}_{g_0,n_0,m_0}$ be the preimage of $\cM$ under $$\varphi_{\bP^N}\cl \Mbar_{g_0,n_0+1}(\bP^N,m_0)\times_{\bP^N} \Mbar_{g_1,n_1+1}(\bP^N,m_1) \to \Mbar_{g,n}(\bP^N,m)$$ and set 
$$\widehat{\cM}_{g_0,n_0} \coloneq  \djun{0\leq m_0\leq m}\widehat{\cM}_{g_0,n_0,m_0}.$$

\begin{lemma}\label{smooth-base-splitting} $\widehat{\cM}_{g_0,n_0,m_0}$ is a complex manifold of the expected dimension.
\end{lemma}

\begin{proof} Suppose $\varphi_{\bP^N}([\iota,C,x_*],[\iota',C',x'_*]) = [u,\Sigma,y_*]\in \cM$. As the normalisation of $\Sigma$ is $\tilde{\Sigma} = \tilde{C}\sqcup \tilde{C}'$ and $u$ is unobstructed, so are $\iota$ and $\iota'$. If $\rho$ is an automorphism of $(\iota,C,x_*)$, it can be extended by the identity to an automorphism of $(u,\Sigma,y_*)$. Hence $\rho = \ide_C$. Similarly, we see that $(\iota',C',x'_*)$ has no isotropy. Thus, we may conclude by \cite{RRS08}.
\end{proof}

Denote by $\hat{\varphi} \cl \widehat{\cM}_{g_0,n_0} \to \cM$ the map induced by $\varphi_{\bP^N}$. It is a $\PGL_\bC(N+1)$-equivariant immersion whose image has normal crossing singularities. Hence, $\hat{\varphi}^*\cK_{g,n,A}$ is a global Kuranishi chart whose thickening $\widehat{\cM}_{g_0,n_0} \times_\cM\cT_{g,n,A}$ admits a canonical relative smooth structure over $\widehat{\cM}_{g_0,n_0}$.\par
Let $[u,C,x_*]\in \Mbar_{g,n}^{\,J}(X,A)$ lie in the image of $\varphi_X$. Any splitting of the domain into two curves of the prescribed genus and prescribed set of points defines (via the restriction of a framing coming from $\cL_u$) a unique element in the image of $\hat{\varphi}$. Conversely, any decomposition of the domain of the framing leads to a corresponding splitting of $[u,C,x_*]$, 
 where the degree of the restrictions of $u$ may vary. This shows that $\cK_{g_0,n_0} \coloneq  \hat{\varphi}^*\cK_{g,n,A}$ is a global Kuranishi chart for
$$\Mbar_{g_0,n_0}(X) \coloneq  \djun{A_0+A_1 = A}\Mbar_{g_0,n_0+1}^{\,J}(X,A_0)\times_X\Mbar_{g_1,n_1+1}^{\,J}(X,A_1).$$
as is $\cK'_{g,n,A} \coloneq  \djun{A_0+A_1 = A}\cK_{g_0,n_0+1,A_0}\times_X \cK_{g_1,n_1+1,A_1}$. 

\begin{lemma}\label{charts-equivalent-splitting} $\cK_{g_0,n_0}$ and $\cK'_{g,n,A}$ are equivalent.
\end{lemma}

\begin{proof} This follows from a double-sum construction as in \cite[\textsection 6.1]{HS22}.
\end{proof}

We can now prove Proposition \ref{splitting-axiom}. The strategy of proof is the same as in Proposition \ref{fundamental-class-axiom} with the additional complication that $\tilde{\varphi}$ is not the pullback of $\varphi$.

%\begin{lemma}\label{lift-to-projection-formula-splitting} We have
%	\begin{equation}\label{split-stab-projection-formula} \hat{\varphi}_!\, \normalfont\text{st}^* = \normalfont\text{st}^*\varphi_!\end{equation}
%	as maps $H^*(\Mbar_{g_0,n_0+1}\times\Mbar_{g_1,n_1+1};\bQ)\to H^{1+*}_{G}(\cM_n;\bQ)$.
%\end{lemma}

\begin{proof} We first sketch the proof in the case of genus zero. In higher genus, the fact that we have to take the fibre products of orbifolds adds a layer of complexity; however, the strategy is the same. Factor $\varphi \cl \widehat{\cM}_{g_0,n_0}\to \cM$ as 
$$\widehat{\cM}_{0,n_0}\xra{\vartheta}\cM\times_{\Mbar_{0,n}}(\Mbar_{0,n_0+1}\times \Mbar_{0,n_1+1})\xra{\varphi'}\cM.$$
Then $\cM\times_{\Mbar_{0,n}}(\Mbar_{0,n_0+1}\times \Mbar_{0,n_1+1})$ is a quasi-projective variety over $\bC$ but not necessarily smooth. However, it is a homology $\bQ$-manifold. As $\varphi'$ is the pullback of the clutching map on the level of moduli space of stable curves, it is a closed immersion of schemes. We will show below that $\vartheta$ has degree $1$. Lifting this factorisation to the level of thickenings, we obtain that $\hat{\varphi}_!\;\text{st}^* = \text{st}^*\;\varphi_!$ as maps $H^*(\Mbar_{0,n_0+1}\times\Mbar_{0,n_1+1};\bQ)\to H^*(\cT/G;\bQ)$ by the results of \textsection\ref{sec:orbifold-intersection}.\par

For the general case, let $P/H$ be a presentation of $\Mbar_{g,n}$ as a global quotient. By \cite[Proposition~11.5.12]{ACG-moduli}, we may assume that $P \sub \Mbar_{g,n}(\bP^\ell,\ell+g)$ for some $\ell\geq 1$. As $\varphi$ is representable by \cite[Proposition 12.10.11]{ACG-moduli}, $P' \coloneq  P\times_{\Mbar_{g,n}} (\Mbar_{g_0,n_0+1}\times \Mbar_{g_1,n_1+1})$ is a $G$-smooth manifold with $P'/H$ representing $\Mbar_{g_0,n_0+1}\times \Mbar_{g_1,n_1+1}$. Moreover, $\cN\coloneq  P\times_{\Mbar_{g,n}}\cM$ is a principal $G'$-bundle over $\cM$, while
	$$\cN_{g_0,n_0} \coloneq  \cN\times_{\cM} \widehat{\cM}_{g_0,n_0} = P'\times_{\Mbar_{g_0,n_0+1}\times\Mbar_{g_1,n_1+1}}\widehat{\cM}_{g_0,n_0}$$
	is one over $\widehat{\cM}_{g_0,n_0}$. Denote by $\normalfont\text{st}\cl \cN\to \Mbar_{g,n}$ and $\normalfont\text{st}\cl \cN_{g_0,n_0}\to \Mbar_{g_0,n_0+1}\times\Mbar_{g_1,n_1+1}$ the induced stabilisation maps. Let $\tilde{\varphi}\cl \cN_{g_0,n_0}\to \cN$ be the morphism of principal bundles covering $\hat{\varphi}$. The diagram
	\begin{center}\begin{tikzcd}
			\cN_{g_0,n_0} \arrow[r,"\vartheta"] \arrow[dr,"\normalfont\text{st}"]&\cN'\arrow[r,"\varphi'"]  \arrow[d,"\normalfont\text{st}"]&\cN\arrow[d,"\normalfont\text{st}"]\\ & \Mbar_{g,n_0+1}\times \Mbar_{g_1,n_1+1}\arrow[r,"\varphi"] & \Mbar_{g,n} \end{tikzcd} \end{center}
	commutes, where $\tilde{\varphi} = \varphi'\vartheta$ and $\cN' = P'\times_P\cN$, so the square is cartesian. To see that $\vartheta$ is a birational equivalence, let $\cN^{\,'\text{sm}}\subset \cN'$ be the preimage (under $\cN' \to \cN\to \cM)$ of the subset of $\cM$ consisting of curves with one node. As all elements of $\cM$ are unobstructed, the complement of $\cN^{\,'\text{sm}}$ is of codimension at least $2$. It follows from a straightforward consideration of the fibre product $\cN'$ that the induced map $\cN^{\text{sm}}_{g_0,n_0}\to \cN^{\,'\text{sm}}$ is an isomorphism. 
    %with preimage given by 
    %s$$([\iota,C,y_*],\psi,[v,\Sigma,y_*],\rho,(\Sigma_i,y^i_*))\to ([\iota,C,y_*],\psi,[v,\Sigma,y_*],[\iota\psi\inv\rho_i|_{\Sigma_i},\Sigma_i,y^i_1,\dots,y^i_{n_i+1}])_{i =0,1}).$$ 
    %Here $\psi \cl C\to \Sigma$ and $\rho\cl (\Sigma_0\vee \Sigma_1)^{\text{st}} \to C^{\text{st}}$ are isomorphisms of marked nodal surfaces.
    Thus, $\vartheta$ is a birational equivalence. In particular, it has degree $1$.\par
    Pull back $\widehat{\cK}_{g_0,n_0}$ along $\cN_{g_0,n_0}\to \widehat{\cM}_{g_0,n_0}$ and change the covering group from $G$ to $G\times G'$ to obtain an equivalent global Kuranishi chart $\widehat{\cK}^\cN_{g_0,n_0}$. Define $\cK^\cN_{g,n}$ analogously to obtain a global Kuranishi chart, which is rel--$C^\infty$ over $\cN$ and equivalent to $\cK_{g,n}$. The pullback of $\hat{\varphi}$ to a map between thickenings descends to $\varphi_X$ when restricted to the zero locus of the obstruction section. It factors as $\wt{\varphi} =\wt\varphi' \wt{\vartheta}$, which are lifts of $\varphi'$ respectively $\vartheta$. Then $\wt{\vartheta}$ has degree $1$, so $\wt\varphi$ satisfies
    $$\wt\varphi_!\;\text{st}^* = \text{st}^*\;\varphi_!$$
    as maps $H^*(\Mbar_{g_0,n_0+1}\times\Mbar_{g_1,n_1+1};\bQ)\to H^*(\cT^{\cN}/G\times H;\bQ)$ by Lemma \ref{push-pull-principal-bundle} and Lemma \ref{projection-formula-up-codimension}. As the obstruction bundle of $\widehat\cK_{g_0,n_0}^\cN$ is the pullback of the obstruction bundle of $\cK_{g,n}^\cN$, this shows that 
    \begin{multline}\label{splitting-vfc-relation}\s{\substack{A_0+A_1 = A}}{{\varphi_X}_*(\normalfont\text{st}^*\gamma\cap \vfc{\Mbar_{g_0,n_0+1}^{\,J}(X,A_0)\times_X \Mbar_{g_1,n_1+1}^{\,J}(X,A_1)})} \\= \normalfont\text{st}^*(\varphi_!\gamma)\cap \vfc{\Mbar_{g,n}^{\,J}(X,A)}\end{multline}
    for any $\gamma\in H^*(\Mbar_{g_0,n_0+1}\times \Mbar_{g_1,n_1+1};\bQ)$. Applying Proposition \ref{vfc-embedded-chart} to $\cK'_{g,n,A}$, we obtain an expression for the left hand side that exactly gives the Splitting axiom.
\end{proof}

\subsection{Genus reduction axiom} Fix $g$ and $n\geq 2$ and let $\eva_{n-1,n}$ be the evaluation at the marked points labelled by $n-1$ and $n$. Given an almost complex manifold $(Y,J_Y)$ and $B\in H_2(Y,\bZ)$, define $$\Mbar_{g,n}(Y,B;J_Y)_{n-1,n}\coloneq  \eva_{n-1,n}\inv(\Delta_Y)\subset \Mbar_{g,n}(Y,B;J_Y).$$ 
Let
$$\psi_Y\cl\Mbar_{g,n}(Y,B;J_Y)_{n-1,n}\to \Mbar_{g+1,n-2}(Y,B;J_Y)$$
be given by gluing the $(n-1)^{\text{nth}}$ and the $n^{\text{nth}}$ marked point of the domain together. It covers the corresponding smooth map $\psi \cl \Mbar_{g,n}\to \Mbar_{g+1,n-2}$. 

\begin{proposition}[Genus reduction]\label{genus-reduction} We have 
	\begin{equation*}
	    \lspan{\alpha_1,\dots,\alpha_{n-2},\pcd(\Delta_X);\sigma}^{X,\omega}_{g,n,A} = \lspan{\alpha_1,\dots,\alpha_{n-2};\psi_*\sigma}^{X,\omega}_{g+1,n-2,A}
	\end{equation*}
for any $\alpha_1,\dots,\alpha_{n-2}\in H^*(X;\bQ)$ and $\sigma \in H_*(\Mbar_{g,n};\bQ)$.
\end{proposition}

\medskip

\begin{lemma}\label{genus-reduction-base} The preimage 
$$\widetilde{\cM} \coloneq  \psi_{\bP^N}\inv(\Mbar_{g+1,n-2}^*(\bP^N,m))\subset \Mbar_{g,n}(\bP^N,m)$$
is a smooth manifold of the expected dimension. The induced map $\psi_{\bP^N}\cl \widetilde{\cM}\to \Mbar_{g+1,n-2}^*(\bP^N,m)$ is smooth and $G$-equivariant. It factors as the composition of a double cover with a map which is generically an embedding.
\end{lemma}

\begin{proof} Suppose $\psi_{\bP^N}([\iota,C,x_*]) = [\iota',C',x_*']$ and let $\rho \in \Aut(\iota,C,x_*)$. Then $\rho$ descends to an automorphism of $(\iota',C',x_*')$, which has to be the identity, as the later stable map has no nontrivial automorphisms by definition. As the gluing map $\kappa\cl C\to C'$ is injective on a dense subset, $\rho = \ide_C$. We have a short exact sequence 
	$$0 \ra \kappa^*{\iota'}^*\cO_{\bP^N}(1)\to \iota^*\cO_{\bP^N}(1)\to \bC_x \to 0,$$ 
	where $\bC_x$ denotes the skyscraper sheaf over $x$ and the last map is given by $s\mapsto s(x_{n-1})-s(x_n)$ (using a trivialisation of $\cO_{\bP^N}(1)$ near $\iota(x_n)$). Part of its long exact sequence is 
	$$0 = H^1(C',{\iota'}^*\cO_{\bP^N}(1))\to H^1(C,\iota^*\cO_{\bP^N}(1))\to H^1(\{x\},\bC) = 0.$$
	Hence $\iota$ is unobstructed.\par
	$\bZ/2$ acts smoothly on $\widetilde{\cM}$ by permuting the last two points and $\psi_{\bP^N}$ factors through $\widetilde{\cM}/(\bZ/2)$. As $\psi_{\bP^N}$ is an immersion and $\widetilde{\cM}/(\bZ/2)\to \Mbar_{g+1,n-2}^*(\bP^N,m)$ is injective over the locus of curves with smooth domain, the claim follows. 
\end{proof}

Let $\cK_{g,n,A}$ be a global Kuranishi chart for $\Mbar_{g,n}^{\,J}(X,A)$ satisfying the conclusion of Lemma \ref{evaluation-on-thickening} and let $\cK_{g+1,n-2,A}$ be a global Kuranishi chart for $\Mbar_{g,n}^{\,J}(X,A)$. Note that these charts are not related in any way.\par 
By Lemma~\ref{genus-reduction-base}, the pullback global Kuranishi chart $\psi_{\bP^N}^*\cK_{g,n,A}$ is a well-defined global Kuranishi chart whose thickening is relatively smooth over $\wt\cM$. Meanwhile, let $\cK'$ be the global Kuranishi chart with thickening $\cT' = \eva_{n-1,n}\inv(\Delta_X)$ and whose other data are given by restriction from $\cK_{g,n,A}$. As before, a double-sum construction as in \cite[\textsection 5.1]{HS22} can be used to show that $\cK'$ and $\psi_{\bP^N}^*\cK_{g+1,n-2}$ are equivalent global Kuranishi charts for $\Mbar_{g,n}^{\,J}(X,A)_{n-1,n}$.\par 
Pulling back $\psi_{\bP^N}^*\cK_{g+1,n-2}$ along $\cN\to \Mbar_{g+1,n-2}^*(\bP^N,m)$ and taking the product of the covering group with $G'$, we obtain an equivalent global Kuranishi chart $\tilde{\cK}_{g+1,n-2}$ for $\Mbar_{g,n}^{\,J}(X,A)_{n-1,n}$. The proof of Proposition \ref{splitting-axiom} now carries over in a straightforward manner.

\section{Gravitational descendants}\label{subsec:grav-desc}

The cohomology ring of $\Mbar_{g.n}$ for higher genus is still relatively unknown. However, it admits a subring, the \emph{tautological ring}, which is relatively computable, \cite{Pix13}. The tautological ring is generated by the Chern classes of the tautological line bundles, as well as pushforwards of these classes under the forgetful and clutching maps, \cite[\textsection0.3]{FP00}. This section defines $\psi$-classes for moduli spaces of stable maps and proves a generalisation of the Kontsevich-Manin axioms.

\begin{definition} Given an $n$-pointed family $(\pi \cl \cC\to \cV,\sigma_1,\dots,\sigma_n)$ of stable curves and $i \leq n$, we define the \emph{$i^{\text{th}}$ tautological line bundle} of $\cV$ to be $\bL^\cV_i \coloneq  \sigma_i^*(\ker(d\pi)^*)$. We define the \emph{$i^{\text{th}}$ $\psi$-class} to be $$\psi^\cV_i \coloneq  c_1(\bL^\cV_i)\in H^2(\cV;\bZ).$$
The \emph{Hodge bundle} of $\cC\to\cV$ is the complex rank-$g$ vector bundle $\bE^\cV \coloneq  \pi_*\omega_{\cC/\cV}$. The \emph{$\lambda$-classes} are $$\lambda^\cV_j \coloneq  c_j(\bE^\cV).$$ 
\end{definition}

These vector bundles patch together to form orbibundles, the \emph{$i^{\text{th}}$ tautological line bundle}, respectively the \emph{Hodge bundle}, on the Deligne-Mumford stack $\Mbar_{g,n}$. In this case, or if the family $\cV$ is clear from the context, we omit the superscripts. See \cite[Chapter~25]{HKK03} and \cite{FaP00} for context and relations satisfied by the integrals of these classes.

Fix a closed symplectic manifold $(X,\omega)$ and let $J \in \cJ_\tau(X,\omega)$ and $A\in H_2(X;\bZ)$.\par Let $\cK_n = (G,\cT_n/\cM_n,\cE_n,\fs_n)$ be a global Kuranishi chart for $\Mbar_{g,n}^{\,J}(X,A)$ as given by Construction \ref{high-level-description-defined}. Recall that $\cM_n$ is a $G$-invariant open subset of the automorphism-free locus of regular maps in $\Mbar_{g,n}(\bP^N,m)$, where $m = N+g$ and $N \gg 1$. It admits a quasi-projective smooth universal family $\cU_n = \cM_{n+1}$ on which $G = \PU(N+1)$ acts smoothly. Let $\Pi_n \cl \cT_n \to \cM_n$ be the structure map and define
\begin{equation*}\bL_i \coloneq  \bL_{n,i} \coloneq  \Pi_n^*(\bL^{\cM_n}_i) \qquad \qquad \bE \coloneq  \Pi_n^*(\bE^{\cM_n}).\end{equation*}
They are, by definition, relatively smooth vector bundles over $\cT_n$. The $G$-action on  $\cT_n$ lifts to a fibrewise linear $G$-action on $\bL_i$ and $\bE$. Define
\begin{equation*}\psi_i \coloneq  \psi_{n,i} \coloneq  c_1(\bL_i/G)\qquad \qquad \lambda_j \coloneq  c_j(\bE/G)\end{equation*}
in $H^*(\cT/G;\bQ)$. They restrict to classes on $\Mbar_{g,n}^{\,J}(X,A)$, also denoted $\psi_i$ and $\lambda$.

\begin{lemma}\label{lem:taut-classes-indep} The $\psi$- and $\lambda$-classes of $\Mbar_{g,n}^{\,J}(X,A)$ are independent of the choice of unobstructed auxiliary datum.
\end{lemma}

\begin{proof} Given two unobstructed auxiliary data, we obtain global Kuranishi charts $\cK_{n,0}$ and $\cK_1$ with base space $\cM_0 \sub \Mbar_{g,n}(\bP^{N_0},m_0)$ and $\cM_1 \sub \Mbar_{g,n}(\bP^{N_1},m_1)$ respectively. By \cite[\textsection6.1]{HS22} there exists a double-sum global Kuranishi chart of $\Mbar_{g,n}^{\,J}(X,A)$ with covering group $G_0\times G_1$ and base $\cM_{01}\sub \Mbar_{g,n}(\bP^{N_0}\times\bP^{N_1},(m_1,m_1))$ contained in $\Phi\inv(\cN_0\times\cN_1)$ such that the projections $\cN_0$ and $\cN_1$ do not collapse irreducible components of the domains. The induced map $\Phi_j \cl \cN_{01}\to \cN_j$ is a principal $G_{j'}$-bundle (for $\{j,j'\} = \{0,1\}$), so
$$\bL_i^{\cN_{01}} = \Phi_j^*\bL_i^{\cN_j}.$$
    In particular, $\bL_i^{\cN_{01}}$ is a principal $G_{j'}$-bundle over $\bL_i^{\cN_j}$; hence  $c_1(\bL_i^{\cN_{01}})_{G_0\times G_1} = \Phi_j^*c_1(\bL_i^{\cN_j})_{G_j}$. Pulling these relations back to the thickening of the double-sum Kuranishi chart, we obtain the claim.
\end{proof}

In particular, if $\Mbar_{g,n}^{\,J}(X,A)$ is unobstructed, we recover the standard definition of the $\psi$- and $\lambda$-classes.

\begin{proof} Given two unobstructed auxiliary data, we obtain global Kuranishi charts $\cK_{n,0}$ and $\cK_1$ with base space $\cM_0 \sub \Mbar_{g,n}(\bP^{N_0},m_0)$ and $\cM_1 \sub \Mbar_{g,n}(\bP^{N_1},m_1)$ respectively. By \cite[\textsection6.1]{HS22} there exists a double-sum global Kuranishi chart of $\Mbar_{g,n}^{\,J}(X,A)$ with covering group $G_0\times G_1$ and base $\cM_{01}\sub \Mbar_{g,n}(\bP^{N_0}\times\bP^{N_1},(m_1,m_1))$ contained in $\Phi\inv(\cN_0\times\cN_1)$ such that the projections $\cN_0$ and $\cN_1$ do not collapse irreducible components of the domains. The induced map $\Phi_j \cl \cN_{01}\to \cN_j$ is a principal $G_{j'}$-bundle (for $\{j,j'\} = \{0,1\}$), so
	$$\bL_i^{\cN_{01}} = \Phi_j^*\bL_i^{\cN_j}.$$
	In particular, $\bL_i^{\cN_{01}}$ is a principal $G_{j'}$-bundle over $\bL_i^{\cN_j}$; hence  $c_1(\bL_i^{\cN_{01}})_{G_0\times G_1} = \Phi_j^*c_1(\bL_i^{\cN_j})_{G_j}$. Pulling these relations back to the thickening of the double-sum Kuranishi chart, we obtain the claim.
\end{proof}

Fix now a global Kuranishi chart $\cK_n$ as given by Construction \ref{high-level-description-defined} with base space $\cM_n$ and structure map $\Pi_n \cl \cT_n\to\cM_n$. Recall that $\cM_n$ is the preimage of $\Mbar_g^*(\bP^N,m)$ under the forgetful map $\Mbar_{g,n}(\bP^N,m)\to \Mbar_{g,n}(\bP^N,m)$ for any $n$; in particular, the forgetful maps restrict to proper smooth maps $\cM_{n+1} \to \cM_n$ with canonical sections $\sigma_{i} \coloneq  \sigma^{(n)}_i \cl \cM_n \to \cM_{n+1}$. We abbreviate $D_{n+1,i} \coloneq  \im(\sigma^{(n)}_i)$ and set
$$D_{n+1,i}(A) \coloneq  \Pi_n\inv(D_{n+1,i})\qquad \qquad \cc{D}_{n+1,i}(A) \coloneq  D_{n+1,i}(A)/G.$$
$\cc{D}_{n+1,i}(A)$ is an oriented suborbifold of $\cT_{n+1}/G$ with Poincar\'e dual $D_{n+1,i,A} = \Pi_n^*\pcd(\cc{D}_{n+1,i})$.

\begin{definition} The \emph{descendent Gromov--Witten invariant} (or \emph{gravitational descendant}) of $(X,\omega)$ associated to $(A,g,n)$ is 
$$\lspan{\tau_{k_1}\alpha_1,\dots,\tau_{k_n}\alpha_n;\sigma}^{X,\omega}_{A,g,n} \coloneq  \lspan{\psi^{k_1}_1\eva_1^*\alpha_1\cdots\psi^{k_n}_n\eva_n^*\alpha_n\cdot \normalfont\text{st}^*\pcd(\sigma),\vfc{\Mbar_{g,n}^{\,J}(X,A)}}$$
    for $k_1,\dots,k_n \geq 0$, $\alpha_1,\dots,\alpha_n \in H^*(X;\bQ)$ and $\sigma\in H_*(\Mbar_{g,n};\bQ)$. The \emph{Hodge integrals} of $X$ are the numbers
$$\lspan{\lambda_1^{b_1}\cdots\lambda_g^{b_g}\cdot\psi^{k_1}_1\eva_1^*\alpha_1\cdots\psi^{k_n}_n\eva_n^*\alpha_n\cdot \normalfont\text{st}^*\pcd(\sigma),\vfc{\Mbar_{g,n}^{\,J}(X,A)}},$$
where $b_1,\dots,b_g \geq 0$ are integers.
\end{definition}

\begin{remark}\label{} Given a different $J' \in \cJ_\tau(X,\omega)$, a cobordism between two global Kuranishi charts for $\Mbar_{g,n}^{\,J}(X,A)$ and $\Mbar_{g,n}^{\,J'}(X,A)$ is constructed in \cite[\textsection6.2]{HS22}. These two global Kuranishi charts have the same base space $\cM$ and the coboridism between them is a rel--$C^\infty$ manifold with boundary over $\cM$. Thus, the pullback of $\bL_i^\cM$ extends to an equivariant line bundle over the cobordism, showing that the gravitational descendants of $X$ do not depend on the choice of $\omega$-tame almost complex structure.\end{remark}

\begin{remark} There are several slightly different definitions of descendent GW invariants in the literature. The $\psi$-classes in \cite{RT97} are defined to be the pullbacks $\text{st}^*\psi_i^{\Mbar_{g,n}}$ of the $\psi$-classes on the moduli space of these curves. These (respectively a slightly modified definition thereof) are called \emph{gravitational ancestors} in \cite{Giv01} and differ from our definition of $\psi$-classes. In \cite{KM98}, whose definition agrees with ours, the exact relationship between the two definitions is elucidated. In \cite{KKP03}, the enumerative \emph{modified $\psi$-classes} $\bar{\psi_i}$ are defined, which satisfy $\pi_{n+1}^*\bar{\psi}_{n,i} = \bar{\psi}_{n+1,i}$, where $\psi_{k,i}$ is the $\psi$-class on $\Mbar_{g,k}^{\,J}(X,A)$.
\end{remark}

\begin{lemma}\label{lem:comparison-psi} Let $\pi_{n+1}\cl \Mbar_{g,n+1}^{\,J}(X,A)\to \Mbar_{g,n}^{\,J}(X,A)$ be the forgetful map. Then 
\begin{equation}\label{eq:comp-psi}\psi_{n+1,i} = \pi_{n+1}^*\psi_{n,i} + D_{n+1,i,A}\end{equation}
    in $H^*_G(\cM_{n+1},\bQ)$. Moreover, $\psi_{n+1,i}\cdot D_{n+1,i,A} = 0$. 
\end{lemma}

\begin{proof} This is classical. Let $U \coloneq  \cM_{n+1}\sm D_{n+1,i}$. Given $[\iota,C,x_1,\dots,x_{n+1}] \in U$, the forgetful map $\pi_{n+1}$ does not change the irreducible component of $C$ containing $x_i$. Thus, $L_i \coloneq  \bL_i\otimes \pi_{n+1}^*\bL_i\inv$ is trivial over $U$. Set $\rho \coloneq  \sigma_{i}^{(n+1)}\sigma_i^{(n)}$ and $V \coloneq  \im(\rho)$. Then, 
	$$\lbr{\sigma_i^{(n)}}^*(\bL_i\otimes\pi_{n+1}^*\bL_i\inv) = \rho^*\omega_{\pi_{n+2}}\otimes \lbr{\sigma_i^{(n+1)}}^*\omega_{\pi_{n+1}}.$$
As the normal bundle of $D_{i,n+1}$ is $(\sigma_i^{(n+1)})^*\omega_{\pi_{n+1}}$, it remains to show that $\rho^*\omega_{\pi_{n+2}}$ is equivariantly trivial. Let $\tilde{\pi}\cl \cM_{n+1} \to \cM_4$ be the map, which forgets all marked points except the first, $i^{\text{th}}$ and last two ones. Then $\tilde{\pi}$ is $G$-equivariant and maps $D_{i,n+2}$ to 
$$\cM_{n,2}\times_{\bP^N}(\bP^N\times\Mbar_{0,4})\cong \cM_{2}\times \Mbar_{0,4}$$
and $V$ to $\cM_{2}\times \{*\}$ under this identification.
  Hence, $\rho^*\omega_{\pi_{n+2}}\inv = \tilde{\pi}^*T_{\Mbar_{0,4},*}|_V$ is trivial. If $\iota \cl \cc{D}_{n+1,i} \hkra \cM_{n+1}/G$ is the inclusion, then $$\psi_{n+1,i,G}\cdot \pcd(\cc{D}_{n+1,i}) = \iota_!c_1(\rho^*\omega_{\pi_{n+2}})_G=0.$$
  Thus, the last claim follows from Lemma \ref{projection-formula}.
\end{proof}

There are three further relations between the Gromov--Witten invariants that take $\psi$-classes into account. The first one generalises the Fundamental Class axiom, while the third corresponds to the Divisor axiom in \textsection\ref{sec:axioms}. Denote by $1_X$ the unit of $H^*(X;\bQ)$ and formally treat $\Mbar_{g,n}$ as a point if $2g-2+ n\leq 0$.

\begin{proposition}[String equation] We have
\begin{multline}\label{eq:string-equation-variant}
	\lspan{\tau_{k_1}\alpha_1,\dots,\tau_{k_n}\alpha_n,1_X;\sigma}^{X,\omega}_{A,g,n+1} = \lspan{\tau_{k_1}\alpha_1,\dots,\tau_{k_n}\alpha_n;{\pi_{n+1}}_*\sigma}^{X,\omega}_{A,g,n}\\+\sum_{i=1}^{n}\lspan{\tau_{k_1}\alpha_1,\dots,\tau_{k_i-1}\alpha_i,\dots,\tau_{k_n}\alpha_n;{\pi_{n+1}}_*(\pcd(D_{n+1,i})\cap \sigma)}^{X,\omega}_{A,g,n}
\end{multline} 
for any $\alpha_1,\dots,\alpha_n\in H^*(X;\bQ)$ and $\sigma \in H_*(\Mbar_{g,n+1};\bQ)$. In particular, if $\sigma = [\Mbar_{g,n+1}]$, then
\begin{equation*}\label{eq:string-equation}
	\lspan{\tau_{k_1}\alpha_1,\dots,\tau_{k_n}\alpha_n,1_X;[\Mbar_{g,n+1}]}^{X,\omega}_{A,g,n+1} = \sum_{i=1}^{n}\lspan{\tau_{k_1}\alpha_1,\dots,\tau_{k_i-1}\alpha_i,\dots,\tau_{k_n}\alpha_n;[\Mbar_{g,n}]}^{X,\omega}_{A,g,n}.
\end{equation*} 
\end{proposition}

\begin{proof} Write $\alpha \coloneq  \alpha_1\times \dots \times\alpha_n$ and $\beta \coloneq  \pcd(\sigma)$. By \eqref{eq:comp-psi}, 
    \begin{multline*} \lspan{\tau_{k_1}\alpha_1,\dots,\tau_{k_n}\alpha_n,1_X;\sigma}^{X,\omega}_{A,g,n+1} \\ = \lspan{\prod_{i=1}^n(\pi_{n+1}^*\psi_i+D_{n+1,i,A})^{k_i}\cdot\pi_{n+1}^*\eva^*\alpha,\normalfont\text{st}^*\beta\cap \vfc{\Mbar_{g,n+1}^{\,J}(X,A)}}
    \end{multline*}	
	Note that 
	$$D_{n+1,i}\cap D_{n+1,j} = \emst$$ for $i \neq j$. Thus, it follows from induction and the last assertion of Lemma \ref{lem:comparison-psi} that
    $$ \prod_{i=1}^n(\pi_{n+1}^*\psi_i+ D_{n+1,i,A})^{k_i} =	\pi_{n+1}^*\prod_{i=1}^n  \psi_i^{k_i} + \sum_{j=1}^n \prod_{i=1}^n D_{n+1,i,A}^{\delta_{ij}}\pi_{n+1}^*\psi_i^{k_i-\delta_{ij}}.$$
    By the Fundamental Class axiom, Proposition \ref{fundamental-class-axiom}, 
    \begin{equation*} {\pi_{n+1}}_*(\normalfont\text{st}^*\pcd(\tau)\cap \vfc{\Mbar_{g,n+1}^{\,J}(X,A)}) = \normalfont\text{st}^*({\pi_{n+1}}_*\tau)\cap \vfc{\Mbar_{g,n+1}^{\,J}(X,A)},
    \end{equation*}
so \eqref{eq:string-equation-variant} follows from the equality $D_{n+1,i,A} = \normalfont\text{st}^*\pcd(\sigma_i)$.
\end{proof}

\begin{proposition}[Dilaton equation] We have
\begin{equation}
     \lspan{\tau_{k_1}\alpha_1,\dots,\tau_{k_{n}}\alpha_{n},\tau_1 1_X;\pi_{n+1}^!\sigma}^{X,\omega}_{A,g,n+1} = (2g-2+n)\,\lspan{\tau_{k_1}\alpha_1,\dots,\tau_{k_{n}}\alpha_{n};\sigma}^{X,\omega}_{A,g,n}
\end{equation}
for any $\alpha_1,\dots,\alpha_n\in H^*(X;\bQ)$ and $\sigma\in H_*(\Mbar_{g,n};\bQ)$.\end{proposition}

\begin{proof} By (the proof of) \cite[Lemma VI.3.7.2]{Ma99},
  \begin{equation}\label{eq:line-bundles-agree} \sigma_{n+1}^*\omega_{\pi_{n+2}}\cong \omega_{\pi_{n+1}}\lbr{\sum_{i=1}^n\sigma^{(n)}_{i}}\end{equation}
  as $G$-equivariant line bundles. Hence, $\psi_{n+1,G} = c_1(\omega_{\pi_{n+1}})_G + \sum_{j=1}^n\delta_{i,n+1,G}$. As $\cM_n$ is automorphism-free, the forgetful map $\pi_{n+1}\cl \cM_{n+1}\to \cM_n$ is its universal curve, and we have $$\deg(\omega_{\pi_{n+1}}|_C) = 2g-2$$ for any fibre $C$ of $\pi_{n+1}$. As $\cT_{n+1} = \cT_n\times_{\cM_{n}}\cM_{n+1}$, the same is true for $\cT_{n+1}\to \cT_n$. Thus, 
  $$ {\pi_{n+1}}_*(\pi_{n+1}^*\normalfont\text{st}^*\pcd(\sigma)\cdot\psi_{n+1}\cap \fcl{\cT_{n+1}}) = (2g-2+n)\,\normalfont\text{st}^*\pcd(\sigma)\cap \fcl{\cT_n}.$$
Since \eqref{eq:line-bundles-agree} is an equivariant isomorphism, we obtain the same equality in equivariant cohomology, and thus on the level of the quotient spaces:
 \begin{equation*}\label{eq:dil-eq-1}
    {\pi_{n+1}}_*(\pi_{n+1}^*\normalfont\text{st}^*\pcd(\sigma)\cdot\psi_{n+1,G}\cap \fcl{\cT_{n+1}/G}) = (2g-2+n)\,\normalfont\text{st}^*\pcd(\sigma)\cap\fcl{\cT_n/G}.
\end{equation*}
All other terms vanish because 
$$\psi_{n+1,G}\cdot \delta_{n+,j,G} = c_1(\omega_{\pi_{n+1}}(D_{n+1,j}))_G\cdot \delta_{n+1,j,G} + \sum_{i\neq j}\delta_{n+1,i,G}\cdot \delta_{n+1,j,G} = 0.$$
This completes the proof since both the obstruction bundle and section of $\cK_{n+1}$ are pullbacks of the relevant datum of $\cK_n$ by $\pi_{n+1}$.
\end{proof}

\begin{proposition}[Divisor equation] Let $\pi_{n+1}^! \cl H_*(\Mbar_{g,n};\bQ)\to H_{*+2}(\Mbar_{g,n+1};\bQ)$ be the exceptional pullback. Then 
    \begin{align}\label{eq:divisor-grav-desc}
\lspan{\tau_{k_1}\alpha_1,\dots,\tau_{k_{n}}\alpha_{n},\gamma;\pi_{n+1}^!\sigma}^{X,\omega}_{A,g,n+1} = \inpr{\gamma}{A}\,&\lspan{\tau_{k_1}\alpha_1,\dots,\tau_{k_{n}}\alpha_{n};\sigma}^{X,\omega}_{A,g,n}\\\notag +& \sum_{i=1}^{n}\lspan{\tau_{k_1}\alpha_1,\dots,\tau_{k_i-1}(\alpha_i\cdot\gamma),\dots,\tau_{k_n}\alpha_n;\sigma}^{X,\omega}_{A,g,n}
\end{align}
for any $\alpha_1,\dots,\alpha_n\in H^*(X;\bQ)$ and $\sigma\in H_*(\Mbar_{g,n};\bQ)$.
\end{proposition}

\begin{proof} Let $\alpha \coloneq \alpha_1\times\dots \times\alpha_n$. By Lemma \ref{lem:comparison-psi}, the left-hand side of \eqref{eq:divisor-grav-desc} splits into a sum of terms, one of which is
\begin{equation*}\lspan{\pi_{n+1}^*\lbr{\eva^*\alpha\cdot \prod_{i=1}^n\psi_{n,i}^{k_i}\cdot \normalfont\text{st}^*\pcd(\sigma)},\eva_{n+1}^*\gamma\cap \vfc{\Mbar_{g,n+1}^{\,J}(X,A)}} = \lspan{\gamma,A}\,\lspan{\tau_{k_1}\alpha_1,\dots,\tau_{k_{n}}\alpha_{n};\sigma}^{X,\omega}_{A,g,n}\end{equation*}
    by Proposition \ref{divisor-axiom}. The other terms are given by 
    \begin{multline*}\lspan{\pi_{n+1}^*(\eva^*\alpha\cdot \prod_{i=1}^n\psi_{n,i}^{k_i-\delta_{ij}}\cdot \normalfont\text{st}^*\pcd(\sigma))\cdot D_{n+1,j,A}\cdot \eva_{n+1}^*\gamma,\vfc{\Mbar_{g,n+1}^{\,J}(X,A)}}\\ =\lspan{\pi_{n+1}^*(\eva^*\alpha\times \prod_{i=1}^n\psi_{n,i}^{k_i-\delta_{ij}}\cdot \normalfont\text{st}^*\pcd(\sigma))\cdot  \eva_j^*\gamma\cdot D_{n+1,j,A},\vfc{\Mbar_{g,n+1}^{\,J}(X,A)}}\\ = \lspan{\tau_{k_1}\alpha_1,\dots,\tau_{k_j-1}(\alpha_j\cdot\gamma),\dots,\tau_{k_n}\alpha_n;\sigma}^{X,\omega}_{A,g,n}\end{multline*}
    for $1\leq j \leq n$. We use in the second equality that $\eva_{n+1}|_{\cc{D}_{n+1,j}(A)} = \eva_j|_{\cc{D}_{n+1,j}(A)}$.
\end{proof}

\section{Virtual localisation and equivariant GW theory}\label{sec:localisation}

In this section, we define global Kuranishi charts endowed with a compatible group action and construct an equivariant virtual fundamental class. We prove a localisation formula analogous to \cite{GP99} in the setting of global Kuranishi charts, Theorem \ref{Virtual Localisation}, and show that it applies to the equivariant GW invariants of Hamiltonian symplectic manifolds constructed in \textsection\ref{subsec:equiv-gw}. We use it to give a formula for the equivariant GW invariants of a GKM Hamiltonian manifold in \textsection\ref{subsubsec:computations-gw}.

\subsection{Equivariant virtual fundamental classes}\label{subsec:equi-vfc}
We define what it means for a global Kuranishi chart to carry a compatible group action and construct the associated equivariant virtual fundamental class. The technical background for this can be found in \textsection\ref{subsec:trace}. The construction can be considered a special case of parameterised virtual fundamental classes.

\begin{definition}\label{} Suppose $K$ is a compact Lie group acting on a space $\fM$. A \emph{$K$-equivariant global Kuranishi chart} $(\cK,\wt G)$ for $\fM$ consists of a global Kuranishi chart $\cK = (G,\cT,\cE,\fs)$ for $\fM$ together with an extension
 $$1 \to G\to \wt G\to K \to  1$$ 
 of compact Lie groups so that 
 \begin{itemize}
 	\item the $G$-actions on $\cT$ and $\cE$ extend to $\wt G$-actions
 	\item all maps are $\wt G$-equivariant
 	\item the induced action of $K$ on $\fs\inv(0)/G$ agrees with the action of $K$ on $\fM$ under the homeomorphism $\fs\inv(0)/G\cong \fM$.
 \end{itemize} 
 We say it is a \emph{rel--$C^\infty$ $K$-equivariant global Kuranishi chart} if it is rel--$C^\infty$ over a base space $S$ so that $\cT\to S$ is $K$-invariant and $\wt G$ acts by rel--$C^\infty$ diffeomorphisms.
\end{definition}

In the examples we consider later, we have $\wt G = G\times K$.

\begin{definition}\label{equivariant-vf-de} Let $(\cK,\wt G)$ be an oriented global Kuranishi chart for a $\bT$-space $\fM$. The \emph{equivariant virtual fundamental class} $\vfc{\fM}_K$ is the element of $\Hom_{H_K^*}(\chml^{*+\vdim}_K(\fM;\bQ),H^*_K)$ given by the composite
	\begin{equation}\label{e-vfc-de} \chml^{*+\vdim}_K(\fM;\bQ) \xra{\fs^*\tau_K(\cE/G) \cdot}H^{*+\dim(\cT/G)}_{K,fc}(\cT/G;\bQ)\xra{\int^K_{\cT/G}} H^*_K\end{equation}
 where the subscript fc denotes cohomology with fibrewise compact support (of the fibration $(\cT/G)_K \to BK$) and $H^*_K\coloneq H^*_K(\Pt;\bQ)$.
\end{definition}

Here we use that 
$$\chml^*_K(\fM;\bQ) = \fcolim_{W\super \fM}\chml^*_K(W;\bQ)$$
where we take the direct limit over open $K$-invariant neighbourhoods of $\fM$ in $\cT/G$. The first map in \eqref{e-vfc-de} is induced by the composition 
\begin{multline*}
    \chml^*_K(W;\bQ)\xra{\fs^*\tau_K(\cE/G)|_W\cdot} \chml^{*+\rank(\cE)}_K(W\mid \fM;\bQ) \xra{\simeq}\chml^{*+\rank(\cE)}_K(\cT/G\mid\fM;\bQ)\\\to \chml^{*+\rank(\cE)}_{K,c}(\cT/G;\bQ)
\end{multline*}
while the second map is the trace map of $(\cT/G)_K \to BK$ defined in \textsection\ref{subsec:trace}.

\begin{lemma}
    If $K$ acts freely on $\fM$ and $\cK$ is a $K$-compatible global Kuranishi chart, then the equivariant virtual fundamental class $\vfc{\fM}_K$ vanishes.
\end{lemma}

\begin{proof} Write $\cK = (G,\cT,\cE,\fs)$. As $K$ acts freely in a neighbourhood of $\fM$ in $\cT/G$, we may shrink the thickening and assume $K$ acts freely on all of $\cT/G$. Then $H^{\dim(\cT/G)+*}_{K,c}(\cT/G;\bQ)\to H^*_K(\Pt;\bQ)$ vanishes, and thus so does $\vfc{\fM}_K.$
\end{proof}

\begin{remark}\label{comparison-with-noneq-vfc} By Corollary \ref{cor:eq-trace} we have a commutative square
\begin{equation}\label{vfc-commute}\begin{tikzcd}
			\chml^{\vdim+*}_K(\fM;\bQ)\arrow[r,"\vfc{\fM}_K"] \arrow[d,""]&H^*_K(\Pt;\bQ) \arrow[d,""]\\ \chml^{\vdim+*}(\fM;\bQ)\arrow[r,"\vfc{\fM}"]&H^0(\Pt;\bQ)  \end{tikzcd} \end{equation}
   allowing us to recover $\vfc{\fM}$ partially from $\vfc{\fM}_K$. 
\end{remark}

The arguments of \textsection\ref{subsec:embedded-charts} can be carried over to the equivariant setting using Lemma \ref{lem:trace-subfibration}. We obtain the analogous statement for equivariant virtual fundamental classes.

\begin{proposition}\label{equivariant-vfc-embedded-chart} Suppose $j\cl \cK'\hkra\cK$ is a rel--$C^\infty$ embedding of oriented global Kuranishi charts as in Proposition \ref{vfc-embedded-chart} and that $K$ acts relatively smoothly on each global Kuranishi chart. If the embedding is $K$-equivariant, then
	$$j_*(e_{K\times G}(j^*\cE/\cE')\cap \vfc{\fM'}_{K}) = \pcd_{K\times G}(\cT')\cap \vfc{\fM}_{K}.$$ 
\end{proposition}

\begin{remark} Suppose $\fM$ is a moduli space parameterised by a topological space $B$. Using the obvious definition of a parameterised global Kuranishi chart, the results in \textsection\ref{subsec:trace} allow for the definition of a parameterised virtual fundamental class of the form 
$$\vfc{\fM}_B \cl \chml^{\vdim+*}_{fc}(\fM;\bQ)\to H^*(B;\bQ).$$
In particular, by Lemma \ref{lem:trace-base-change} and Lemma \ref{lem:trace-functoriality}, the results of \textsection\ref{subsec:embedded-charts} generalise to this setting. The equivariant virtual fundamental class defined here is just a special case of this construction. This is also discussed (for smooth parameter spaces) in \cite[\textsection 4.7]{AMS23}. 
\end{remark}

\subsection{Virtual localisation}\label{subsec:localisation} We recall the following localisation theorem, due to \cite{Hsi75,AB84}.

\begin{theorem}[Localisation formula]\label{} Let $M$ be a closed smooth manifold equipped with a smooth action by a torus $\bT$. Then the restriction $H^*_\bT(M;\bQ)\to H^*_\bT(M^\bT;\bQ)$ becomes an isomorphism after tensoring with the fraction field of $H_\bT^*$.
\end{theorem}

 We first discuss a generalisation of this result to orbifolds. This is classical, cf. \cite[Theorem~2.1]{Mei98}, but all accounts we found required the manifold and the group action to be smooth, so we first explain how to adapt the argument to our slightly more general situation.\par 
 For the subsequent discussion, let $M$ be a topological manifold and $\bT = (S^1)^k$ a torus. Suppose $M$ is equipped with a locally linear $\wt G$-action, where $1 \to G \to \wt G\to \bT\to 1$ is an extension by compact Lie groups and $G$ acts almost freely on $M$. In particular, we obtain an orbifold structure on $M/G$ given by the groupoid $[M/G] = [G\times M\rightrightarrows M]$ as discussed in \textsection\ref{sec:orbifold-intersection}. The $\wt G$-action induces to a $\bT$-action on $[M/G]$, as well as on the coarse moduli space $M/G$.  

\begin{lemma}\label{lem:fixed-locus-orbifold} Each component of the fixed point locus $(M/G)^\bT$ of the quotient is a canonical suborbifold of $[M/G]$. \end{lemma}

By slight abuse of notation, we denote the union of these components with their respective orbifold structure by $[M/G]^\bT$.

\begin{proof} Let $q \cl M\to \cc{M} \coloneq M/G$ be the quotient map, and set $M^{qf} \coloneq q\inv((M/G)^\bT)$.Then $x \in M^{qf}$ if and only if for any $t \in \bT$ with lift $\wt g \in G$, we have $\wt g\cdot x = g \cdot x$ for some $g \in G$. As $G$ acts almost freely and $\bT$ is connected, this is shows that
	$$M^{qf} = \set{y \in M\mid \dim(\wt G_y) = \dim(\bT)}.$$
	Let $y \in M^{qf}$ be arbitrary and fix a local slice $S\sub M$ of the $\wt G$-orbit through $y$. Then $y'\in S^{qf} \coloneq S\cap M^{qf}$ if and only if the identity component $\wt G_{y',0}$ of $\wt G_{y'}$ agrees with $\wt G_{y,0}$. This shows that $S^{qf}$ is a submanifold of $S$. Moreover, an open neighbourhood $U$ of $y$ in $M^{qf}$ is given by 
	\begin{equation}\label{eq:local-quotient-fixed-locus}U \cong \wt G\times_{\wt G_y} S^{qf}. \end{equation}
	It follows that each path component of $M^{qf}$ is a locally flat submanifold of $M$, on which $G$ acts almost freely and locally linearly.
\end{proof}

\begin{lemma}\label{lem:fixed-locus-normal-bundle} Suppose the tangent microbundle $T_\mu M$ admits a $\wt G$-equivariant vector bundle lift $E$ with a subbundle $E'\sub E$ so that $E'|_{G(x)}$ is a vector bundle lift of $T_\mu (G\cdot x)$. Then the embedding $j\cl [M/G]^\bT\hkra [M/G]$ admits a $\bT$-equivariant normal bundle. 
\end{lemma}

\begin{proof} Let $V\sub E$ be a $\wt G$-invariant complement of $E'$ and given a path-component $M_\lambda$ of $M^{qf}$, define 
	$$N_\lambda \coloneq \{(x,v)\in V|_{M_\lambda}\mid \exists \wt g\in \wt G_{x,0}: \wt g\cdot v\neq v\}.$$
	By definition, there exists a $\wt G$-equivariant neighbourhood $U \sub E$ of the zero section and an equivariant embedding $\psi \cl U \to M\times M$ onto a neighbourhood of the diagonal. Then $\psi$ descends to an open embedding $$(U\cap V)/G\to M/G\times M/G.$$ 
	Hence, we can think of $V/G$ as an orbibundle lift of the tangent microbundle of $M/G$. Let $x \in M_\lambda$ be arbitrary and let $S$ be a slice through $x$ as in \eqref{eq:local-quotient-fixed-locus}. By the choice of $S$, the stabiliser $\wt G_{x'}$ of any $x'\in S$ is conjugate to a subgroup of $\wt G_{x}$.  Due to the equivariance of $\psi$, we have $\pr_2\psi(x,v)\in S\cap M^{qf}$ for $v \in V_x$ if and only if $\wt g\cdot v = v$ for any $\wt g\in \wt G_{x,0}$. As $G\cdot x \sub M_\lambda$, this shows that the normal bundle of $M_\lambda$ is given by $N_\lambda$.
\end{proof}

We abbreviate $R \coloneq H^*_\bT(\Pt;\bQ) \cong \bQ[u_1,\dots,u_k]$ and let $Q$ denote its fraction field.

\begin{proposition}[Localisation formula]\label{prop:orbifold-localisation} Suppose $[M/G]$ satisfies the assumptions of Lemma \ref{lem:fixed-locus-normal-bundle}. Then the inclusion $j \cl [M/G]^\bT\to [M/G]$ induces an isomorphism 
$$H^*_{\bT,c}([M/G];\bQ)\otimes_R Q\xra{\simeq} H^*_c([M/G]^\bT;\bQ)\otimes_\bQ Q$$ with inverse given by $\frac{1}{e(N)}(j|_{M})_!$. In particular, for $\alpha\in H^*_{\bT,c}([M/G];\bQ)$ we have 
	\begin{equation}\label{eq:integration-formula-orbifold} \int^\bT_{M/G} \alpha = \s{\lambda}{\int^\bT_{M_\lambda/G}}\frac{j^*\alpha}{e(N_\lambda)}.\end{equation}
\end{proposition}

\begin{proof} By Lemma \ref{lem:fixed-locus-orbifold}, there exists an $\bT$-invariant neighbourhood $W$ of $(M/G)^\bT$ inside $M/G$ and an equivariant retraction $r \cl W \to (M/G)^\bT$. By \eqref{gysin-embedding}, we have 
	$$j^*j_!(-) = \s{\lambda}{e(N_\lambda)\cdot -} \qquad \qquad j_!j^*(-) = \s{\lambda}{r^*e(N_\lambda)\cdot -}.$$
	Hence it suffices to show that $e(N_\lambda)$ is invertible in $H^*_\bT([M_\lambda/G];\bQ)\otimes_RQ$. As $M_\lambda/G$ has finite cohomological dimension, it suffices to show that the image of $e(N_\lambda)$ in $H^0([M_\lambda/G];\bQ)\otimes_\bQ Q$ is nonzero. This is can be checked at an arbitrary point $\bar{x}\in M_\lambda/G$ and follows from the fact that the $\wt G_{x,0}$-action on $(N_\lambda)_x$ has no nonzero fixed points, whence the $\bT$-action on $(N_\lambda/G)_{\bar{x}}$ only fixes the origin.
\end{proof}

Let now $\fM$ be a compact space with an oriented rel--$C^\infty$ global Kuranishi chart $\cK = (G,\cT/\cM,\cE,\fs)$ over a smooth manifold $\cM$. Suppose $\fM$ admits a $\bT$-action and $\cK$ lifts to a rel--$C^\infty$ equivariant global Kuranishi chart $(\cK,\wt G)$. By Lemma \ref{lem:fixed-locus-orbifold}, we have a decomposition $$\cT^{qf} =\djun{\lambda}\cT_\lambda$$ into its path-components, where the $\cT_\lambda$ are $\wt G$-invariant submanifolds of possibly varying dimension. For each $\lambda$, we have a splitting $\cE|_{\cT_\lambda} = \cE^f_\lambda\oplus \cE^m_\lambda$, where 
$$\cE^f_\lambda \coloneq \set{(x,e)\in \cE|_{\cT_\lambda} \mid \forall  \wt g\in \wt G_{x,0} : \wt g\cdot e = e},$$
and $\cE^m_\lambda$ is a $\wt G$-invariant complement, the choice of which is irrelevant. Since $\fs$ is $\wt G$-equivariant, $\fs(\cT_\lambda)\subset \cE^f_\lambda$. Moreover, the rank of $\cE^f_\lambda$ is constant along $\cT_\lambda$. This shows that

\begin{lemma} $\cK_\lambda\coloneq (G,\cT_\lambda,\cE^f_\lambda,\fs|_{\cT_\lambda})$ is a global Kuranishi chart for $\fM_\lambda\coloneq \fs\inv(0)\cap \cT_\lambda/G$ and 
$$\fM^\bT = \djun{\lambda} \fM_\lambda.$$
\end{lemma} 

Our assumptions imply that $T_\mu\cT$ admits a $\wt G$-equivariant vector bundle lift of $T_\mu\cT$, given by $T_{\cT/\cM}\oplus \pi^*T_\cM$, where $\pi \cl \cT\to \cM$ is the structural map. By a straightforward generalisation of \cite[Lemma~5.9]{HS22} to compact groups, the $\wt G$-action is locally linear, and we may shrink $\cT$ to assume that the $\wt G$-action has only finitely many orbit types. 

\begin{lemma}\label{lem:gkc-normal-bundle} Up to replacing $\cK$ by an equivalent rel--$C^\infty$ global Kuranishi chart over $\cM$, we may assume that $\cT$ satisfies the assumptions of Lemma \ref{lem:fixed-locus-normal-bundle}.
\end{lemma}

\begin{proof} By \cite[Theorem 3.6]{La79}, there exists an orthogonal $G$-representation $V$ so that $\cT\times V$ admits a smooth structure with respect to which the diagonal $G$-action is smooth. Moreover, there exists a $G$-equivariant isomorphism 
	$$T(\cT\times V)\to T_{\cT/\cM}\oplus \pi^*T_\cM\oplus V$$ 
	by \cite[Addendum 4.14]{AMS21}. The infinitesimal action on $\cT\times V$ thus provides the desired embedding $(\cT\times V)\times\fg \hkra T_{\cT/\cM}\oplus \pi^*T_\cM\oplus V$. Replacing $(G,\cT,\cE,\fs)$ by $(G,\cT\times V,\cE\times V,\fs\times\ide_V)$, the claim follows.
\end{proof}

\begin{theorem}[Virtual localisation formula]\label{Virtual Localisation} Let $\cK= (G,\cT/\cM,\cE,\fs)$ be a rel--$C^\infty$ global Kuranishi chart for the $\bT$-space $\fM$ and suppose $\cK$ is equipped with a compatible rel--$C^\infty$ $\bT$-action. Then, 
	\begin{equation}\label{virt-loc}\vfc{\fM}_\bT =  \s{\lambda}{{j_\lambda}_*\lbr{\frac{e_{\bT}(\cE^m_\lambda/G)}{e_{\bT}(N_\lambda/G)}\cap \vfc{\fM_\lambda}_\bT}}\end{equation}
 in $(\chml\ucc_\bT(\fM;\bQ)\otimes_R Q)\dul$.\end{theorem}

\begin{proof} By definition, 
	$$\lspan{\vfc{\fM}_\bT,\alpha} = \int^\bT_{\cT/G} \fs^*\tau_\bT(\cE/G)\cdot \alpha$$ 
	for any $\alpha\in \chml\ucc_\bT(\fM;\bQ)$. Thus the claim follows from Proposition \ref{equivariant-vfc-embedded-chart} combined with \eqref{eq:integration-formula-orbifold}. 
\end{proof}

%\begin{remark}\label{rem:localisation-extended} Suppose we can find for each component $\fM_\lambda$ a global Kuranishi chart $\cK'_\lambda$ that is equivalent to $\cK_\lambda$ and also embeds into $\cK$ so that $\cT'_\lambda$ admits a normal bundle in $\cT$. Then the same formula holds where $\cE^m|_{\cT_\lambda}$ is replaced by the complement of $\cE'_\lambda\sub \cE|_{\cT'_\lambda}$ and $N_{\cT_\lambda/\cT}$ is replaced by $N_{\cT'_\lambda/\cT}$, as long as the latter is defined.
%\end{remark}

%\begin{corollary}\label{cor:vanishing-fixed-point-components} Suppose $\cK$ and $\cM$ are as in Theorem \ref{Virtual Localisation}. If $\cE = \cE'\oplus V$, where $V$ is a trivial bundle on which $\bT$ acts trivially, then we only have to consider path-components $\cT_\lambda$ of $\cT^{qf}$ with stabiliser group $\bT$.\end{corollary}
%
%\begin{proof} Indeed, if $H \leq \wt G$ is the stabiliser group of some $y\in \cT_\lambda$ with $H\neq \bT$, then $V_\lambda^m \neq 0$. Thus $e(\cE^m/G) = 0$ and so there is no contribution of $\vfc{\fM_\lambda}$ 
%\end{proof}

	%\footnote{We omitted the shift in degree in the equation because it depends on the dimension of the path component of $\cT^{qf}$.} 

\subsection{Equivariant GW invariants}\label{subsec:equiv-gw}

Suppose $(X,\omega)$ is a closed symplectic manifold equipped with a Hamiltonian action by a compact connected Lie group $K$ with moment map $\mu$. Let $A\in H_2(X;\bZ)$. We first note the following compatibility with our construction of a global Kuranishi chart.

\begin{lemma}\label{equivariant-global-chart} If $J\in \cJ_\tau(X,\omega)$ is $K$-invariant, the following holds.
	\begin{enumerate}[\normalfont 1),leftmargin=15pt,ref=\arabic*]
		\item\label{equi-aux-datum} There exists an unobstructed auxiliary datum $(\conn^X,\cO_X(1),p,\cU,\lambda,k)$ such that $\conn^X$ is $K$-invariant, $\cO_X(1)$ admits a unitary $K$-linearisation, and $\lambda_\cU$ is $K$-invariant.
		\item\label{equi-existence} The resulting global Kuranishi chart obtained by \cite[Construction~3.13]{HS22} using this auxiliary datum admits a relatively smooth compatible $K$-action.
		\item\label{equi-uniqueness} If $(\conn^{X'},\cO_{X'}(1),p',\cU',k')$ is another unobstructed auxiliary datum satisfying the conditions of \eqref{equi-aux-datum}, then the associated global Kuranishi charts are equivalent via $K$-compatible charts such that the moves respect the $K$-actions. If $J'$ is another $\omega$-compatible almost complex structure such that the $K$-action is $J'$-holomorphic, then the cobordism constructed in \cite[\textsection6.2]{HS22} can be chosen to be $K$-compatible.
	\end{enumerate}
\end{lemma}

\begin{proof} By \cite[Corollary 1.4]{MiR01} we can find a polarisation $\cO_X(1)$ as in Definition \ref{aux-choices-defined}\eqref{polarization-on-target} such that the $K$-action on $X$ lifts to a fibrewise linear unitary $K$-action on $\cO_X(1)$.
If $\tilde{\conn}^X$ is a $J$-linear connection on $T_X$, then so is the connection $\conn^X$ obtained by averaging $\tilde{\conn}^X$ over $K$. Given these two data, \cite[Theorem~3.17]{HS22} asserts that we can complete them to an unobstructed auxiliary datum $(\conn^X,\cO_X(1),p,\tilde{\cU},\tilde{k})$. Averaging the cut-off functions in $\tilde{\cU}$ over $K$ and possibly increasing $\tilde{k}$ to $k$, we obtain an unobstructed auxiliary datum $(\conn^X,\cO_X(1),p,\cU,\lambda,k)$ as claimed. By \cite[Construction~3.13]{HS22}, the associated global Kuranishi chart is $K$-compatible. The statements about relative smoothness follow from the description of its universal property in \cite[Definition~5.3]{HS22}. Finally, \eqref{equi-uniqueness} can be seen by noting that the proofs of \cite[\textsection 6]{HS22} carry over verbatim to the equivariant setting.
\end{proof}

\begin{definition} The \emph{equivariant Gromov--Witten invariants} of $(X,\omega,\mu)$ are the maps
\begin{equation}\normalfont\text{GW}^{X,\omega,\mu}_{g,n,A} \coloneq ((\eva\times \normalfont\text{st})_K)_*\vfc{\Mbar_{g,n}(X,A;)J}_K \cl H^{*+\vdim}_K(X^n\times \Mbar_{g,n};\bQ)\to H^*_K\end{equation}
where $J$ is any $K$-invariant $\omega$-tame almost complex structure on $X$. We write 
$$\lspan{\alpha_1,\dots,\alpha_n;\sigma}^{X,\omega,\mu}_{g,n,A} \coloneq \normalfont\text{GW}^{X,\omega,\mu}_{g,n,A}(\alpha_1\times\dots\times \alpha_n\times\pcd(\sigma))$$
for $\alpha_j \in H^*_K(X;\bQ)$ and $\sigma\in H_*(\Mbar_{g,n};\bQ)$.
\end{definition}

\begin{remark} By \cite[Proposition 5.8]{Kir84}, $(X,\mu)$ is equivariantly formal. Hence, we can recover the non-equivariant GW invariants from the equivariant invariants by \eqref{vfc-commute}.
\end{remark}

\begin{proposition}\label{equivariant-axioms} The invariants $\normalfont\text{GW}^{X,\omega,\mu}_{g,n,A}$ satisfy the equivariant analogue of the Kontsevich--Manin axioms.
\end{proposition}

\begin{proof} The equivariant Kontsevich-Manin axioms are the properties of the GW homorphisms, considered as maps $H^*(X^n\times\Mbar_{g,n};\bQ)\to H^{*}(\Pt;\bQ)$, in the setting of equivariant cohomology. The arguments of \textsection\ref{sec:axioms} carry over, using Proposition \ref{equivariant-vfc-embedded-chart} instead of Proposition \ref{vfc-embedded-chart}.  As an example, we discuss the Fundamental class axiom. It asserts that
\begin{equation*}
   \lspan{\alpha_1,\dots,\alpha_n,1_X;\sigma}^{X,\omega,\mu}_{g,n+1,A} =\lspan{\alpha_1,\dots,\alpha_n, {\pi_{n+1}}_*\sigma}^{X,\omega,\mu}_{g,n,A}
\end{equation*}
as elements of $H^*_K$. Let $\cK_n$ be a global Kuranishi chart for $\Mbar_{g,n-1}^{\,J}(X,A)$ equipped with a compatible $K$-action and let $\cK_{n+1}$ be its pullback along the forgetful map $\pi_{n+1}$. By the proof of Proposition \ref{fundamental-class-axiom}, we may replace them with global Kuranishi charts $\wt \cK_n$ and $\wt \cK_n{n+1}$ which are still compatible with the group action and where the forgetful map $\wt\pi_{n+1}\cl \wt\cT_{n+1}/\wt G\to \wt\cT_n/\wt G$ satisfies 
$$(\wt\pi_{n+1})_!\;\text{st}^* = \text{st}^*\;{\pi_{n+1}}_!$$
in ordinary cohomology. Now we may conclude by using the straightforward generalisation of Lemma \ref{projection-formula-up-codimension} to equivariant cohomology. The other axioms are left to the interested reader. 
\end{proof}

Set $QH_K^*(X,\omega) \coloneq H^*_K(X;\bQ)\otimes_\bQ \Lambda$ and endow it with the product 
\begin{equation*}\alpha \concat \beta = \s{A\in H_2(X;\bZ)}{(\alpha\concat \beta)_A t^{\omega(A)}},\end{equation*}
where 
\begin{equation}\label{eq:3-eq-gw}\int_{X}^K(\alpha\concat \beta)_A\cdot \gamma = \normalfont\text{GW}^{X,\omega,\mu}_{0,3,A}(\alpha,\beta,\gamma)\end{equation}
for any $\gamma\in H^*_K(X;\bQ)$. By the equivariant Symmetry and Splitting axiom, this is graded commutative and associative. Note that \eqref{eq:3-eq-gw} determines $(\alpha \concat \beta)_A$ uniquely since $(X,\mu)$ is equivariantly formal.

\subsection{Equivariant GW invariants of Hamiltonian GKM manifolds}\label{subsubsec:computations-gw} Throughout, $\bT$ will denote a torus. We consider actions satisfying certain conditions, which are essential for our computations later on. These actions are named after Goresky--Kottwitz--MacPherson, who studied their cohomological properties in \cite{GKM98}. 

\begin{definition}\label{de:gkm} An effective $\bT$-action on a closed smooth manifold $M$ \emph{satisfies the GKM condition} if 
	\begin{enumerate}[\normalfont i),leftmargin=20pt,ref=\roman*]
		\item\label{eq-formal} $M$ is equivariantly formal, i.e., $H^*_\bT(M;\bZ)\cong H^*(M)\otimes H^*_\bT$,
		\item $M^\bT$ is finite,
		\item for each $p \in M^\bT$, the weights of the irreducible $\bT$-subrepresentation of $T_pM$ are linearly independent.
	\end{enumerate} 
\end{definition}

By \cite[Proposition~2.3]{GKZ22b}, this shows that the subspace 
$$M_1 \coloneq \{x \in M:\dim(\bT\cdot x) \leq 1\}$$ 
is a union of cleanly-intersecting spheres. Note that in the Hamiltonian case, condition \eqref{eq-formal} is automatically satisfied by \cite[Proposition~5.8]{Kir84}. Refer to \cite{Kur09} for a general introduction to GKM theory. We just mention an important invariant of GKM manifolds.

\begin{definition}\label{} Suppose $\bT$ is a $k$-dimensional torus with Lie algebra $\ft$ and $X$ admits a GKM action by $\bT$. Its \emph{GKM graph} $\Gamma_{X}$ has vertices $V(\Gamma_X) = X^\bT$ with $\{x,y\}\in E(\Gamma_X)$ if there exists an $\bT$-invariant sphere containing both $x$ and $y$. It is equipped with a map $V(\Gamma_X)\to \bZ^*_\ft/\{\pm 1\}$, which maps a vertex $x$ to the corresponding weight of the isotropy representation on $T_xX$.
 \end{definition}

\begin{example} In \cite{Tol98}, Tolman constructed a symplectic $6$-manifold with an effective Hamiltonian $T^2$-action that does not admit an invariant K\"ahler form. The manifold is constructed by gluing two $T^2$-invariant open subsets of projective manifolds along an invariant hyperplane. It satisfies the GKM condition and her construction generalises to give an infinite family of such examples, \cite{LiPa21}.
\end{example}

For the rest of this subsection, fix a Hamiltonian $\bT$-manifold $(X,\omega,\mu)$ that satisfies the GKM condition. Let $\cS_X = \{S_1,\dots,S_\ell\}$ be the set of $\bT$-invariant spheres and fix an $\bT$-invariant almost complex structure $J$ that leaves each $TS_i$ invariant. The virtual localisation formula, Theorem \ref{thm:formula-eq-gw}, requires notation that will be introduced in the next subsection. The result relies on Theorem~\ref{Virtual Localisation} and an explicit identification of the components of the fixed point locus of the moduli space and the relevant Euler classes. We first identify the components of $\Mbar_{g,n}^{\,J}(X,A)^\bT$ in \textsection\ref{subsubsec:components}, and subsequently state Theorem \ref{thm:formula-eq-gw}. Its proof, that is, the computation of the respective (fraction of) Euler classes, is carried out in \textsection \ref{subsubsec:euler}. 

\subsubsection{Components of fixed points locus}\label{subsubsec:components} Due to the GKM assumption, we can associate to each component of $\Mbar_{g,n}^{\,J}(X,A)^\bT$ a unique decorated graph $\Gamma$ as in the case of smooth toric varieties discussed in \cite{GP99,Liu13}. For this, let

\begin{definition}\label{de:decorated-graph} A \emph{decorated graph} $\Gamma = (V(\Gamma),E(\Gamma))$ is a finite non-empty graph together with 
	\begin{enumerate}[\normalfont i),leftmargin=20pt]
		\item a marking $\Lambda\cl \{1,\dots,n\}\to V(\Gamma)$; denote $n_v \coloneq \Lambda\inv(\{v\})$;
		\item $p_\Gamma \cl V(\Gamma) \to X^{\bT}$ maps each vertex to a fixed point; 
		\item $S_\Gamma\cl E(\Gamma) \to \cS_X$ sends an edge to a $\bT$-invariant symplectic sphere so that if $e = \{v,v'\}$, then $p_\Gamma(v),p_\Gamma(v') \in S_\Gamma(e)$;
		\item $g \cl V(\Gamma) \to \bN_0$ associates to each vertex a genus;
		\item $d \cl E(\Gamma) \to \bN$ associates to $e$ a degree.
	\end{enumerate}
	We denote by $E_v$ the set of edges adjacent to the vertex $v$ and decompose the vertex set as
	$$V(\Gamma) = V^s(\Gamma)\sqcup V^n(\Gamma)\sqcup V^b(\Gamma)\sqcup V^m(\Gamma),$$ 
	where 
	\begin{itemize}
		\item $V^s(\Gamma) = \{v \in V(\Gamma)\mid 2g(v)-2 + |n_v| +|E_v| \geq 3\}$
		\item $V^n(\Gamma) = \{v\in V(\Gamma) \mid g(v) =0,\, |n_v| = 0,\;|E_v| = 2\}$
		\item $V^b(\Gamma) = \{v\in V(\Gamma) \mid g(v) =0,\, |n_v| = 0,\;|E_v| = 1\}$
		\item $V^m(\Gamma) = \{v \in V(\Gamma)\mid g(v) = 0,\, |n_v| = 1\}$.
	\end{itemize}
	Let $F^s(\Gamma) \coloneq \{(e,v)\in E(\Gamma)\times V^s(\Gamma)\mid e \in E_v\}$ be the set of stable flags.
	We associate to $\Gamma$ the moduli space $$\Mbar_\Gamma \coloneq \p{v\in V^s(\Gamma)}{\Mbar_{g(v),n_v\cup E_v}}$$ 
	and let $\varphi_\Gamma\cl \Mbar_\Gamma \to \Mbar_{g,n}$ be the induced clutching map.
\end{definition}

 We call two decorated graphs $\Gamma$ and $\Gamma'$ \emph{isomorphic} if there exists an isomorphism $\Gamma\to \Gamma'$ of graphs that preserves all decorations. 
Given $A\in H_2(X;\bZ)$ and $g,n\geq 0$, denote by $G_{g,n}(X,A)$ the set of isomorphism classes of graphs $\Gamma$ as above such that 
$$\s{v \in V(\Gamma)}{g(v)} = g \qquad \qquad \s{e\in E(\Gamma)}{d_e[S_\Gamma(e)]} = [A].$$

In particular, we can think of any $[\Gamma]\in G_{g,n}(X,A)$ as a graph equipped with a morphism $\Gamma\to \Gamma_X$ of graphs, as well as a labelling and genus and degree decorations. 

\begin{proposition}\label{prop:fixed-points-gw-moduli} There exists a canonical bijection between the components of $\Mbar_{g,n}^{\,J}(X,A)^\bT$ and $G_n(X,A)$.
\end{proposition}

\begin{proof} We make a few preliminary observations about the elements of the fixed-point locus.\par
	Suppose $[u,C,x_*]\in \Mbar_{g,n}^{\,J}(X,A)^{\bT}$ and let $C'\sub C$ be an irreducible component such that $g(C') > 0$ or $C'$ has at least $3$ special points. Then for $x \in C'$, its orbit
	$$\bT\cdot u(x) \sub u(\Aut(C',y_1,\dots,y_k)\cdot x),$$ 
	is finite. As $\bT$ is compact and connected, $\bT_{u(x)} =\bT$, so $u(x)$ must be a fixed point. As the fixed points are isolated, $u$ contracts $C'$ to a single fixed point. If $x \in C$ is an arbitrary point, then $\bT\cdot u(x)\sub \im(u)$. If $v_\xi$ denotes the vector field associated to $\xi\in \ft$, this implies $\set{v_\xi(x),Jv_\xi(x)\mid \xi \in \ft}\sub \im(du(x))$. As $v_\xi(u(x))\notin \bR\lspan{v_\eta(u(x)),J_{u(x)}v_\eta(u(x))}$ for linearly independent $\xi,\eta\in \ft$, we have $$\dim(\{\xi \in \ft\mid v_\xi(u(x))\neq 0\}) \leq 1.$$ 
	Thus, $\im(u)$ is contained in the union of $\bT$-orbits of dimension at most one. If $C'$ is an irreducible component not contracted by $u$, then $C'$ is a sphere with at most two special points. Since the complement of a finite number of points in the sphere is connected, $u$ maps $C'$ onto a sphere $S\in \cS_X$. Hence, it induces a surjective holomorphic map $u'\cl \bP^1 \to (S,J|_S)$. The $S^1$-action on $(S,J|_{S})$ induced by $\mu$ extends to a holomorphic $\bC\units$-action with respect to which $u'$ remains invariant (up to automorphisms). Identify $S$ with $\bP^1$ and let $R \geq 0$ be sufficiently large such that ${u'}\inv(S^2\sm \bD_R)$ is a disjoint union of open subsets containing each at most one pole. Suppose $u'$ has two poles; let $\gamma$ be a path between them that is disjoint from ${u'}\inv(\{0\})$. Then there exists $t > 0$ so that $|t\cdot u(\gamma(s))| > R$ for any $s \in [0,1]$. Let $\phi\in \text{PSL}_\bC(2)$ be such that $t \cdot u = u\phi$. Then $\phi(\gamma)$ has endpoints in two distinct components of ${u'}\inv(S^2\sm \bD_R)$, a contradiction. Similarly, $u'$ has exactly one zero. It now follows from the Riemann-Hurwitz formula that $u'$ is totally branched over $0$ and $\infty$ and of degree 
	$$\deg_u(C') = \frac{\lspan{\omega,u_*[C']}}{\omega(S)}.$$
	Given a component $\cP$ of $\Mbar_{g,n}^{\,J}(X,A)^\bT$ and $[u,C,x_*]\in\cP$, we obtain a graph $\Gamma \in G_{g,n}(X,A)$ by setting
	$$V(\Gamma) \coloneq \pi_0(u\inv(X^\bT))$$
	with an (unoriented) edge connecting $C_0$ and $C_1$ in $V(\Gamma)$ if and only if there exists a non-contracted component $C' $ so that $C_0 \cup C' \cup C_1$ is a connected, possibly nodal surface contained in $C$. Then $V^s(\Gamma)$ corresponds to connected sub-curves of $C$ contracted to a point and $V^n(\Gamma)$ and $V^m(\Gamma)$ correspond to nodal, respectively marked points of vertices lying on an edge, while $V^b(\Gamma)$ corresponds to non-special branch points. The decorations are defined in the obvious way. 
\end{proof}

We denote by $\Mbar(\Gamma,J)$ the component of $\Mbar_{g,n}^{\,J}(X,A)^\bT$ corresponding to $\Gamma\in G_{g,n}(X,A)$.

\begin{theorem}\label{thm:formula-eq-gw} If $(X,\omega,\mu)$ is a Hamiltonian $\bT$-manifold satisfying the GKM condition, then its equivariant GW invariants are given by 
	\begin{equation}\label{eq:formula-equivariant-gw}
		\lspan{\alpha_1,\dots,\alpha_n;\sigma}^{X,\omega,\mu}_{g,n,A} = \s{\Gamma\in G_{g,n}(X,A)}{\frac{1}{|\Aut(\Gamma)|}\lspan{j_{\Gamma}^*\eva^*\alpha\cdot\varphi_\Gamma^*\pcd(\sigma)\cdot\textbf{w}(\Gamma),\fcl{\Mbar_\Gamma}_\bT}}
	\end{equation}
	for any $\alpha_1,\dots,\alpha_n\in H^*_\bT(X;\bQ)$ and $\sigma \in H_*(\Mbar_{g,n};\bQ)$. 
	%	where $\varphi_\Gamma$ is the canonical clutching map induced by $\Gamma$. 
\end{theorem}

Here $\textbf{w}(\Gamma) \in H^*(\Mbar_\Gamma;\bQ)\otimes_R Q$ is a certain class defined below. Recall that $R = H^*_\bT$ and $Q$ is its fraction field. Define
$$w(p) \coloneq e_\bT(T_pX) \qquad\normalfont\text{and} \qquad w(S,p) \coloneq e_\bT(T_pS)$$
in $R$ for $p \in X^\bT$ and $S\in \cS_X$ containing $p$. 
For $\Gamma \in G_n(X,A)$ and $(e,v)\in F^s(\Gamma)$, we set 
$$w(e,v) \coloneq  e_\bT(T_{x_{e,v}}C_e) = \frac{w(S_\Gamma(e),p_\Gamma(v))}{d_e}\in Q,$$ 
where $\{x_{e,v}\} = C_v\cap C_e$.\par  
We can now define the cohomology classes on moduli spaces of stable curves that will appear in the definition of $\textbf{w}(\Gamma)$. Given $p \in X^\bT$ and $g \geq 0$, define 
	\begin{equation}\label{eq:vertex-weight}\textbf{h}(p,g) = \p{\substack{S\in \cS_X\\p\in S}}{\frac{\sum_{j = 0}^g(-1)^j\lambda_j w(S,p)^{g-j}}{w(S,p)}},
	\end{equation}
	where $\lambda_j$ is the $j^{\normalfont\text{th}}$ Hodge class of $\Mbar_\Gamma$.
For a $\bT$-invariant sphere $S$, write its normal bundle as
$$N_{S/X} =  L_1 \oplus \dots \oplus L_{m-1},$$
 where $L_i$ is a $\bT$-invariant complex line bundle of degree $a_i$. Given an integer $d$, define
\begin{equation}\label{eq:edge-weight}\textbf{h}(S,d) \coloneq  \frac{(-1)^{d}d^{2d}}{(d!)^2w(S,p)^{2d}}\prod_{i =1}^{m-1}{b\lbr{\frac1d w(S,p),e_\bT((L_i)_p),da_i}} \in Q,\end{equation}
where $p \in S\cap X^\bT$ is arbitrary and 
$$b(x,y,r) \coloneq \begin{cases}
	\p{\substack{0 \leq k \leq r}}{(x-ky)\inv} \quad & r \geq 0\\
	\p{\substack{0 \leq k \leq -r-1}}{(x+ky)} \quad & r < 0
\end{cases}$$
 for $x,y \in Q$ so that $x \notin \bZ y$, and $r \in \bZ$. This is given in for the classes we apply the function $b$ on by the third property of Definition \ref{de:gkm}.

\begin{definition}\label{} The \emph{weight} $\textbf{w}(\Gamma)$ of $\Gamma\in G_n(X,A)$ is 
	\begin{multline}\label{de:weight-function}\textbf{w}(\Gamma) \coloneq \p{\substack{v\in V^n(\Gamma)\\E_v = \{e,e'\}}}{\frac{w(p_\Gamma(v))}{w(e,v)+w(e,v')}}\p{(e,v)\in F^s(\Gamma)}{\frac{w(p_\Gamma(v))}{w(e,v)-\psi_{(v,e)}}} \p{\substack{v\in V^b(\Gamma)\\ (e,v)\in F(\Gamma)}}{w(e,v)}\\\cdot\p{v \in V^s(\Gamma)}{\textbf{h}(p_\Gamma(v),g_v)} \p{e \in E(\Gamma)}{\textbf{h}(S_\Gamma(e),d_e)} \end{multline}
	where $\psi_{(v,e)}$ is the $\psi$-class on $\Mbar_\Gamma$ at the marked point $(v,e)$, which corresponds to $e \in E_v$.
\end{definition}

%Here $\textbf{w}(\Gamma)$ is the \emph{weight} of $\Gamma\in G_n(X,A)$ defined by
%\begin{multline}\label{de:weight-function}\textbf{w}(\Gamma) \coloneq \p{\substack{v\in V^n(\Gamma)\\E_v = \{e,e'\}}}{\frac{w(p_\Gamma(v))}{w(e,v)+w(e,v')}}\p{(e,v)\in F^s(\Gamma)}{\frac{w(p_\Gamma(v))}{w(e,v)-\psi_{(v,e)}}} \p{\substack{v\in V^b(\Gamma)\\ (e,v)\in F(\Gamma)}}{w(e,v)}\\\cdot\p{v \in V^s(\Gamma)}{\textbf{h}(p_\Gamma(v),g_v)} \p{e \in E(\Gamma)}{\textbf{h}(S_\Gamma(e),d_e)} \end{multline}
%where $\psi_{(v,e)}$ is the $\psi$-class on $\Mbar_\Gamma$ at the marked point $(v,e)$, which corresponds to $e \in E_v$, while 
%\begin{equation}\label{eq:vertex-weight}\textbf{h}(p,g) = \p{\substack{S\in \cS_X\\p\in S}}{\frac{\sum_{j = 0}^g(-1)^j\lambda_j w(S,p)^{g-j}}{w(S,p)}},
%\end{equation}
%where $\lambda_j$ is the $j^{\normalfont\text{th}}$ Hodge class of $\Mbar_\Gamma$. The last term is defined to be
%\begin{equation}\label{eq:edge-weight}\textbf{h}(S,d) =  \frac{(-1)^{d}d^{2d}}{(d!)^2w(S,p)^{2d}}\p{p\in S'\neq S}{b(d\inv w(S,p),w(S',p),da_i)},\end{equation}
%where $p \in S\cap X^\bT$ is arbitrary and for $r \in \bZ$
%$$b(x,y,r) = \begin{cases}
%	\prod_{k = 0}^r\frac{1}{x-ky} \quad & r \geq 0\\
%	\prod_{k = 1}^{-1-r}{(x+ky)} \quad & r < 0.
%\end{cases}$$

\begin{remark}[Comparison with \cite{Liu13}] In \cite{Liu13}, Liu streamlines \eqref{de:weight-function} by introducing additional notation to deal with the unstable vertices. We refrain from this as the global Kuranishi chart set-up already requires an abundance of symbols. However, despite the superficial difference between Theorem \ref{thm:formula-eq-gw} and \cite[Theorem 73]{Liu13}, they give the same formula in the case of a toric symplectic manifold, that is, a smooth toric variety by \cite{Del88}.\end{remark}

\subsubsection{Proof of Theorem \ref{thm:formula-eq-gw}}\label{subsubsec:euler} Let $\Gamma\in G_n(A,J)$. We identify the class with an arbitrary representative. Fix an unobstructed $\bT$-compatible auxiliary datum $(\conn^X,\cO_X(1),p,\cU,\lambda,k)$ as in Lemma \ref{equivariant-global-chart}\eqref{equi-aux-datum}, and let $\cK = (G,\cT/\cM,\cE,\fs)$ be the associated global Kuranishi chart for $\Mbar_{g,n}^{\,J}(X,A)$. Denote by $\cT_\Gamma$ the preimage of $\Mbar(\Gamma,J)$ under the forgetful map ${\fs}\inv(0) \to \Mbar_{g,n}^{\,J}(X,A)$. Given Theorem \ref{Virtual Localisation} and Proposition \ref{prop:fixed-points-gw-moduli}, it remains to show that the Euler class of the normal bundle of the global Kuranishi chart of $\Mbar(\Gamma,J)$ agrees with $\textbf{w}(\Gamma)$.\par
We will make use of an intermediate global Kuranishi chart to simplify our computations. To motivate this, observe that lying in $\Mbar(\Gamma,J)$ imposes constraints on the topological type of the domain, which are encoded in $\Gamma$. Thus, we can consider $\Mbar(\Gamma,J)$ as the fixed point locus inside the $\bT$-invariant moduli space of maps $\Mbar_\Gamma^{\,J}(X,A)$, whose domains satisfy the same constraints on their topological type. Finally, we compute the virtual normal bundle of $\Mbar_\Gamma(X,A;J)$ inside the full moduli space of stable maps. This allows us to separate the contributions coming from deformations of the map and those of the domain.

\subsection*{Base space:} Given $[\iota,C,x_*]\in \cM$, define 
$$V'_\iota \coloneq \set{C' \sub C \mid \text{ maximal connected curve such that }\deg_{C'}(\iota^*\cO(1)) = p\deg(\omega_{C'}(D_{C'}))}$$
where $D_{C'}$ is the divisor given by the nodal points connecting $C'$ to other irreducible components of $C$. Set $$V_\iota \coloneq V'_\iota \cup  \{x\mid x \normalfont\text{ a special point, }x \notin \union{C'\in V'_\iota}{C'}\}.$$ Let $\cM_\Gamma \sub \cM$ be the subset of elements for which
\begin{enumerate}[\normalfont i)]
	\item there exists a bijection $\phi_\iota\cl V_\iota \to V(\Gamma)\sm V^b (\Gamma)$, which maps a marked (respectively nodal) point to a marked (respectively nodal) point and preserves the arithmetic genus and the markings (defined in the obvious way on $V_\iota$);
	\item an irreducible component $C'$ of $C$ not contained in $\bigcup V'_\iota$ has genus $0$, corresponds to a unique edge $e\in E(\Gamma)$, and satisfies
	$$\deg(\iota^*\cO(1)) = m_e: = p(n_e-2) 3pd_e \,\Omega(S_\Gamma(e)).$$
\end{enumerate} 
where $\Omega$ was defined in Definition \ref{aux-choices-defined}\eqref{polarization-on-target}. By construction, $\cT_\Gamma$ is mapped to $\cM_\Gamma$ by $\pi \cl \cT\to \cM$.
% $\cM_\Gamma$ contains the image of $\cT_\Gamma$ under the forgetful map $\pi\cl\cT\to \cM$.\par 
Let $n_e\in \{1,2\}$ be the number of special points on the corresponding irreducible component and set $m_v \coloneq p(2g_v-2+|n_v| +|E_v|)$ for $v \in V^s(\Gamma)$. Denote by
\begin{equation}\label{eq:new-base}\Mbar_\Gamma(\bP^N) \sub \p{v\in V^s(\Gamma)}{\Mbar_{g(v),n_v\cup E_v}(\bP^N,m_v)} \times \p{e\in E(\Gamma)}{\Mbar_{0,n_e}(\bP^N,m_e)} \end{equation}
 the domain of the clutching map $\varphi_\Gamma\cl \Mbar_\Gamma(\bP^N)\to \Mbar_{g,n}(\bP^N,X)$ and define $\wh\cM_\Gamma \coloneq \varphi_\Gamma\inv(\cM_\Gamma)$. By Lemma \ref{smooth-base-splitting}, applied iteratively, $\wh\cM_\Gamma$ is a smooth manifold of the expected dimension.\par
\subsection*{Kuranishi chart:} 
 Let $\wh\cT\sub \wh\cM_\Gamma\times_\cM \cT$ be the subspace consisting of maps $[\zeta,(u,C,x_*,\iota,\alpha,\eta)]$ with $u_*[C_e] = d_e[S_\Gamma(e)]$. Denote by $\phi_\Gamma\cl \wh\cT\to \cT$ the pullback of $\varphi_\Gamma$, and set $\wh\cE \coloneq \phi_\Gamma^*\cE$, and $\wh\fs \coloneq \phi_\Gamma^*\fs$.
This defines an oriented rel--$C^\infty$ global Kuranishi chart $\wh\cK \coloneq (G,\wh\cT/\wh\cM_\Gamma,\wh\cE,\wh\fs)$ for 
\begin{equation}\label{eq:domain-restricted}
	\Mbar_\Gamma^{\,J}(X,A)\coloneq \prod_{v\in V^s(\Gamma)}{\Mbar_{g_v,n_v\cup E_v}^{\,J}(X,0)}\times_X \p{e\in E(\Gamma)}{\Mbar_{0,n_e}^{\,J}(X,d_e[S_e])},
\end{equation}
where we implicitly take the fibre product over $X$ in the big products. Clearly, $\wh\cK$ admits a compatible $\bT$-action and its virtual normal bundle in $\cK$ is the pullback of $N_{\wh\cM_\Gamma/\cM}$. 

Let $\wh\cM(\Gamma,J)$ be the component of the fixed-point locus of $\Mbar_\Gamma^{\,J}(X,A)$ that is mapped to $\Mbar(\Gamma,J)$ by the clutching map. Then
$$\wh\cM(\Gamma,J) \cong \p{v\in V^s(\Gamma)}{\Mbar_{g(v),n_v\cup E_v}} = \Mbar_\Gamma,$$
since edge components are completely determined by their degree, which is encoded in $\Gamma$. Recall that $\Aut(\Gamma)$ is the group of automorphisms of $\Gamma$ that preserve the decorations.

\begin{lemma} $\Aut(\Gamma)$ acts on $\wh\cK$ and $\phi_\Gamma\cl \wh\cT\to \cT$ is the quotient map onto its image. In particular, $\wh\cM(\Gamma,J)/\Aut(\Gamma) \cong \Mbar(\Gamma,J)$.
\end{lemma}

\begin{proof}  The degree of each factor of the right-hand side of \eqref{eq:new-base} is determined by the decorations of $\Gamma$, so $\Aut(\Gamma)$ acts by biholomorphisms on $\wh\cM_\Gamma$, leaving $\varphi_\Gamma$ invariant. Since an element of $\varphi_\Gamma([\iota,C,x_*])$ corresponds to an identification of the (partially contracted) dual graph of $(C,x_*)$ with $\Gamma$, the action is transitive, so $\varphi_\Gamma$ is the quotient map. The claim for $\phi_\Gamma$ now follows since it is a pullback of $\varphi_\Gamma$. As $\phi_\Gamma^*\fs = \wh\fs$ and $\phi_\Gamma$ is equivariant, we may conclude.\end{proof}

Combining this observation with Lemma \ref{vfc-of-cover} we obtain
$$\vfc{\Mbar(\Gamma,J)} = \frac{1}{|\Aut(\Gamma)|}\vfc{\wh\cM(\Gamma,J)},$$
so we may work with $\wh\cM(\Gamma,J)$ for the remainder of the proof.\par

We first determine the Euler classes of the virtual normal bundle in $\wh\cK$ and compute the Euler class of the normal bundle of $\wh\cM_\Gamma$ in $\cM$ at the end. To relate the Euler classes to more explicit data depending on $\Gamma$ and $X$, we will use as inspiration the exact sequence 
$$0 \ra H^0(C,u^*T_X) \to T_{\cT,y}\xra{d^v\fs} \cE_y \to H^1(C,u^*T_X)\to 0$$
for $y = (u,C,x_*,\iota,\alpha,0)$ in $\cT^{qf}$, respectively its lift to $\wh\cT$. While the kernel and cokernel of $D\hpd_J (u)$ do not form vector bundles over $\cT$, respectively $\wh\cT$, in general, we will determine explicit descriptions for the moving part of both kernel and cokernel over the fixed point locus of $\{\eta = 0\}$ in $\wh\cT$, see Lemma \ref{lem:euler-class-obstruction-bundle}. For this, it is easier to work on the level of $\wh\cT$ as we can use well-defined partial normalisation sequences.\par 
Write
 $$\wh\cE = V \oplus \cD\oplus \cL,$$ 
using the decomposition of the obstruction bundle in \eqref{obstruction-bundle-fibre-def}. Note that $V$ is the trivial bundle with fibre $\fs\fu(N+1)$, but its fixed and moving part will no longer be trivial bundles.
Furthermore, let 
$$Z \sub  {\wh\fs}\inv(0)\cap\wh\cT^{qf}$$ 
be the preimage of $\wh\cM(\Gamma,J)$ under the quotient map and let $\wh\cT_Z\sub \wh\cT^{qf}$ be the component containing $Z$.

\begin{remark}\label{rem:projection-from-identity-component} The projection $p \cl (\bT\times G)_y \to \bT$ is surjective for any $y \in \cT^{qf}$ since $\cT^{qf}$ is the preimage of the fixed-point locus of the quotient. As $p$ is a morphism of Lie groups and $\bT$ is connected, the restriction $(\bT\times G)_{y,0} \to \bT$ to the identity component is surjective as well.
\end{remark}

\begin{lemma}\label{lem:fixed-point-locus-regular}  $\wh \cM(\Gamma,J)$ is regular. In particular,
	$$\vfc{\wh\cM(\Gamma,J)} = \fcl{\wh\cM(\Gamma,J)} = \fcl{\Mbar_\Gamma}.$$
\end{lemma}

\begin{proof} Given $y  = [(\zeta,(u,C,x_*,\iota,0,0))]\in Z$ be as above, we consider the partial normalisation sequence 
	\begin{equation}\label{eq:normalisation-sequence} 0 \to \cO_{C}\to \bigoplus\limits_{v \in V^s(\Gamma)}\cO_{C_v}\oplus  \bigoplus\limits_{e \in E(\Gamma)}\cO_{C_e}\to \bigoplus\limits_{v \in V^n(\Gamma)}\cO_{x_v}\oplus  \bigoplus\limits_{(e,v) \in F^s(\Gamma)}\cO_{x_{e,v}}\to 0,\end{equation}
	whose associated long exact sequence shows that 
	$$H^1(C,\cO_C) \cong \bigoplus\limits_{v \in V^s(\Gamma)}H^1(C_v,\cO_{C_v}).$$ 
	As any $C_v$ is contracted to a fixed point, we deduce
	\begin{equation}\label{eq:fixed-hodge-bundle}\cL^f = \cL. \end{equation} 
	To see the claim, twist \eqref{eq:normalisation-sequence} by $u^*T_X$ and take the fixed part of its long exact sequence to obtain 
	$$H^1(C,u^*T_X)^f \cong \bigoplus\limits_{v \in V^s(\Gamma)}H^1(C_v,T_{p_\Gamma(v)}X)^f\oplus  \bigoplus\limits_{(e) \in E(\Gamma)}H^1(C_e,u^*T_X)^f.$$
	Since the fixed points of $X$ are isolated, the first direct sum vanishes. The second direct sum vanishes because $S_\Gamma(e)\cong \bP^1$ is convex by \cite[Proposition 7.4.3]{MS12}, so $$H^1(C_e,u^*T_X)^f\sub H^1(C_e,u^*T_{S_\Gamma(e)}) = 0.$$
 	Taking the fixed part of the exact sequence
	\begin{equation}\label{eq:deformation-sequence}
		0\to H^0(C,u^*T_X) \to T_{\wh\cT/\wh\cM_\Gamma,y} \xra{d\fs'_\Gamma(y)} \cD_y\to H^1(C,u^*T_X)\to 0,
	\end{equation}
	the above shows that the fixed part of the vertical derivative $d\fs'_\Gamma(y)$ is surjective, which is exactly $d\fs_\Gamma^f(y)$.
	Finally, we note that by the proof of \cite[Lemma~6.1]{HS22}, variations of the framing suffice to achieve transversality for the projection of the obstruction section to $\cL$ and that the section to $V^f$ is transverse to $V^f$ since the $G$-action on $\wh\cT$ extends to a $\PGL_\bC(N+1)$-action on $Z$.
\end{proof}

\begin{lemma} The dimensions of $H^1(C,u^*T_X)$ and $H^0(C,u^*T_X)^m$ are independent of the point $[u,C,x_*,\iota,0,0]\in Z$. 
\end{lemma}

\begin{proof} 
	If $v \in V^s(\Gamma)$, then 
	$$H^1(C_v,u^*T_X) = H^1(C_v,\cO_{C_v})\otimes T_{p_\Gamma(v)}X,$$
	whose rank only depends on $v$. If $e\in E(\Gamma)$ with $S = S_\Gamma(e)$, then 
	$$H^1(C_e,(u|_{C_e})^*T_X) = H^1(\bP^1,(u|_{C_e})^*TS)\oplus  H^1(\bP^1,(u|_{C_e})^*N_{S/X})$$
	where $N_{S/X}$ is the normal bundle of $S$ inside $X$. Then 
	$$(u|_{C_e})^*N_{S/X} = \cO(\ell_1)\oplus \dots \oplus \cO(\ell_{\dim_\bC(X)-1}),$$ 
	where $\ell_i$ is determined by $d_e$ and $S$, both encoded in $\Gamma$. Similarly, the dimension of $$\bigoplus\limits_{v\in V^s(\Gamma)}H^0(C_v,u^*T_X) \oplus \bigoplus\limits_{e\in E(\Gamma)}H^0(C_e,u^*T_X)$$ 
	does not depend on $(u,C)$. Twist \eqref{eq:normalisation-sequence} by $u^*T_X$ to obtain the long exact sequence
	\begin{multline}\label{eq:twisted-normalisation-sequence}
		0 \to H^0(C,u^*T_X) \to \bigoplus\limits_{v\in V^s(\Gamma)}H^0(C_v,u^*T_X) \oplus \bigoplus\limits_{e\in E(\Gamma)}H^0(C_e,u^*T_X)\\\to \bigoplus\limits_{v\in V^n(\Gamma)}T_{p_\Gamma(v)}X \oplus \bigoplus\limits_{(e,v)\in F^s(\Gamma)} T_{p_\Gamma(v)}X \to H^1(C,u^*T_X) \\\to \bigoplus\limits_{v\in V^s(\Gamma)}H^1(C_v,u^*T_X) \oplus \bigoplus\limits_{e\in E(\Gamma)}H^1(C_e,u^*T_X)\to 0
	\end{multline}
	as in \cite[\textsection5.3.4]{Liu13}. As the map from $H^0$ to the tangent spaces of $X$ at the respective fixed points is surjective, we obtain the claim for $H^1$. The claim for $H^0(C,u^*T_X)^m$ now follows by taking the moving part of the first short exact sequence derived from \eqref{eq:twisted-normalisation-sequence}.
\end{proof}

For each $e \in E$, fix an equivariant biholomorphism $\rho_e \cl \bP^1 \to S_\Gamma(e)$ and set $W_e \coloneq \rho_e^*T_X$. Then $Z$ admits two vector bundles, the first one trivial:
$$\cB_\Gamma = \ker\lbr{\bigoplus\limits_{v\in V^s(\Gamma)}T_{p_\Gamma(v)}X\oplus \bigoplus\limits_{e \in E(\Gamma)}H^0(\bP^1,u_{d_e}^*W_e)^m\twoheadrightarrow \bigoplus\limits_{(e,v)\in V^s(\Gamma)}T_{p_\Gamma(v)}X\oplus \bigoplus\limits_{v\in V^n(\Gamma)}T_{p_\Gamma(v)}},$$ 
where $u_d \cl \bP^1\to \bP^1$ is given by $u_d(z) = z^d$. The second one is
$$\text{Ob}_\Gamma \coloneq \bigoplus\limits_{v\in V^s(\Gamma)}\cL^v\otimes T_{p_\Gamma(v)}X\oplus \bigoplus\limits_{e\in E(\Gamma)}H^1(\bP^1,u_{d_e}^*W_e)$$ 
pulled back from $\wh\cM_\Gamma$. Here $\cM_v$ is the factor in $\wh\cM_\Gamma$ indexed by $v\in V^s(\Gamma)$ and $\cL^v \to \cM_v$ is the dual of the Hodge bundle, i.e., has fibre $\cL^v_{[\iota,C]} = H^1(C,\cO_C)$.\par 
To see how these bundles are related to the normal bundle $\cN$ of $\wh\cT_Z$ in $\wh\cT$, decompose it as
$$\cN = \cN^v\oplus \cN^b,$$ 
i.e., into the moving part of the vertical tangent bundle of $\wh\cT$ restricted to $\wh\cT_Z$ and the moving part of the pullback of $T_{\wh\cM_\Gamma}$ restricted to $\wh\cT^{qf}$.

\begin{lemma}\label{} The identifications $(\cB_\Gamma)_y \cong H^0(C,u^*T_X)^m$ and $(\normalfont\text{Ob}_\Gamma)_y \cong H^1(C,u^*T_X)^m$ fit into an exact sequence
	\begin{equation}\label{eq:important-ses}0 \to \cB_\Gamma\to \cN^v\xra{d\fs^{'m}}\cD^m|_Z \to \normalfont\text{Ob}_\Gamma\to 0 \end{equation}
	of $(G\times \bT)$-vector bundles over $Z$.
\end{lemma}

\begin{proof} As $\cB_\Gamma$ is the kernel of the linearisation of the restriction cutting out $\wh\cT$ from a product of thickenings, it suffices to show the inclusion $Z\times \cB_\Gamma\hkra T_{\wh\cT/\cM}|_{Z}$ of $(G\times\bT)$-vector bundles where we have a single stable vertex, respectively a single edge. In this case, it is straightforward. As $G\times\bT$ acts nontrivially on $\cB_\Gamma$, it follows that $Z \times\cB_\Gamma\sub N^v_{\wh\cT_Z/\cT}|_Z$. 
	%Clearly, $d\fs^{'m}$ vanishes on the image of $\cB_\Gamma$. 
	The vertical tangent space of $\wh\cT$ at $y = [u,C,\iota,x_*,\alpha,\eta]$ satisfies
	$$T_{\wh\cT/\wh\cM,y} \sub C^\infty(C,u^*T_X)\oplus E_{(\iota,u)}$$
	and the vertical derivative $T_{\wh\cT/\wh\cM,y}\to \cD_y$ of $\fs'$ is the projection onto the second summand. Hence $(\xi,\eta)\in \ker(d\fs^{'m}(y))$ if and only if $\xi \in \ker(D\hpd_J(u))$ and $\xi$ is not fixed by the torus action. Thus, $\cB_\Gamma$ is mapped bijectively onto the kernel of $d\fs^{'m}$. To construct the map $\cD^m \to \text{Ob}_\Gamma$, we can again reduce to the case of having just one stable vertex or one edge. In the case where $\Gamma$ consists of a single stable vertex and no edges, we can apply the argument in Proposition \ref{mapping-to-point}. In the case of a single edge (and no stable vertices), $\Mbar(\Gamma,J)$ is just a single point, and $Z$ is thus a single $G$-orbit. Letting $\iota \cl \bP^1 \to \bP^N$ be given by $\iota([z_0:z_1]) = [z_0^{m_e}:z_1^{m_e}:0:\dots:0]$, we obtain, by the choice of $\rho_e$, a basepoint $y_0$ in $Z$. The composite
	$$\cD^m_{y_0} \to  C^\infty(\bP^1,\cc{\Hom}_\bC(\iota^*T_{\bP^N},u_{d_e}^*\rho_e^*T_X))\xra{d\iota^*} C^\infty(\bP^1,\cc{\Hom}_\bC(T_{\bP^1},u_{d_e}^*\rho_e^*T_X)) \to H^1(\bP^1,u_{d_e}^*W_e)$$
	can be extended via the $G$-action to a map $\cD^m \to Z\times H^1(\bP^1,u_{d_e}^*W_e)$ of vector bundles. Exactness of \eqref{eq:important-ses} at $\text{Ob}_\Gamma$ can now be checked pointwise. The argument is analogous to the proof of Proposition \ref{mapping-to-point}.
\end{proof}

Let $\cB_\Gamma$ and $\normalfont\text{Ob}_\Gamma$ also denote the orbibundles induced by $\cB_\Gamma$ and $\normalfont\text{Ob}_\Gamma$ on $\wh\cM(\Gamma,J)$.

\begin{lemma}\label{lem:euler-class-obstruction-bundle}
	 The Euler class $e_\bT(\cB_\Gamma)$ is invertible in $\chml^*(\wh\cM(\Gamma,J);\bQ)\otimes_R Q$ and
	\begin{equation*}\label{eq:euler-class-obstruction}
		\frac{e_{\bT}(\normalfont\text{Ob}_\Gamma)}{e_\bT(\cB_\Gamma)} =  \p{v \in V^n(\Gamma)}{w(p_\Gamma(v))} \p{(e,v) \in F^s(\Gamma)}{w(p_\Gamma(v))}\p{v \in V^s(\Gamma)}{\textbf{h}(p_\Gamma(v),g_v)} \p{e \in E(\Gamma)}{\textbf{h}(S_\Gamma(e),d_e)},
	\end{equation*}
	where $\textbf{h}(p,g)$ and $\textbf{h}(S,d)$ are given by \eqref{eq:vertex-weight} and \eqref{eq:edge-weight} respectively.
\end{lemma}

\begin{proof} By definition,
	$$ e_{\bT}(\normalfont\text{Ob}_\Gamma) = \p{v \in V^s(\Gamma)}{e_{\bT}(\bE_{g_v}\dul\otimes T_{p_\Gamma(v)}X)} \p{e \in E(\Gamma)}{e_{\bT}(H^1(\bP^1,u_{d_e}^*W_e)^m)}.$$
	As $\cB_\Gamma$, respectively its fibre, fits into the short exact sequence 
	$$0 \to\cB_\Gamma\to \bigoplus\limits_{v\in V^s(\Gamma)}T_{p_\Gamma(v)}X\oplus \bigoplus\limits_{e \in E(\Gamma)}H^0(\bP^1,u_{d_e}^*W_e)^m\to \bigoplus\limits_{(e,v)\in V^s(\Gamma)}T_{p_\Gamma(v)}X\oplus \bigoplus\limits_{v\in V^n(\Gamma)}T_{p_\Gamma(v)} \to 0,$$
	we have
	$$ e_{\bT}(\cB_\Gamma) = \p{v \in V^n(\Gamma)}{\frac{1}{w(p_\Gamma(v))}}\p{(e,v)\in F^s(\Gamma)}{\frac{1}{w(p_\Gamma(v))}}\p{v \in V^s(\Gamma)}{w(p_\Gamma(v))} \p{e \in E(\Gamma)}{e_{\bT}(H^0(\bP^1,u_{d_e}^*W_e)^m)};$$
	whose last product is nonzero since we take the non-fixed part of the $\bT$-representation.
	For $e \in E(\Gamma)$, we observe that
	$$e_\bT(H^i(\bP^1,u_{d_e}^*W_e)^m) = e_\bT(H^i(\bP^1,u_{d_e}^*T_{\bP^1})^m)\cdot\prod_{j =1}^{m-1} e_\bT(H^i(\bP^1,u_{d_e}^*L_j)^m),$$
	where $L_j \to \bP^1$ is a complex line bundle so that $\rho_e^*N_{S_\Gamma(e)/X} = L_1 \oplus \dots \oplus L_{m-1}$.\par The main ingredient of \cite[Lemma 66]{Liu13} is Example 19 op. cit., which holds for any torus action on a line bundle over $\bP^1$. Hence, \cite[Lemma 66]{Liu13} applies directly to our situation. The claim follows.
\end{proof}

Given $y = [\zeta,u,C,x_*,\iota,\alpha,\eta]$ in $Z$, we turn to a second exact sequence, induced by the Euler exact sequence:
\begin{equation*}\label{eq:derived-from-euler-sequence} 0 \to H^0(C,\cO_C)\to  H^0(C,\iota^*\cO(1)^{\oplus N+1})\to H^0(C,\iota^*T_{\bP^N})\to H^1(C,\cO_C)\to 0 .\end{equation*}
This sequence is the fibre over $\bar{y}$ of an exact sequence of vector bundles 
\begin{equation}\label{eq:from-euler-sequence}0 \to  R^0\pi_*\cO_{\cC_\Gamma}\to \fg\fl(N+1)\to R^0\pi_*\eva^*T_{\bP^N}\to R^1\pi_*\cO_{\cC_\Gamma}\to 0, \end{equation}
 on $\cM_\Gamma$, pulled back to $\wh\cT$, where $\pi \cl \cC_\Gamma\to \cM_\Gamma$ is the universal curve.\par 
 To see this, note that the bundle $R^0\pi_*\eva^*\cO_{\bP^N}(1)^{\oplus N+1}$ admits a global trivialisation over $\wh\cM_\Gamma$ given by the global sections $\tilde{\sigma}_i([\iota,C,x_*]) = \iota^*\sigma_i$ for $0 \leq i \leq N$, where $\sigma_i$ denotes the $i^{th}$-standard section of $\cO_{\bP^N}(1)$. 

\begin{remark}\label{rem:lie-algebras} Since $\iota \cl C\hkra \bP^N$ is determined by $\iota^*\sigma_0,\dots,\iota^*\sigma_N$, we obtain an identification of the stabiliser of $[\iota,C,x_*]$ under the $\PGL_\bC(N+1)$-action with the automorphism group of the domain $(C,x_*)$. In particular, the kernel of $H^0(C,\iota^*\cO(1)^{\oplus N+1})\to T_{\wh\cM_\Gamma,[\iota,C,x_*]}$ is the Lie algebra of $\Aut(C,x_*)$.\end{remark}

\begin{lemma}\label{lem:euler-class-automorphisms} We have 
\begin{equation}\label{eq:euler-class-automorphisms}\frac{e_\bT(V^m)}{e_\bT(N^b_{\wh\cT_Z/\wh\cT}/G)} = \p{\substack{v\in V^b(E)\\ E_v = \{e\}}}{w(e,v)}. \end{equation}
\end{lemma}

\begin{proof} The induced map $\normalfont\text{st}_\Gamma\cl \wh\cM_\Gamma\to \Mbar_\Gamma$ is a submersion, so we can write 
	$$T_{\wh\cM_\Gamma} = \normalfont\text{st}_\Gamma^*T_{\Mbar_\Gamma} \oplus T^{\normalfont\text{rel}}_{\wh\cM_\Gamma}.$$ 
	As $\wh\cM_\Gamma$ is unobstructed, the canonical map $R^0\pi_*\eva^*T_{\bP^N}\to T^{\normalfont\text{rel}}_{\wh\cM_\Gamma}$ is a surjection. By Remark \ref{rem:lie-algebras}, its kernel is the Lie algebra of the automorphisms of $(C,x_*)$. If $(C,x_*)$ is the domain of $y \in Z$, then
	$$\normalfont\text{Lie}(\Aut(C,x_*)) =\normalfont\text{Lie}\lbr{\dS{e \in E(\Gamma)}{\Aut(C_e,\textbf{x}_e)}}$$
	since all other components are stable. Here $\textbf{x}_e = \{x_{e,v}\mid v \notin V^b(\Gamma)\}$. The Lie algebra of the automorphism group of each edge component is given by $H^0(C_e,T_{C_e}(-\textbf{x}_e))$ with 
	$$H^0(C_e,T_{C_e}(-\textbf{x}_e))^f = H^0(C_e,T_{C_e}(-\textbf{x}'_e)),$$
	where $\textbf{x}'_e \coloneq \{x_{e,v}\mid (e,v)\in F(\Gamma)\}$, and 
	$$H^0(C_e,T_{C_e}(-\textbf{x}_e))^m = \begin{cases} T_{x_{e,v}}C_e\quad & \textbf{x}_e \neq \textbf{x}'_e\\
		0 \quad & \textbf{x}_e = \textbf{x}'_e\end{cases}$$
	Pull \eqref{eq:from-euler-sequence} back to $Z$ and take the natural complement of the image $\bC\to \fg\fl(N+1)$ to obtain 
	$$0 \to{\fs\fl(N+1)}\to q^*R^0\pi_*\eva^*T_{\bP^N}\to \cL\to 0$$
	By \eqref{eq:fixed-hodge-bundle}, the moving part of this sequence is 
	$0\to {\fs\fl(N+1)}^m\to (q^*R^0\pi_*\eva^*T_{\bP^N})^m\to 0.$
	
	Now, take the complement to the tangent space of the $G$-action on $\wh\cM_\Gamma$, respectively, of the adjoint action on $\fs\fl(N+1)$, to obtain the exact sequence 
	\begin{equation}\label{eq:moving-automorphisms}0 \to \bigoplus\limits_{e\in E(\Gamma)}H^0(C_e,T_{C_e}(-\textbf{x}_e))^m\to  V^m \to N^b_{\wh\cT_Z/\wh\cT}|_Z\to 0. \end{equation}
	The first bundle of this sequence is a trivial vector bundle whose equivariant Euler class is the right-hand side of \eqref{eq:euler-class-automorphisms}. As $e_\bT(N^v_{\wh\cT_Z/\cT}/G)\cdot e_\bT(N^b_{\wh\cT_Z/\cT}/G)$ is invertible in the localised ring $\chml^*(\wh\cM(\Gamma,J);\bQ)\otimes_R Q$, so is $e_\bT(N^b_{\wh\cT_Z/\cT}/G)$. Thus, the claim follows by taking the equivariant Euler class of \eqref{eq:moving-automorphisms} and inverting $e_\bT(N^b_{\wh\cT_Z/\wh\cT})$.
\end{proof}

\begin{remark} Equation \eqref{eq:moving-automorphisms} is the linearisation of the fact that $[u\psi,C,x_*]$ is no longer fixed by the torus action for an automorphism $\psi$ of $(C,x_*)$ exactly if $\psi$ moves at least one branch point on an irreducible component $C_e$ corresponding to an edge.\\
\end{remark}

It remains to compute the Poincar\'e dual of the immersion $\wh\cT_\Gamma\to \cT$. 
\begin{lemma}\label{} We have
	\begin{equation}\label{eq:euler-class-normal-bundle-base}
		e(N_{\wh\cT_\Gamma/\cT}) = \p{\substack{v\in V^n(\Gamma)\\E_v = \{e,e'\}}}{(w(e,v)+w(e,v'))}\p{(e,v)\in F^s(\Gamma)}{(w(e,v)-\psi_{(v,e)})},
	\end{equation}
where $\psi_{(v,e)}$ is the $\psi$-class on $\Mbar^*_{g_v,n_v\cup E_v}(\bP^N,m_v)$ at the marked point $(v,e)$.
\end{lemma}

\begin{proof} As $\wh\cT_\Gamma$ is the fibre product $\cT\times_\cM \wh\cM_\Gamma$, it suffices to compute $e(N_{\wh\cM_\Gamma/\cM})$. To this end, let $y = (([\iota_v,C_v,\textbf{x}_v])_{v\in V^s(\Gamma)},([\iota_e,C_e,\textbf{x}_e])_{e\in E(\Gamma)})$ be an element of $\wh\cM_\Gamma$. The fibre of $N_{\wh\cM_\Gamma/\cM}$ at $y$ is given by
	$$\bigoplus\limits_{\substack{v\in V^n(\Gamma)\\E_v = \{e,e'\}}}T_{x_{e',v}}C_e\otimes T_{x_{e,v}}C_{e'} \oplus \bigoplus\limits_{(e,v) \in F^s(\Gamma)}{T_{x_{e,v}}C_v\otimes T_{x_{e,v}}C_e}.$$
	where $F^s(\Gamma) = \{(e,v) \in E(\Gamma)\times V(\Gamma)\mid v\in V^s(\Gamma)\}$ is the set of flags associated to the stable vertices of $\Gamma$. Hence,
	\begin{equation*}
		e(N_{\wh\cM_\Gamma/\cM}) = \p{\substack{v\in V^n(\Gamma)\\E_v = \{e,e'\}}}{(w(e,v)+w(e,v'))}\p{(e,v)\in F^s(\Gamma)}{(w(e,v)-\psi_{(v,e)})},
	\end{equation*}
as claimed.
\end{proof}

 In summary, we obtain that 
\begin{multline*}
	\frac{e_\bT(\cE^m_\Gamma)}{e_\bT(N_{\wh\cT/\cT})} =\p{\substack{v\in V^n(\Gamma)\\E_v = \{e,e'\}}}{\frac{w(p_\Gamma(v))}{w(e,v)+w(e,v')}}\p{(e,v)\in F^s(\Gamma)}{\frac{w(p_\Gamma(v))}{w(e,v)-\psi_{(v,e)}}}\\ \cdot \p{\substack{v\in V^b(\Gamma)\\ (e,v)\in F(\Gamma)}}{w(e,v)}\p{v \in V^s(\Gamma)}{\textbf{h}(p_\Gamma(v),g_v)} \p{e \in E(\Gamma)}{\textbf{h}(S_\Gamma(e),d_e)}
\end{multline*}
which is exactly $\textbf{w}(\Gamma)$.
Therefore, by Theorem \ref{Virtual Localisation}, the equivariant GW invariants of $(X,\omega,\mu)$ are given by \eqref{eq:formula-equivariant-gw}.

\section{Comparison with the Gromov--Witten invariants of Ruan and Tian}\label{sec:compare-pseudocycles}

\subsection{Definition of GW invariants via pseudocycles}

We call a symplectic manifold $(X,\omega)$ \emph{semipositive} if for any $A\in \pi_2(X)$
\begin{equation}\label{semipos}\omega(A) > 0,\; c_1(A)\geq 3-n \quad \rimp\quad c_1(A) \geq 0.\end{equation}
In particular, any symplectic manifold of complex dimension at most $3$ is semipositive.
Let us recall the definition of GW invariants in \cite{RT97} for $(X,\omega)$ satisfying \eqref{semipos}.\par
Suppose first that $2g-2+ n > 0$. A \emph{good cover} $p_\mu \cl \Mbar_{g,n}^\mu \to \cc{M}_{g,n}$ of the (coarse) moduli space of stable curves is a finite cover such that $\Mbar_{g,n}^\mu$ admits a universal family that is a projective normal variety. Such good covers can be constructed using level-m structures; refer to \cite[Chapter XVI]{ACG-moduli} or \cite{Mu83} for more details. Let $\ff_\mu\cl \cc{\cU}^\mu_{g,n}\to \Mbar^\mu_{g,n}$ be the universal curve. Removing the bar means we only consider smooth curves. Fix a closed embedding $\phi \cl  \cc{\cU}^\mu_{g,n}\hkra \bP^k$.\par 

Given $J \in \cJ_\tau(X,\omega)$ our perturbations $\nu$ are sections of $\cc{\Hom}_\bC(p_1^*T_{\bP^k},p_2^*T_X)$ over $\bP^k\times X$. Let $\Mbar_{g,n}^\mu(A;J,\nu)$ be the space of equivalence classes of \emph{stable $(J,\nu)$-maps of type $(g,n)$} $(u,j,C,\fj,x_*)$ where 
\begin{enumerate}
    \item $(C,\fj,x_*)$ is of type $(g,n)$,
    \item $j \cl C\to \cc{\cU}_{g,n}^\mu$ is a holomorphic map onto a fibre of $\cc{\cU}_{g,n}^\mu$,
    \item $u \cl C\to X$ is a stable smooth map, representing $A$ and satisfying 
    $$\hpd_Ju = \nu(\phi j,u)\g d(\phi j).$$
\end{enumerate}
Here $(u,j,C,x_*)$ is \emph{equivalent} to $(u',j',C',x'_*)$ if there exists a biholomorphism $(C,x_*)\to (C',x'_*)$ pulling back $u'$ to $u$. We call a stable $(J,\nu)$-map \emph{simple} if 
\begin{enumerate}
		\item for each irreducible component $Z\sub C$ on which $u$ is nonconstant, $u|_Z$ is a simple map, i.e., does not factor through a branched holomorphic covering,
		\item $u(Z) \neq u(Z')$ for any two irreducible components $Z \neq Z' \subset C$ on which $u$ is nonconstant. 
	\end{enumerate}
 The space of simple $(J,\nu)$-maps is denoted by $\Mbar_{g,n}^{\mu,*}(A;J,\nu)$.
It admits a canonical forgetful map $\Mbar_{g,n}^\mu(A;J,\nu)\to \Mbar_{g,n}^\mu$ through which the stabilisation map $\Mbar_{g,n}^\mu(A;J,\nu)\to \Mbar_{g,n}$ factors. 

$\Mbar_{g,n}^\mu(A;J,\nu)$ can be stratified by the topological type of the domains together with the distribution of the homology class. These are most concisely described in terms of their \emph{dual graph}. Explicitly, to each stable $(J,\nu)$-map we can associate a \emph{marked graph} $\gamma$ consisting of a $n$-marked graph $G$ together with maps $\fd \cl V(G)\to H_2(X,\bZ)$ and $g\cl V(G)\to \bN_0$ so that 
$$\dim(H^1(G)) + \s{v\in V(G)}{g(v)} = g \qquad \qquad \s{v\in V(G)}{\fd(v)} = A$$
and for any $v\in V(G)$ the stability condition $2g(v) + |\{f \in \text{Fl}(G) : s(f) = v\}|\geq 3$ holds, where $\text{Fl}(G)$ is the set of flags of $G$. 
%See \cite[\textsection3.2]{Zi17} for a detailed exposition in the real setting.
We denote by $\Mbar_{\gamma}^\mu(A;J,\nu)$ the stratum of stable maps whose dual graph is given by $\gamma$ and by $\Mbar_{\gamma}^{\mu,*}(A;J,\nu)$ its intersection with the space of simple maps.\par 

By \cite[Proposition 2.3, Theorem 3.1]{RT97}, respectively \cite[Theorem 3.3]{Zi17} (whose arguments simplify to our setting), there exists a Baire subset $\cH$ of tuples $(J,\nu)$ so that for any $(J,\nu)\in \cH$
\begin{enumerate}
		\item $\Mbar_{\gamma}^{\mu,*}(A;J,\nu)$ is a smooth oriented manifold of dimension 
  $$2(1-g)\dim_\bC(X) + 2\inpr{c_1(T_X)}{A} + \dim_\bR(\Mbar_{\gamma}^\mu).$$
		\item The maps $\eva$ and $\normalfont\text{st}_\mu$ define a pseudocycle $\eva\times\normalfont\text{st}_\mu\cl \cM^{\mu,*}_{g,n}(A;J,\nu)\to X\times \Mbar_{g,n}^\mu$.
\end{enumerate}

We use the definition of a pseudocycle given in \cite{IP19b}; see also \cite[Chapter 6.1]{MS12} or \cite{Zi08}.

\begin{definition} A \emph{$d$-dimensional pseudocycle} $f \cl M\to N$ is a continuous map from a $d$-dimensional oriented manifold $M$ to a locally compact space $N$ so that $f(M)$ has compact closure and $$\Omega_f \coloneq \inter{\substack{K \sub M\\\text{compact}}}{\cc{f(M\sm K)}}$$ has Lebesgue covering dimension $\leq d-2$. 
\end{definition}

Given a $d$-dimensional pseudocycle $f\cl M \to N$, we can define a class $[f]\in H^{\text{BM}}_d(N,\bZ)$ as follows. Let $M\inn \coloneq f\inv(N\sm \Omega_f)$. Then $M\inn$ is an oriented $d$-dimensional manifold and the restriction $f\inn\cl M\inn \to N\sm \Omega_f$ is proper. We define $[f]\in H_d^{\text{BM}}(N,\bZ)$ to be the image of $[M\inn]$ under $$H_d^{\text{BM}}(M\inn,\bZ)\xra{f\inn_*}H^{\text{BM}}_d(N\sm \Omega_f,\bZ)\cong H^{\text{BM}}_d(N,\bZ).$$
Here the isomorphism is an immediate consequence of $\dim(\Omega_f) < d-1$ and the long exact sequence in Borel--Moore homology. Refer to  \cite[\textsection 3, \textsection A.3]{IP19b} for more details.

\begin{definition}\cite{RT97} The \emph{pseudocycle Gromov--Witten invariant} of $(X,\omega)$ is
    $$\sigma^A_{g,n} = \frac{1}{d_\mu}(\ide\times p_\mu)_*[\eva\times \text{st}_\mu]\in H_*(X^n\times \Mbar_{g,n};\bQ).$$
\end{definition}

%	Suppose we are in the unstable range, that is, $(g,n)\in \{(0,0),(0,1),(0,2),(1,0)\}$. If $A = 0$, then there is nothing to show by Lemma \ref{mapping-to-point}. Suppose thus $\omega(A) >0 0$ and let $\beta\in H^2(X;\bZ)$ be such that $\lspan{\beta,A} > 0$. Then define 
%$$\sigma^A_{1,0} = \frac{1}{\lspan{\beta,A}}\eva_1^*\beta\cap \sigma^A_{1,1},$$
%and 
%$$\sigma^A_{0,2} = \frac{1}{\lspan{\beta,A}}\eva_3^*\beta\cap \sigma^A_{0,3}\qquad \quad \sigma^A_{0,1} = \frac{1}{\lspan{\beta,A}}\eva_2^*\beta\cap \sigma^A_{0,2}\qquad \quad \sigma^A_{0,0} = \frac{1}{\lspan{\beta,A}}\eva_1^*\beta\cap \sigma^A_{0,1}.$$

\begin{theorem}\label{pseudocycle-comparison} Suppose $(X,\omega)$ is semipositive. Then the pseudocycle Gromov--Witten invariants agree with the invariants defined in \cite{HS22}. Explicitly,
	\begin{equation}\label{comparision-pseudocycles}\sigma^A_{g,n}= (\eva\times\normalfont\text{st})_*\vfc{\Mbar_{g,n}^{\,J}(X,A)}\end{equation}
  for any $g,n \geq 0$, where $\vfc{\cdot}$ is the virtual fundamental class constructed in \cite{HS22}. 
\end{theorem}

\begin{remark}[Unstable range] The GW invariants for $(g,n)$ in the unstable range are defined via the Divisor and Symmetry axioms, see \cite{MS12}. Since both the invariants defined via global Kuranishi charts as the ones of \cite{RT97} satisfy these axioms, we only have to prove Theorem \ref{pseudocycle-comparison} for the cases where $2g-2 + n> 0$.\end{remark}

\smallskip
\subsection{Proof of Theorem \ref{pseudocycle-comparison}}  Fix $(g,n)$ with $2g-2 + n >0$. The case of $(g,n)$ in the unstable range will follow from the cases in the stable range and the Divisor axiom. Let $p_\mu \cl \Mbar_{g,n}^\mu \to \Mbar_{g,n}$ be a good finite cover. Set $d_\mu \coloneq \deg(p_\mu)$ and define 
$$\cW_{g,n,A}^{\mu}\coloneq \set{(t,u,C,x_*,j)\;\bigg|\; \substack{ u \cl (C,x_*)\to X \text{ smooth stable of type }(g,n),\\ t \in [0,1],\;j \cl (C,x_*)\to \cc{\cU}_{g,n}^\mu,\;u_*[C] = A,\\  \hpd_J u = t\,(\phi j\times u)^*\nu}},$$
where $j\cl C\to \cc{\cU}^\mu_{g,n}$ is a contraction onto a fibre $F$ of $\cc{\cU}^\mu_{g,n}$. We have a canonical map $P \cl \cW_{g,n,A}^{\mu}\to [0,1]$. 
%The strategy of the proof is to construct a global Kuranishi chart with boundary on the naive ‘cobordism' from the moduli space of $J$-holomorphic stable maps to the moduli space of $(J,\nu)$-maps. On one side this recovers the virtual fundamental class of \cite{HS22}. On the other side, we show that the virtual fundamental class recovers the pseudocycle defined by evaluation and stabilisation map.\par

\begin{lemma}\label{}$\cW_{g,n,A}^{\mu}$ is compact and Hausdorff when endowed with the topology induced by Gromov convergence.
\end{lemma}

\begin{proof} Compactness follows from \cite[Proposition~3.1]{RT95}, while the Hausdorff property follows by the arguments of the proof of \cite[Theorem~5.5.3]{MS12}.\end{proof}

\begin{remark}\label{support-away-from-nodes} Denote by $\cc{\cU}_{g,n}^{\mu,\text{sing}}$ the nodes and marked points of the universal curve. In order to apply \cite{Swa21} later on, we restrict to perturbations $\nu$ that are supported away from $\phi(\cc{\cU}_{g,n}^{\mu,\text{sing}})$. By elliptic regularity, such perturbations $\nu$ suffice to achieve transversality.\end{remark}
We will construct a global Kuranishi chart with boundary for $\cW_{g,n,A}^{\mu}$, which restricts to a cover of the previously constructed global Kuranishi chart for $\Mbar_{g,n}^{\,J}(X,A)$ over $P\inv(\{0\})$ and to a global Kuranishi chart for $\Mbar^\mu_{g,n}(X,A;J,\nu)$ over $P\inv(\{1\})$. By Lemma \ref{vfc-of-cover} and Lemma \ref{cobordant-charts}, this will imply that 
$$d_\mu\,(\eva\times \text{st})_*\vfc{\Mbar_{g,n}^{\,J}(X,A)} = (\eva\times \text{st})_*\vfc{\Mbar_{g,n}^\mu(X,A;J,\nu)}$$
in $H_*(X^n\times \Mbar_{g,n};\bQ)$. Then we compare $(\eva\times \text{st})_*\vfc{\Mbar_{g,n}^\mu(X,A;J,\nu)}$ with $\sigma^A_{g,n}$ in order to conclude.\\

Fix an unobstructed auxiliary datum $(\conn^X,\cO_X(1),p,\cU,\lambda,k)$ where 
\begin{enumerate}
	\item $p\gg 0$ is so large that $\fL_u^{\otimes p}$ is very ample for any $(t,u)\in \cW^\mu_{g,n,A}$;
	\item $\cU$ is a good covering in the sense of \cite[Definition~3.10]{HS22} where we take the image of $\cW^\mu_{g,n,A}$ in the polyfold of smooth maps to $X$ instead of $\Mbar_{g,n}^{\,J}(X,A)$ in the third condition;
 \item the integer $k \gg 1$ will be determined later.
\end{enumerate}
Define $\widetilde{\cT}$ to be the set $\set{(t,u,C,x_*,j,\iota,\alpha,\eta)}/\sim$ so that
\begin{itemize}
	\item $u \cl (C,x_*)\to X$ is a smooth stable map of genus $g$ representing $A$.
	\item $j \cl (C,x_*)\to \cc{\cU}_{g,nb}$ is a holomorphic map onto a fibre $F$ of the universal curve so that the induced map $C\to F$ is a contraction of pointed nodal surfaces,
	\item $\iota \cl (C,x_*)\to \bP^N$ is an element of $\Mbar^*_{g,n}(\bP^N,m)$,
	\item $\alpha\in H^1(C,\cO_C)$ satisfies $[\iota^*\cO_{\bP^N}(1)] = p\cdot [\fL_u] + \alpha$ in $\Pic(C)$,
	\item $\eta \in E_{(\iota,u)}\coloneq H^0(C,\iota^*{T_{\bP^N}^*}^{0,1}\otimes u^*T_X\otimes\iota^*\cO(k))\otimes \cc{H^0(\bP^N,\cO(k))}$ is such that 
	\begin{equation}\label{ps-pert-cr}\delbar_J\tilde{u} + \lspan{\eta}\g d\tilde{\iota} - t\,\nu(\phi \tilde{j},\tilde{u}) = 0 \end{equation}
	on the normalisation $\tilde{C}$ of $C$, 
\end{itemize}
and $\sim$ denotes the equivalence given by reparametrisations of the domain. Denote by $P \cl \tilde{\cT}\to [0,1]$ the obvious projection as well, and define $\widetilde{\cT}_t \coloneq P\inv(\{t\})$. 

Set $\widetilde{\cM} \coloneq \Mbar_{g,n}^\mu \times_{\cc{M}_{g,n}}\Mbar_{g,n}^*(\bP^N,m)$ and let $\pi \cl \widetilde{\cT}\to \widetilde{\cM}$ be the forgetful map. Define $\widetilde{\cE}\to \widetilde{\cT}$ by letting its fibre over $y = (t,u,C,j,\iota,\alpha,\eta)$ be
$$\widetilde{\cE}_y = \fs\fu(N+1)\oplus H^1(C,\cO_C) \oplus E_{(\iota,u)},$$
while the obstruction section $\widetilde{\fs}$ is given by $\widetilde{\fs}(y) = (i\log(\lambda(u,\iota)),\alpha,\eta)$. 
Let $G \coloneq \PU(N+1)$ act via post-composition on the framings and the perturbation terms $\eta$. For $i \in \{0,1\}$, denote 
$$\widetilde{\cK}_i \coloneq (G,\widetilde{\cT}_i,\widetilde{\cE}_i|_{\widetilde{\cT}_i},\widetilde{\fs}|_{\widetilde{\cT}_i}).$$

\begin{lemma}\label{} We can choose $k$ sufficiently large so that the linearisation of \eqref{ps-pert-cr} restricted to $C^\infty(C,u^*T_X)_{\{x_i\}} \oplus E_{(\iota,u)}$ is surjective for any element in $\widetilde{\fs}\inv(0)$.\footnote{ The subscript denotes the subspaces of vector fields which vanish at the marked points.}
\end{lemma}

\begin{proof} This follows from a straightforward adaptation of the proof of \cite[Lemma~ 4.19]{HS22}. The proof of said lemma, respectively its predecessor, uses the linear gluing analysis of \cite[Appendix~B]{P16}, which considers the presence of perturbations such as the one we consider here.
\end{proof}

\begin{proposition}\label{} For $k\gg 0$, $\widetilde{\cT}^{\normalfont\text{reg}}$ is naturally a rel--$C^\infty$ manifold over $[0,1]\times \widetilde{\cM}$ and the structure map is a topological submersion. The restriction $\widetilde{\cE}^{\reg} \coloneq \widetilde{\cE}|_{\widetilde{\cT}^{\reg}}$ is a rel--$C^\infty$ vector bundle, and the restriction of $\widetilde{\fs}$ is of class rel--$C^\infty$. Moreover, $\eva \cl \widetilde{\cT}\to X^n$ is a rel--$C^\infty$ submersion.
\end{proposition}

\begin{proof} Forgetting the $\alpha$-parameter and using Gromov's shearing trick, we can consider $\widetilde{\cT}$ as a subset of the moduli space of embedded regular perturbed holomorphic maps to the total space of a vector bundle $E\to \bP^N\times X$. Fixing a splitting $T_E = \pi^*T_{\bP^N\times X}\oplus \pi^*E$, we define the family of almost complex structures on $E$ by
$$\tilde{J}^t_{e}(\hat{x},v,e') = (J_0\hat{x}, Jv + \lspan{e}(\hat{x}),J^E e')$$
for $(\hat{x},v)\in T_{\bP^N\times X,\pi_E(e)}$ and $e'\in E_e$.\par
By Remark \ref{support-away-from-nodes}, we can use \cite{Swa21} as in \cite[Proposition~5.4]{HS22} to deduce the relative smoothness of $\widetilde{\cT}$ over $[0,1]\times \widetilde{\cM}$.
% The structural map is a submersion since we obtain transversality without variation of the domain or the $t$-parameter. 
The other claims follow from the same reasoning as in \cite[\textsection 5]{HS22}, respectively a straightforward generalisation of Lemma \ref{evaluation-on-thickening} for the last assertion.
\end{proof}

As the arguments in \cite[\textsection 5]{HS22} carry over word by word, we obtain the first step of our proof.

\begin{corollary} $\cW^\mu_{g,n,A}$ admits an oriented global Kuranishi chart $\widetilde{\cK}_n$ with boundary of the expected dimension. 
\end{corollary}

In particular, $\widetilde{\cK}_{n,0}$ is an oriented global Kuranishi chart for $\Mbar^\mu_X \coloneq \Mbar_{g,n}^{\,J}(X,A)\times_{\cc{M}_{g,n}}\Mbar_{g,n}^\mu$ with 
	\begin{equation}\label{vfc-of-good-cover}(\eva\times \normalfont\text{st})_*\vfc{\Mbar^\mu_X}_{\widetilde{\cK}_0} = d_\mu\, (\eva\times \normalfont\text{st})_*\vfc{\Mbar_{g,n}^{\,J}(X,A)}.\end{equation}
due to 

\begin{lemma}\label{vfc-of-cover} Suppose $\cK = (G,\cT,\cE,\fs)$ is an oriented global Kuranishi chart of a space $M$ and $\cT$ admits a degree-$d$ cover $p \cl \cT' \to\cT$. If $p$ is $G$-equivariant with respect to some $G$-action on $\cT'$, then $\cK' \coloneq (G,\cT',p^*\cE,p^*\fs)$ is a global Kuranishi chart for $M' \coloneq (p^*\fs)\inv(0)/G$. The canonical map $\cc{p}\cl M'\to M$ is a degree-$d$ cover and
	$$\cc{p}_*\vfc{M'}_{\cK'} = d\,\vfc{M}_\cK.$$
\end{lemma}

\begin{proof} The first part is straightforward. The relation between the virtual fundamental classes follows from the functoriality of Thom classes and because the map $\cT'_G\to \cT_G$ of homotopy quotients has degree $d$.
\end{proof}
 
It remains to show that 
	\begin{equation*}\label{vfc-ruan-tian}d_\mu \sigma^A_{g,n}  = (\eva\times \normalfont\text{st}_\mu)_*\vfc{\Mbar_{g,n}^\mu(A;J,\nu)}_{\widetilde{\cK}_1}.\end{equation*}

 This is a consequence of the following general result. It is the analogue of \cite[Theorem 5.2]{IP19b} in our setting. Compare also with Lemma 3.6 op. cit.

 \begin{lemma}\label{pseudocycle-vfc} Let $M$ be an oriented manifold of dimension $d$ inside a compact space $\cc{M}$ that admits a global Kuranishi chart $\cK = (G,\cT,\cE,\fs)$ of dimension $d$. Suppose $\fs$ intersects the zero section transversely over the preimage of $M$ and $G$ acts freely on that locus. Let $f \cl \cc{M}\to N$ be a continuous map to (the orbit space of) a smooth compact oriented orbifold, so that $f|_M$ is a pseudocycle. Then 
$f_*\vfc{\cc{M}} = [f|_M]$ in $H_*(N;\bQ)$.
 \end{lemma}

\begin{proof} Set $P \coloneq f(\cc{M}\sm M)$ and $M\inn \coloneq f\inv(N\sm P)$. Then $M\inn$ is an open submanifold of $M$ and $f \cl M\inn \to N\sm P$ is proper. In particular, $[f] = f_*[M\inn]\in H_d^{\text{BM}}(N\sm P;\bQ) \cong H_d^{\text{BM}}(N;\bQ)$. Let $j \cl M \hkra \cc{M}$ be the inclusion inducing $j_! \cl \chml^*_c(M\inn;\bQ)\to \chml^*(\cc{M};\bQ)$. By assumption on the Kuranishi section $\fs$, the class $(j_!)^*\vfc{\cc{M}}$ in $\chml^d_c(M\inn;\bQ)\dul$ corresponds to evaluation at the fundamental class $[M\inn]$. This implies that the diagram
\begin{center}\begin{tikzcd}
    \chml^d_c(N\sm P;\bQ)\arrow[r,"(f|_{M})^*"]\arrow[d,"i_!"] & \chml^d_c(M;\bQ)\arrow[d,"j_!"]\arrow[r,"\pcd"]&H_0(M;\bQ)\arrow[d,""]\\
     \chml^d(N;\bQ)\arrow[r,"f^*"] & \chml^d(\cc{M};\bQ)\arrow[r,"\vfc{\bar{M}}"]&\bQ
\end{tikzcd}\end{center}
commutes. Hence, $f_*\vfc{\cc{M}}$ agrees with the evaluation at $[f|_M]$ and thus the two define the same class in homology.
\end{proof}

\begin{remark} In \cite{IP19}, Ionel and Parker define the virtual fundamental class associated to a space with a thin compactification. The proof above shows that if $M$ is a smooth manifold with a thin compactification $\cc{M}$ that satisfies the assumptions of Lemma \ref{pseudocycle-vfc}, then the two notions of virtual fundamental class agree. \end{remark}

\appendix
\section{Virtual fundamental classes of cut-down moduli spaces}\label{subsec:embedded-charts}

This section is the technical backbone of \textsection\ref{sec:axioms}. We investigate how geometric relations between global Kuranishi charts translate into relations between their virtual fundamental classes.

\begin{definition}\label{morphism} A \emph{morphism of global Kuranishi charts} $\ff\cl \cK' =(G',\cT',\cE',\fs')\to \cK= (G,\cT,\cE,\fs)$ consists of a group morphism $\alpha\cl G\to G'$, an $\alpha$-equivariant map $f\cl \cT\to \cT'$ and an $\alpha$-equivariant vector bundle morphism $\tilde{f}\cl \cE\to f^*\cE'$ so that $\tilde{f}\fs = \fs' f$. We call $\ff$ an \emph{embedding} if $\alpha = \ide$, $f$ is an embedding of manifolds, and if $\tilde{f}$ is an injection of vector bundles.\\ If $\cK$ and $\cK'$ are rel--$C^\infty$ over base spaces $\cM$ respectively $\cM'$, we say the morphism is of \emph{class rel--$C^\infty$} if $\alpha$ is smooth and $f$ and $\tilde{f}$ are rel--$C^\infty$ covering a smooth morphism $\cM\to \cM'$. 
\end{definition}

\begin{remark} If $\ff = (\alpha,f,\tilde{f})$ is such that $\alpha$, $f$ and $F$ are embeddings in the respective category, we can replace $(G',\cT',\cE',\fs')$ with $(G,G\times_{G'} \cT', G\times_{G'} \cE',\ide\times\fs')$. \end{remark}

If $\cT/\cM$ is a rel--$C^\infty$ manifold with smooth base, its tangent microbundle has a canonical (equivariant) vector bundle lift given by $T_\cT\coloneq q^*T_\cM\dsm T_{\cT/\cM}$, where $q\cl \cT\to \cM$ is the structural map. Given a rel--$C^\infty$ embedding $j\cl \cT'/\cM' \hkra \cT/\cM$, where $\cM'$ is a smooth submanifold of $\cM$, we define the \emph{normal bundle} of $\cT'/\cM'$ inside $\cT/\cM$ to be 
$$N_{\cT'/\cT} \coloneq q^*N_{\cM'/\cM} \oplus N^v_{\cT'/\cT_{\cM'}},$$
where $N^v_{\cT'/\cT_{S'}} \coloneq j^*T_{\cT_{\cM'}/\cM'}/T_{\cT'/\cM'}$ is the \emph{vertical normal bundle} with $\cT_{\cM'} \coloneq \cT\times_\cM\cM'$.

\begin{definition}\label{Virtual Normal Bundle} Suppose $j \cl \cK'\hkra \cK = (G,\cT/\cM,\cE,\fs)$ is a rel--$C^\infty$ embedding. We call $ N_{\cK'/\cK}\coloneq N_{\cT'/\cT} - \cD$ 
	its \emph{virtual normal bundle}, where $\cD= \coker(\tilde{j})$. 
\end{definition}
%we will assume that our global Kuranishi charts admit an equivariant vector bundle lift of their tangent microbundle. 
Throughout, we assume our global Kuranishi charts to be oriented in the following sense, equivalent to \cite[\textsection5.4]{AMS21}. Clearly, if both $\cK$ and $\cK'$ are oriented, then so is $N_{\cK'/\cK}$. 

\begin{definition}\label{} A \emph{(Borel equivariant) orientation} of a Kuranishi chart $\cK= (G,\cT,\cE,\fs)$ consists of a $\bQ$-orientation of the virtual vector bundle $(T\cT)_G - \lc{\gk}$ and $\cE_G$ over $\cT_G$.   
\end{definition}

We need orientations of both $\cT/G$ as well as $\cE$ in order to define the virtual fundamental class. By \cite[Lemma 5.11]{AMS21}, this is equivalent to a $\bQ$-orientation of $(T\cT)_G - \lc{\gk}-\cE$.

\begin{notation*}
	Given $A\sub B$, we write $H_*(B\rht A;\bQ) \coloneq H_*(B,B\sm A;\bQ)$ and similarly for cohomology.
\end{notation*}

\begin{example} If $\cT$ and the action on it are smooth, there exists an embedding $\cT\times\fg\hkra T\cT$, where $\fg = \text{Lie}(G)$. Taking a $G$-invariant complement $\cD$ of this distribution, any choice of equivariant Thom class $\tau \in H^{\dim(\cT/G)}_G(\cD\rht\cT,\bQ)$ defines a $\bQ$-orientation of $\cT-\lc{\fg}$. 
\end{example}

Given an oriented orbifold $\cc{\cT}$ and an oriented suborbifold $\cc{\cT'}\hkra \cc{\cT}$ of codimension $k$, we have the Poincar\'e duality isomorphisms
$$H_{k}(\cc{\cT}\rht \cc{\cT}';\bQ) \cong H^{\dim(\cc{\cT'})}_c(\cc{\cT}\rht \cc{\cT}';\bQ)\cong H_0(\cc{\cT}';\bQ)$$
Thus, $H^{k}(\cc{\cT}\rht \cc{\cT}';\bQ) \cong \bQ^{|\pi_0(\cc{\cT}')|}$ and taking the sum of all generators, we obtain the \emph{Poincar\'e dual} $\pcd(\cc{\cT}')$ of $\cc{\cT}'$ in $\cc{\cT}$. The composite $j_!j^*\cl H\ucc(\cc{\cT};\bQ)\ra H^{*+k}(\cc{\cT};\bQ)$ is given by multiplication with the image of $\pcd(\cc{\cT}')$ in $H^k(\cc{\cT};\bQ)$.

\begin{remark}\label{rem:poinare-dual} In the case of thickenings as above, we can give an explicit description of the Poincar\'e dual in terms of the normal bundle. Factor the inclusion $j \cl \cT'\to \cT$ as 
	$$\cT'\xra{i} \widehat{\cT} \coloneq \cT'\times_\cM\cM' \to \cT,$$
	where we equip $\widehat{\cT}$ with the canonical $G$-orientation. A relative version of the equivariant tubular neighbourhood theorem shows that $\pcd_{\widehat{\cT}}(\cT'/G)$ corresponds to the equivariant Thom class of $N^v_{\cT'/\cT}$ under the canonical isomorphism induced by the tubular neighbourhood. Meanwhile, $\pcd_{\cT/G}(\widehat{\cT}/G) = q^*\pcd_{\cM/G}(\cM'/G)$, which corresponds to the equivariant Thom class of $q^*N_{\cM'/\cM}$. Thus, 
	\begin{equation}\label{}\pcd_{\cT/G}(\cT'/G) = \pcd_{\cT/G}(\widehat{\cT}/G)\cdot\pcd_{\widehat{\cT}}(\cT'/G). \end{equation}
\end{remark}

\begin{proposition}\label{vfc-embedded-chart} Let $j\cl \cK'\hkra\cK$ be a rel--$C^\infty$ embedding of oriented rel--$C^\infty$ global Kuranishi charts, covering a smooth embedding $\cM' \hkra\cM$ of oriented base spaces. If $N_{\cK'/\cK} \oplus \cD = N_{\cT'/\cT}$, then
	\begin{equation}\label{vfc-ec}
		j_*(e_G(\cD)\cap \vfc{\fM'}) = \pcd(\cT'/G)\cap \vfc{\fM}.
	\end{equation} 
\end{proposition}

Here we identify $e_G(\cD)\in H\ucc_G(\cT';\bQ)$ with the corresponding element in $H\ucc(\cT'/G;\bQ)$ and equip $\cD$ with the unique orientation so that $e_G(\cE'\oplus D) = e_G(\cE)|_{\cT'}$.

\begin{proof} Suppose first that $j^*\cE = \cE'$. The equality
	\begin{equation}\label{normal-vfc-same-obstruction} j_*\vfc{\fM'}_{\cK'} = \pcd(\cT'/G)\cap \vfc{\fM}_\cK\end{equation}
	follows from the commutativity of 
	\begin{center}\begin{tikzcd}
			\chml^{\vdim(\fM')+*}(\cM';\bQ) \arrow[rr,"{\fs'}^*\tau(\cE'/G)"] &&H^{\dim(\cT'/G)+*}_c(\cT'/G\rht\cM';\bQ) \arrow[d,"j_!"]\\ 
			\chml^{\vdim(\fM')+*}(\cM;\bQ)\arrow[u,"j^*"]\arrow[r,"\cdot\pcd(\cT'/G)"]
			&\chml^{\vdim(\fM)+*}(\cM;\bQ)\arrow[r,"\fs^*\tau(\cE/G)"] &H_c^{\dim(\cT/G)+*}(\cT/G;\bQ) \end{tikzcd} \end{center}
	where we need the convention 
	$$\inpr{\alpha\cap \vfc{\fM}}{\beta} = \inpr{\vfc{\fM}}{\beta\cdot \alpha}.$$
	Now assume $\rank(\cD) > 0$. Let $\wt{K} \coloneq (G,\cT',j^*\cE,\fs')$. This is also a global Kuranishi chart for $\cM'$, albeit with a larger obstruction bundle. By the definition of the virtual fundamental class,
	$$e_G(\cD)\cap \vfc{\fM'}_{\cK'} = \vfc{\fM'}_{\wt{\cK}}.$$
	This completes the proof.
\end{proof}

Proposition \ref{vfc-embedded-chart} is not optimal since one might have $\cE' = j^*\cE\dsm N_{\cT'/\cT}$, in which case we would expect that the virtual fundamental classes agree, at least under certain assumptions.

\begin{lemma}\label{quasi-isomorphic-global-charts} Let $\cK'$ and $\cK$ be rel--$C^\infty$ smooth global Kuranishi charts over $\cM$ for $M$. Suppose there exists a rel--$C^\infty$ embedding $j \cl \cK' \hkra \cK$ over $\cM$, inducing a quasi-isomorphism 
	\begin{equation}\label{quasi-gkc}[T_{\cT'/\cM}|_{{\fs'}\inv(0)}\xra{D\fs'}\cE'|_{{\fs'}\inv(0)}]\to [T_{\cT/\cM}|_{{\fs}\inv(0)}\xra{D\fs} \cE|_{{\fs}\inv(0)}] \end{equation} 
	of complexes of vector bundles. Then $\vfc{M}_{\cK'} = \vfc{M}_\cK$.
\end{lemma}

\begin{proof} Using a relative tubular neighbourhood, we may assume that $\cT$ admits a vector bundle structure $p \cl \cT\to \cT'$. Fix a splitting $\cE|_{\cT'} = \cE'\oplus \cD$. We may assume without loss of generality that $\cE= p^*\cE'\oplus p^*\cD$, where $p \cl \cT\to \cT'$ is the bundle map. Write $\fs = \fs_1\oplus \fs_2$ with respect to this splitting.\par 
	As $\cK' \hkra \cK$, we have $\cT'\sub \fs_2\inv(0)$. Given $x \in Z \coloneq {\fs'}\inv(0) = \fs\inv(0)$ we have an associated commutative diagram of vertical derivatives
	\begin{center}\begin{tikzcd}
			T_{\cT'/\cM,x} \arrow[rr,"d^v\fs'(x)"] \arrow[d,hook,""]&& \cE'_x \arrow[d,hook,""]\\
			T_{\cT'/\cM,x}\oplus \cT_x \arrow[rr,"d^v\fs_1(x)\oplus d^v\fs_2(x)"] && \cE'_x\oplus \cD_x \end{tikzcd} \end{center}
	Since $\coker(d^v\fs(x)) = \coker(d^v\fs'(x))$ is a quotient of $\cE'_x$, it follows that $\coker(d^v\fs_2(x)) = 0$. Replacing $\cT$ by a neighbourhood of $Z$, we may assume $\fs_2\pf 0$. Set $S \coloneq \fs_2\inv(0)$.  As $\dim(S) = \dim(\cT')$, the two global Kuranishi charts $\cK_2 \coloneq (G,S,p^*\cE'|_S,\fs_1|_S)$ and $\cK'$ are related by (Germ equivalence).\par 
	Finally, $\vfc{M}_{\cK_2} = \vfc{M}_{\cK}$, since the Poincar\'e dual of $S$ in $\cT$ is $\fs_2^*\tau_{p^*\cD}$.   
\end{proof}

In other words, the virtual fundamental class only depends on the global Kuranishi chart up to quasi-isomorphism, similar to \cite[Proposition 5.3]{BF97}.

\begin{example} Both the embedding condition and \eqref{quasi-gkc} are necessary for this to hold. To see this, consider $\cT= \bR =\cT'$ and $\cE = \lc{\bR^2} = \cE'$ with $\fs(t) = t^2$ and $\fs(t) = t^3$. Or $\cT' = \bR\times\{0\}\sub \cT =\bR^2$ with $\cE' = \lc{\bR}$ and $\cE = \lc{\bR^2}$ and $\fs(t,r) = (t^2,r^2)$ and $\fs'(t) = t^2$.
\end{example}

\section{Intersection theory on orbifolds}\label{sec:orbifold-intersection}

\subsection{Orbifolds}

In this appendix, we prove several results we needed in the previous sections. They mainly consist of projection formulas and the definition of a trace map in a quite general setting. We found neither in the literature in the form, respectively, generality necessary for our purposes.\par
For us, an \emph{orbifold} is a tuple $\fX = (\cc{\fX},[X])$, where $\cc{\fX}$ is a topological space, called the \emph{coarse moduli space of $\fX$}, $X$ is a (topological or smooth) proper \'etale groupoid and $[\,]$ denotes its Morita equivalence class. Here proper means that the source and target maps combine to give a proper map $X_1\to X_0\times X_0$, while the groupoid is \'etale if both source and target map are local homeomorphisms, respectively diffeomorphisms. $X$ is called a \emph{presentation} of $\fX$. Refer to \cite{Be04} for more details. By \cite[Corollary~1.2]{Par22}, any orbifold with only finitely many isotropy types can be represented by a \emph{global quotient}, that is, a proper \'etale groupoid of the form $[G\times M\rightrightarrows M]$ where $G$ is a compact Lie group acting almost freely, and possibly smoothly, on the manifold $M$. We denote by $[M/G]$ the orbifold presented by such a global quotient.\par
A proper \'etale groupoid $X = [X_1\rightrightarrows X_0]$ is \emph{orientable} if $X_0$ and $X_1$ are orientable and if the source and target maps from $X_1$ to $X_0$ are orientation-preserving. We call an orbifold \emph{orientable} if each representing groupoid is orientable. Thus, $[M/G]$ is orientable if and only if $M$ is orientable and and $G$ acts by orientation-preserving homeomorphisms.

\begin{remark} This notion of orientability is strictly stronger than requiring the orientability of the coarse moduli space as the example of the Klein bottle shows, whose quotient by an $S^1$-action is $S^1$ itself.
\end{remark}

By \cite[p.27]{Be04}, there is a canonical isomorphism 
\begin{equation}\label{eq:orbifold-cohomology}H^*([M/G],\bQ)\cong H^*_G(M,\bQ). \end{equation}
Let $q \cl M_G \to M/G$ denote the canonical map to the quotient; it defines by \cite[Proposition 36]{Be04} an isomorphism 
$$H^*(M/G,\bQ) \to H^*_G(M,\bQ).$$

As the coarse moduli space of an oriented orbifold is an oriented homology $\bQ$-manifold, it satisfies rational Poincar\'e duality
$$\chml\ucc_c(Z,\bQ)\cong H_{\dim(\fX)-*}(\cc{\fX},\cc{\fX}\sm Z,\bQ)$$
by \cite{Bre12}.\footnote{The proof of \cite[Lemma~A.6.4]{P16} also generalises easily to this setting.} Replacing singular homology with Borel--More homology, the same isomorphism holds with ordinary \v{C}ech cohomology on the left-hand side.\\

An \emph{orbibundle} $\fE$ over $\fX= (\cc{\fX},X)$ consists of an equivalence class of proper \'etale groupoids $E = [E_1\rightrightarrows E_0]$ so that there exists a morphism of groupoids $\pi \cl E\to X$ with $\pi_0 \cl E_0 \to X_0$ a vector bundle such that the square
	\begin{center}\begin{tikzcd}
			E_1 \arrow[r,"s^E"] \arrow[d,"\pi_1"]& E_0\arrow[d,"\pi"]\\ 
			X_1 \arrow[r,"s^X"] & X_0 \end{tikzcd} \end{center}
	is cartesian and $t^X\pi_1 = \pi_0 t^E$. 

\begin{example} If $\fX = [M/G]$, then orbibundles $\fE\to \fX$ correspond to global quotients $[E/G]$ of $G$-vector bundles $E\to M$.
\end{example}

\begin{remark}[Tangent bundle]\label{rem:orbifold-tangent-bundle} If $G$ acts smoothly and almost freely, the infinitesimal action defines an inclusion $M\times \fg\hkra TM$ of vector bundles, where $\fg = \text{Lie}(G)$. Let $\cD\subset TM$ be a $G$-invariant complement. Then the orbibundles $[G\times\cD\rightrightarrows \cD]$ and $[TG\times TM\rightrightarrows TM]$ are equivalent. By \eqref{eq:orbifold-cohomology}, an \emph{orientation} of $[M/G]$ is therefore the same as an equivariant Thom class of $\cD \to M$. 
\end{remark}

\subsection{Exceptional pushfoward}\label{sec:exc-push}

Let $f \cl M^m\to N^n$ be a $G$-equivariant map between smooth manifolds on which $G$ acts smoothly and almost freely. Suppose $[M/G]$ and $[N/G]$ are oriented. Then $f$ induces a morphism $[M/G]\to [N/G]$ and the \emph{exceptional pushforward} 
$$f_! \cl H^*(M/G,\bQ) \to H^{* +n-m}(N/G,\bQ)$$
is defined by $f_! \coloneq \pcd\, f_*\,\pcd$. We clearly have $g_!f_! = (g f)_!$.\par 
By \cite[Theorem 4.1]{Bre72}, we can factor $f$ as a composite $M\xra{j} N\times S^V\xra{\pr_1} N$, where $S^V$ is the one-point compactification of a finite-dimensional orthogonal $G$-representation $V$ and $j$ is an equivariant embedding. We can describe $j_!$ and ${\pr_1}_!$ explicitly.

\begin{example}[Embedding] Suppose $f$ is an embedding with Poincar\'e dual $\pcd(M/G)$. Then
	\begin{equation}\label{gysin-embedding} f_!f^*(\alpha) = \alpha \cdot \pcd(M/G)\end{equation}
	for $\alpha\in H^*(N/G;\bQ)$. If there exists an equivariant retraction $r \cl W\to M$, then $f_!(\alpha) = r^*\alpha \cdot  \pcd(M/G)$.
\end{example}

\begin{example}[Projection]\label{gysin-projection} Suppose $M = N\times S^V$, where $V$ is a finite-dimensional $G$-representation, $S^V$ is its one-point compactification (to which the $G$-action extends trivially), and $f$ is the projection. Then $f_G \cl (N\times S^V)_G \to N_G$ is the sphere bundle of $(N\times (V\dsm \bR))_G \to N_G$ and 
	\begin{equation*}\begin{tikzcd}
			H^*_G(N\times S^V,\bQ)\arrow[r,hook,""] \arrow[dr,"f_!"]&H^{*+1}_G(N\times (V\oplus \bR),N\times ((V\oplus \bR)\sm 0),\bQ) \arrow[d,"\cong"]\\ &H^{*-k}_G(N,\bQ)  \end{tikzcd} \end{equation*}
	commutes by \cite[\textsection 3]{Du03}, where the vertical map comes from the Thom isomorphism. 
\end{example}

In particular, we have the following observation.

\begin{corollary}\label{push-pull-principal-bundle} Suppose we have a cartesian square
	\begin{center}\begin{tikzcd}
			P \arrow[r,"f"] \arrow[d,"q"]&P' \arrow[d,"q'"]\\ B\arrow[r,"\bar{f}"] & B' \end{tikzcd} \end{center}
	where $q$ and $q'$ are principal $G$-bundles for a compact Lie group $G$, $B, B'$ are oriented smooth manifolds (and $P$, $P'$ are equipped with the corresponding $G$-orientation), $\bar{f}$ is smooth and proper, and a compact Lie group $G'$ acts on the whole square almost freely. Then 
	\begin{equation}\label{psh-pull-pb} f_!q^* = {q'}^*\bar{f}_!.\end{equation}
\end{corollary}

\begin{lemma}[Projection formula]\label{projection-formula} Suppose we have a cartesian square
\begin{center}\begin{tikzcd}
		X' \arrow[r,"f'"] \arrow[d,"q"]&Y'\arrow[d,"p"]\\ X \arrow[r,"f"] & Y \end{tikzcd} \end{center}
of oriented $G$-manifolds, where $G$ acts almost freely. If $f$ is proper and $p$ is a submersion, then 
\begin{equation}\label{eq:proj-formula}f'_!q^* = p^*f_!. \end{equation}
If the $G$-action on $X'$ and $Y'$ extends to an almost $G\times G'$-action with respect to which $f'$, $q$ and $p$ are equivariant, then the equality \eqref{eq:proj-formula} holds for the maps $H^*(X/G;\bQ)\to H^{k+*}(Y'/G\times G';\bQ)$.
\end{lemma}

\begin{proof} If $f$ is an embedding, so is $f'$ and the claim follows from \eqref{gysin-embedding}. If $f$ is a projection $Y\times S^V\to Y$ for some finite-dimensional $G$-representation $V$, we can assume $f'$ is the projection $Y'\times S^V\to Y'$. Denoting the induced map $Y'_{G\times G'}\to Y_G$ by $p$ as well, we have $$(Y'\times (V\dsm \bR))_{G\times G'} = p^*(Y\times (V\dsm \bR))_G.$$ 
	Thus, the claim follows from Example \eqref{gysin-projection} and the functoriality of the Thom class.
\end{proof}

\begin{corollary}\label{cor:poincare-dual-orbifold} Suppose we have a cartesian square as in Lemma \ref{projection-formula}, where $f$ is an embedding. Then $p^*\pcd(X/G) = \pcd(X'/G\times G')$. 
\end{corollary}

\begin{proof} Let $\cc{Y'}$ and $\cc{X'}$ denote the quotients, and suppose $k = \codim(X')$. Since both $p^*\pcd(X/G)$ and $\pcd(X'/G\times G')$ live in $H^k(\cc{Y'}\mid \cc{X'};\bQ)\cong \bQ^{|\pi_0(\cc{X'})|}$, they differ by multiplication with a locally constant function $b$. Thus, $p^*f_!f^*\alpha = p^*\alpha\cdot p^*\pcd(\cc{X})$ on the one hand, while 
$$p^*f_!f^*\alpha = f'_!q^*f^*\alpha = f'_!{f'}^*p^*\alpha =  = p^*\alpha \cdot \pcd(\cc{X'}) = bp^*\alpha \cdot p^*\pcd(\cc{X}).$$
As the same equality holds locally on $Y$ and $p$ admits local sections, we must have $b \equiv 1$.
\end{proof}

\begin{lemma}\label{projection-formula-up-codimension} Let 
\begin{center}\begin{tikzcd}
		X' \arrow[r,"f'"] \arrow[d,"q"]&Y'\arrow[d,"p"]\\ X \arrow[r,"f"] & Y \end{tikzcd} \end{center}
be a cartesian square of smooth manifolds, where $f$ is proper. Suppose there exists an open subset $V\sub Y'$ so that $Y'\sm V$ has codimension at least $2$ and $p|_V \cl V\to Y$ is a (not necessarily surjective) submersion. Then $$f'_!q^* = p^*f_!.$$ The same is true in equivariant cohomology if given the assumptions of lemma \ref{projection-formula}.
\end{lemma}

\begin{proof} Set $U \coloneq X'\times_{Y'} V$ and let $i \cl U \to X'$ and $j \cl V\to Y'$ be the inclusions. Let $\hat{f}\cl U\to V$ be the induced map. The fundamental classes appearing below are elements of Borel--Moore homology. By \cite[\textsection V.10(57),Corollary V.10.2]{Bre12} and Lemma \ref{projection-formula}, the claim holds in the nonequivariant case.\par 
%we have for $\gamma\in H^*(X,\bZ)$ that 
%	\begin{align*}f'_*(q^*\gamma \cap \fcl{X'}) &=f'_*(q^*\gamma \cap i_*\fcl{U}) \\&= f'_*i_*(i^*q^*\gamma\cap \fcl{U})  \\&= j_*\hat{f}_*(i^*q^*\gamma\cap \fcl{U}) \\&= j_*(j^*p^*(f_!\gamma)\cap \fcl{V}) \\& = p^*(f_!\gamma)\cap \fcl{Y'}.\end{align*}
Suppose now that $G$ and $G'$ are compact Lie groups so that $G\times G'$ acts almost freely on $X'$ and $Y'$ and $G$ acts almost freely on $Y$ with $f$ and $f'$ being equivariant and $p,q$ being invariant under the $G'$-action and restricting to principal bundles over $V,U$. Since we can check the equality $f'_!q^* = p^*f!$ as maps $H^*_G(X,\bQ)\to H^{*+k}_{G\times G'}(Y',\bQ)$ degree for degree, we can use finite-dimensional approximations of the classifying space of $G$. To these, apply Corollary \ref{push-pull-principal-bundle} to see that the projection formula holds for the pullbacks to $U$. Then use the same argument as in the first step to conclude. 
\end{proof}

\begin{lemma}\label{resolution-fibre-product} Suppose $f\cl X\to Y$ is a smooth map of \'etale proper groupoids and $[M/G]$ is a global quotient representing $Y$. If $X$ is a manifold, then the orbifold fibre product $M \times_Y X$ is a principal $G$-bundle over $X$.
\end{lemma}

\begin{proof} Since $M$ and $X$ are manifolds, so is $Z \coloneq M \times_Y X$. Let $q_0 \cl M\to Y_0$ be the canonical map. Then $Z = \{(p,\alpha,x)\in M\times Y_1\times X: \alpha \cl q(p)\to f(x)\}$
and $\pi \cl Z \to X$ is given by $\pi(p,\alpha,x) = x$. Define a $G$-action on $Z$ by setting 
$$g \cdot (p,\alpha,x) \coloneq (g\cdot p,\alpha \g q_1(g,p)\inv,x).$$
Clearly, $\pi$ is $G$-invariant and $g\cdot (p,\alpha,x) = (p,\alpha,x)$ implies that $g\cdot p = p$ and $q_1(g,p) = \ide$. Since $q\cl [M/G]\to Y$ is \'etale, we must have $g = e$. Thus, $G$ acts freely on $Z$. To see that $\pi$ is locally trivial, it suffices to consider the case where $Y = [V/\Gamma]$ for some finite group $\Gamma$ and $M = [S/G_p]$ for some slice $S$ through $p$. In this case, $\pi$ is the pullback of a covering map and thus a local diffeomorphism. This completes the proof.
\end{proof}

\subsection{Trace maps}\label{subsec:trace} In the definition of the equivariant virtual fundamental class, we make use of a \emph{trace map} $H^{*+m}_{K,fc}(\cT;\bQ)\to H^*_K(\Pt;\bQ)$, which is a special case of integration along the fibre. While this is classical for fibre bundles of closed smooth manifolds, \cite{BT82}, and has been generalised in algebraic geometry, \cite{Iv86,KS94}, we found no results for the specific situation needed in this paper. Thus, we give a brief definition and show the required properties. To avoid any subtleties with families of supports, we will assume that all spaces are locally compact, Hausdorff and paracompact. 

\begin{definition}[Integration along the fibre] Suppose $\pi \cl P \to B$ is an oriented fibre bundle over a paracompact base $B$ with fibre an oriented topological orbifold $X = [M/G]$. Denote by $\cH^*_c(X)$ the locally constant sheaf on $B$ with stalks given by $\cH^*_c(X)_b = H^*_c(P_b;\bQ)$ for $b \in B$. By \cite[Theorem 6.1]{Bre12} there exists a spectral sequence $\{E^{p,q}_r\}$ converging to $H^*_{fc}(P;\bQ)$ with 
$$E^{p,q}_2 = H^p(B;\cH^q_c(X)).$$
In particular, $E^{p,q}_2 = 0$ for $q > n\coloneq \dim(X)$, so $E^{p,n}_r \sub E^{p,n}_{r-1}$ for any $r > 2$. We have a canonical map $\int_X \cl\cH^n_c(X)\to \lc{\bQ}$ of locally constant sheaves on $B$; it is given at the stalk over $b \in B$ by 
$$\cH^n_c(X)_b =
H^n_c(X_b;\bQ)\xra{\Pt_!} \bQ.$$

We define the \emph{integration along the fibre} $\pi_*\cl H^{n+*}_{fc}(P;\bQ)\to H^*(B;\bQ)$ to be the composite
$$H^{n+*}_{fc}(P;\bQ) \to E^{*,n}_\infty \hkra E^{*,n}_2 = H^*(B;\cH^n_c(X)) \xra{(\int_X)_\#} H^*(B;\bQ).$$
\end{definition}

By \cite{Au73}, this agrees with the standard definition of integration along the fibre for smooth fibre bundles.

\begin{lemma}[Base change]\label{lem:trace-base-change}Suppose $\pi \cl P \to B$ is an orientable fibre bundle over a paracompact base with fibre $\cT$ an oriented orbifold and $f \cl B'\to B$ is a proper continuous map from another paracompact space. Then
\begin{center}\begin{tikzcd}
		H^{m+*}_{c}(P;\bQ)\arrow[r,] \arrow[d,"\tilde{f}^*"] & H^*(B;\bQ)\arrow[d,"f^*"]\\
		H^{m+*}_{c}(f^*P;\bQ) \arrow[r,]& H^*(B';\bQ)\end{tikzcd} \end{center} 
commutes. 
\end{lemma}

\begin{proof} This follows from the functoriality of the Leray-Serre spectral sequence associated to a fibration; see \cite[\textsection 6.2]{Bre12}.
\end{proof}

\begin{corollary}\label{cor:eq-trace} Suppose $\cT$ is an oriented topological orbifold of dimension $m$ with a continuous action by a compact connected Lie group $G$. Then 
\begin{center}\begin{tikzcd}
		H^{m+*}_{G,c}(\cT;\bQ)\arrow[r,] \arrow[d,"i^*"] & H^*_G(\Pt;\bQ)\arrow[d,""]\\
		H^{m+*}_{c}(\cT;\bQ) \arrow[r,]& H^*(\Pt;\bQ)\end{tikzcd} \end{center}
commutes, where $H^{*}_{G,c}(\cT;\bQ) = H^{*}_{fc}(\cT_G;\bQ)$.
\end{corollary}

\begin{lemma}[Functoriality]\label{lem:trace-functoriality} Suppose $\pi \cl P\to B$ and $\rho \cl E\to P$ are two oriented fibre bundles with fibres $X$ and $Y$ the coarse moduli spaces of oriented orbifolds. Then 
$$(\pi\rho)_* = \pi_*\rho_*\iota,$$ where $\iota$ is the canonical map from cohomology with $(\pi\rho)$-fibrewise compact support to cohomology with $\pi$-fibrewise compact support. 
\end{lemma}

\begin{proof} We will assume both $X$ and $Y$ are compact, of dimension $k$, respectively $\ell$, in order to simplify the notation. The general case can be obtained by restricting to cohomology with suitable support. Let $\theta \coloneq \pi\rho$ and set $Z_b \coloneq \theta\inv(\{b\})$ for $b \in B$. The maps ${\rho_b}_*\cl H^{k+\ell}(Z_b;\bQ) \to H^k(P_b;\bQ)$ induce a morphism $\rho_*\cl \cH^{k+\ell}(Z)\to \cH^k(X)$ and $\theta_*$ factors as 
$$H^{k+\ell+*}(E;\bQ)\to H^*(B,\cH^{k+\ell}(Z)) \xra{(\rho_*)_\#} H^*(B,\cH^{k}(X)) \xra{\pi_*}H^*(B;\bQ).$$
By \cite[Corollary IV.7.3]{Bre12}, the Leray sheaf $\mathscr{H}(\pi;\cH^\ell(Y))$ of $\pi$ with coefficients in $\cH^\ell(Y)$, defined in \cite[\textsection IV.4]{Bre12}, is locally constant with stalks of the form $H^*(P_b;\cH^\ell(Y))$.\footnote{We use here that the monodromy of $\cH^*(\ell)$ is trivial in degree $\ell$ since we work with oriented fibre bundles.} Due to the functoriality of the Serre spectral sequence, there exists a canonical morphism $\cH^{*+\ell}(Z)\to \mathscr{H}^*(\pi;\cH^\ell(Y))$ of locally constant sheaves on $B$, given stalkwise by 
$$\cH^{*+\ell}(Z)_b = H^{*+\ell}(\theta\inv(\{b\};\bQ) \twoheadrightarrow E(b)^{*,\ell}_\infty \hkra  H^*(P_b;\cH^\ell(Y)) = \mathscr{H}^*(\pi;\cH^\ell(Y))_b.$$
where $\{E(b)^{p,q}_r\}$ is the Leray-Serre spectral sequence of $\theta\inv(\{b\})\to P_b$. 
%As $P_b \cong X$ has rational cohomological dimension $k$, for $* = k$, the first map is an isomorphism. 
This stalkwise description shows that $\rho_* \cl \cH^{*+\ell}(Z)\to \cH^k(X)$ factors through $\mathscr{H}^*(\pi;\cH^\ell(Y))$. Thus
\begin{center}\begin{tikzcd}
		H^{k+\ell+*}(E;\bQ)  \arrow[r,] \arrow[d]& H^{*}(B;\cH^{k+\ell}(Z))\arrow[r,"(\rho_*)_\#"]\arrow[d,] & H^*(B;\cH^k(X))\\
		H^{k+*}_{c}(P;\cH^\ell(Y)) \arrow[r,""]& H^*(P;\mathscr{H}^*(\pi;\cH^\ell(Y)))\arrow[ur]\end{tikzcd} \end{center}
commutes. By \cite[\textsection 6.2]{Bre12}, 
\begin{center}\begin{tikzcd}
		H^{k+*}_{c}(P;\cH^\ell(Y)) \arrow[r,""]\arrow[d,"(\int_Y)_\#"]& H^*(P;\mathscr{H}^*(\pi;\cH^\ell(Y)))\arrow[d,]\\
		H^{k+*}(P;\bQ)\arrow[r] & H^*(B;\cH^k(X))\end{tikzcd} \end{center}
commutes as well. The claim now follows by composing with $\pi_*$.
\end{proof}

\begin{lemma}\label{lem:trace-subfibration} Let $\pi \cl P \to B$ be an oriented locally trivial fibration over a locally contractible space with fibre the orbit space of $[M/G]$. Suppose $P'\subset P$ is a subspace so that the induced map $\pi'\cl P'\to B$ is an oriented fibre bundle with fibre given by the orbit space of $[M'/G]$, for a $G$-invariant submanifold $M'\sub M$. Assume the inclusion $P'\hkra P$ admits a normal bundle and a tubular neighbourhood. Then 
\begin{center}\begin{tikzcd}
		H^{m'+*}_{fc}(P';\bQ)\arrow[r,"\pi'_*"] \arrow[d,"{j_!}"] & H^*(B;\bQ)\\
		H^{m+*}_{fc}(P;\bQ) \arrow[ur,"\pi_*"]\end{tikzcd} \end{center}
commutes, where $m = \dim([M/G])$ and $m' = \dim([M'/G])$. Moreover, $j_!j^* = \sigma\cdot$ for a class $\sigma \in H^k(P\mid P';\bQ)$ restricting to the Poincar\'e dual of $\fX'$ over a fibre.
\end{lemma}

\begin{proof} Let $\rho\cl\cN \to P'$ be the normal bundle of the embedding and $\psi \cl V\to W\sub \cN$ be a tubular neighbourhood. Then $\sigma =\psi^*\tau_{\cN}$ and $j_!$ are given by the composite
$$H^*_{fc}(P';\bQ)\xra{\simeq} H^{k+*}_{fc}(\cN\mid P';\bQ) \xra{\psi^*} H^{k+*}_{fc}(W\mid P';\bQ) \to H^{k+*}_{fc}(P;\bQ).$$
It suffices thus to show that the above triangle commutes with $P$ replaced by $\cN$. In this case, $j_!$ is an isomorphism with inverse given by $\rho_*$. Thus, the claim follows from Lemma \ref{lem:trace-functoriality}.
\end{proof}

\begin{lemma}\label{boundary-integration} Suppose $X$ is the orbit space of an oriented global quotient orbifold of dimension $n$ with boundary and $\pi \cl P \to B$ is a fibre bundle with fibre $X$. If $j\cl P'\hkra P$ denotes the subbundle with fibre $\del X$, then the composition $$H^{*+ n-1}_c(P;\bQ)\xra{j^*} H^{*+n-1}_c(P';\bQ)\xra{\pi_*} H^*(B;\bQ)$$ vanishes.
\end{lemma}

\begin{proof} The homology $\bQ$-manifold $\tilde{X} \coloneq X\cup_{\del X} \del X\times[0,1)$ admits a proper deformation retraction onto $X$, as does $\tilde{P} \coloneq X\cup_{P'}  P'\times[0,1)$ onto $P$ (where it is fibrewise proper). The deformation retraction fixes $P'$ pointwise and is a map of fibre bundles over $B$. By the long exact sequence in compactly supported cohomology, it suffices to show that $H^{*+n-1}_c(P';\bQ)\to H^*(B;\bQ)$ factors through $H^{*+n-1}_c(P';\bQ)\to H^{*+n}_c(\tilde{P};\bQ)$.\par
The results of \cite{Br62} generalise directly to the setting of topological manifolds with boundary, on which a compact group $G$ acts almost freely and locally linearly, and to fibrations thereof. Hence, $\pd X$ admits a collar inside $X$ and $P'$ admits one inside $P$. Thus, we can find a neighbourhood $U\sub \tilde{P}$ with $U \cong P'\times (-1,1)$ and the claim reduces to showing the commutativity of 
\begin{center}\begin{tikzcd}
		H^{*+n-1}_c(P';\bQ)\arrow[r,] \arrow[d,"\pi_!"] & H^{*+n}_c(P'\times (-1,1);\bQ)\arrow[dl,"\pi_!"]\\
		H^*(B;\bQ)\end{tikzcd} \end{center}
which is an immediate consequence of the K\"unneth theorem.\end{proof}
%Since the connecting homomorphism $H^{n-1}_c(\pd X;\bQ)\xra{\delta} H^n_c(X,\pd X;\bQ)$ is Poinar\'e dual to $H_0(\pd X;\bQ)\to H_0(X;\bQ)$, the integration map $H^{n-1}_c(\pd X;\bQ)\to \bQ$ factors through $\delta$. As $\delta$ vanishes on the image of $H^{n-1}_c(X;\bQ)$, we may conclude.

\begin{lemma}\label{cobordant-charts} Suppose two oriented global Kuranishi charts $\cK_i = (G,\cT_i,\cE_i,\fs_i)$ for $\fM_i$ are cobordant via $\cK= (G,\cT,\cE,\fs)$. If $f_i \cl \fM_i \to N$ is a continuous map so that $f_0\sqcup f_1$ extends over $\fs\inv(0)/G$, then ${f_0}_*\vfc{\fM_0} = {f_1}_*\vfc{\fM_1}$ in $\chml^*(N;\bQ)\dul$. The same is true in the equivariant setting.
\end{lemma}

\begin{proof} Set $\fW \coloneq \fs\inv(0)/G$. The claim follows from Lemma \ref{boundary-integration} and the commutativity of
\begin{center}\begin{tikzcd}
		\chml^{d+ n}(N;\bQ)\arrow[r,"f^*"]\arrow[dr,"f_0^*\sqcup -f_1^*"]&\chml^{d+*}(\fW;\bQ) \arrow[d,""]\arrow[r,"\fs^*\tau_{\cE}"] &H^{m+*}_c(\cT_0/G\sqcup \cT_1/G;\bQ)\arrow[d]\\
		&\chml^{d+*}(\fM_0\sqcup \fM_1;\bQ)\arrow[r,"\fs^*\tau_{\cE}"] &H^{m+*}_c(\cT_0/G\sqcup \cT_1/G;\bQ)\arrow[r,"\Pt_!"]& \bQ \end{tikzcd} \end{center}
where $d$ is the virtual dimension of $\fM_0$ and $\fM_1$. Since the equivariant version of this diagram also commutes and Lemma \ref{boundary-integration} is phrased for the equivariant setting, the last assertion is immediate.
\end{proof}

\bibliographystyle{amsalpha}
\bibliography{HGGW}

@book {ACG-moduli,
    AUTHOR = {Arbarello, E. and Cornalba, M. and Griffiths, P.
              A.},
     TITLE = {Geometry of algebraic curves. {V}olume {II}},
    SERIES = {Grundlehren der Mathematischen Wissenschaften [Fundamental
              Principles of Mathematical Sciences]},
    VOLUME = {268},
      NOTE = {With a contribution by Joseph Daniel Harris},
 PUBLISHER = {Springer, Heidelberg},
      YEAR = {2011},
     PAGES = {xxx+963},
      ISBN = {978-3-540-42688-2},
   MRCLASS = {14H10 (32G15)},
  MRNUMBER = {2807457},
MRREVIEWER = {E. Looijenga},
       DOI = {10.1007/978-3-540-69392-5},
       URL = {https://doi.org/10.1007/978-3-540-69392-5},}

@misc{AMS21,
	title={Complex cobordism, {H}amiltonian loops and global {K}uranishi charts}, 
	author={Abouzaid, M. and McLean, M. and Smith, I.},
	year={2021},
	note={arXiv:2110.14320},
	archivePrefix={arXiv},
	primaryClass={math.SG}}

@article {AMS23,
	AUTHOR = {Abouzaid, Mohammed and McLean, Mark and Smith, Ivan},
	TITLE = {Gromov-{W}itten invariants in complex and {M}orava-local
	{$K$}-theories},
	JOURNAL = {Geom. Funct. Anal.},
	FJOURNAL = {Geometric and Functional Analysis},
	VOLUME = {34},
	YEAR = {2024},
	NUMBER = {6},
	PAGES = {1647--1733},
	ISSN = {1016-443X,1420-8970},
	MRCLASS = {53D45 (14N35 55N15)},
	MRNUMBER = {4823210},
	DOI = {10.1007/s00039-024-00697-4},
	URL = {https://doi.org/10.1007/s00039-024-00697-4},}

@article {AB84,
    AUTHOR = {Atiyah, M. F. and Bott, R.},
     TITLE = {The moment map and equivariant cohomology},
   JOURNAL = {Topology},
  FJOURNAL = {Topology. An International Journal of Mathematics},
    VOLUME = {23},
      YEAR = {1984},
    NUMBER = {1},
     PAGES = {1--28},
      ISSN = {0040-9383},
   MRCLASS = {58F05 (55N91 57R20 57R22)},
  MRNUMBER = {721448},
MRREVIEWER = {G. J. Heckman},
       DOI = {10.1016/0040-9383(84)90021-1},
       URL = {https://doi.org/10.1016/0040-9383(84)90021-1},}

@article {Au73,
	AUTHOR = {Auer, J. W.},
	TITLE = {Fiber integration in smooth bundles},
	JOURNAL = {Pacific J. Math.},
	FJOURNAL = {Pacific Journal of Mathematics},
	VOLUME = {44},
	YEAR = {1973},
	PAGES = {33--43},
	ISSN = {0030-8730,1945-5844},
	MRCLASS = {57D20 (58A10)},
	MRNUMBER = {314065},
	MRREVIEWER = {J.\ R.\ Vanstone},
	URL = {http://projecteuclid.org/euclid.pjm/1102948641},}

@article {Beh99,
	AUTHOR = {Behrend, K.},
	TITLE = {The product formula for {G}romov-{W}itten invariants},
	JOURNAL = {J. Algebraic Geom.},
	FJOURNAL = {Journal of Algebraic Geometry},
	VOLUME = {8},
	YEAR = {1999},
	NUMBER = {3},
	PAGES = {529--541},
	ISSN = {1056-3911},
	MRCLASS = {14N35},
	MRNUMBER = {1689355},
	MRREVIEWER = {Gary P. Kennedy},}

@article{Be04,
  title={Cohomology of stacks},
  author={Behrend, Kai},
  journal={Intersection theory and moduli, ICTP Lect. Notes},
  volume={19},
  pages={249--294},
  year={2004}}

@article {BF97,
    AUTHOR = {Behrend, K. and Fantechi, B.},
     TITLE = {The intrinsic normal cone},
   JOURNAL = {Invent. Math.},
  FJOURNAL = {Inventiones Mathematicae},
    VOLUME = {128},
      YEAR = {1997},
    NUMBER = {1},
     PAGES = {45--88},
      ISSN = {0020-9910},
   MRCLASS = {14F99 (14C15 14D20)},
  MRNUMBER = {1437495},
MRREVIEWER = {Tohru Nakashima},
       DOI = {10.1007/s002220050136},
       URL = {https://doi-org.ezp.lib.cam.ac.uk/10.1007/s002220050136},}

@article {Sie99,
	AUTHOR = {Siebert, Bernd},
	TITLE = {Algebraic and symplectic {G}romov-{W}itten invariants
	coincide},
	JOURNAL = {Ann. Inst. Fourier (Grenoble)},
	FJOURNAL = {Universit\'{e} de Grenoble. Annales de l'Institut Fourier},
	VOLUME = {49},
	YEAR = {1999},
	NUMBER = {6},
	PAGES = {1743--1795},
	ISSN = {0373-0956,1777-5310},
	MRCLASS = {14N35 (32Q65 53D45)},
	MRNUMBER = {1738065},
	MRREVIEWER = {Jim\ A.\ Bryan},
	URL = {http://www.numdam.org/item?id=AIF_1999__49_6_1743_0},
}

@book {Bre72,
	AUTHOR = {Bredon, G. E.},
	TITLE = {Introduction to compact transformation groups},
	SERIES = {Pure and Applied Mathematics, Vol. 46},
	PUBLISHER = {Academic Press, New York-London},
	YEAR = {1972},
	PAGES = {xiii+459},
	MRCLASS = {57E15},
	MRNUMBER = {0413144},
	MRREVIEWER = {P. Y. Wang},}

@book{Bre12,
	AUTHOR = {Bredon, G. E.},
	TITLE = {Sheaf theory},
	SERIES = {Graduate Texts in Mathematics},
	VOLUME = {170},
	EDITION = {Second},
	PUBLISHER = {Springer-Verlag, New York},
	YEAR = {1997},
	PAGES = {xii+502},
	ISBN = {0-387-94905-4},
	MRCLASS = {55N30 (18F20 54B40 55-02)},
	MRNUMBER = {1481706},
	DOI = {10.1007/978-1-4612-0647-7},
	URL = {https://doi.org/10.1007/978-1-4612-0647-7},}

@article{Br62,
  title={Locally flat imbeddings of topological manifolds},
  author={Brown, Morton},
  journal={Annals of Mathematics},
  pages={331--341},
  year={1962},
  publisher={JSTOR}}

@article {Br14,
	AUTHOR = {Brown, Jeff},
	TITLE = {Gromov-{W}itten invariants of toric fibrations},
	JOURNAL = {Int. Math. Res. Not. IMRN},
	FJOURNAL = {International Mathematics Research Notices. IMRN},
	YEAR = {2014},
	NUMBER = {19},
	PAGES = {5437--5482},
	ISSN = {1073-7928,1687-0247},
	MRCLASS = {53D45 (14M25)},
	MRNUMBER = {3267376},
	MRREVIEWER = {Eduardo\ A.\ Gonzalez},
	DOI = {10.1093/imrn/rnt030},
	URL = {https://doi.org/10.1093/imrn/rnt030},}

@book {BT82,
	AUTHOR = {Bott, Raoul and Tu, Loring W.},
	TITLE = {Differential forms in algebraic topology},
	SERIES = {},
	VOLUME = {82.},
	PUBLISHER = {Springer-Verlag, New York-Berlin},
	YEAR = {1982},
	PAGES = {xiv+331},
	ISBN = {0-387-90613-4},
	MRCLASS = {57R19 (55-02 58-01 58A12)},
	MRNUMBER = {658304},
	MRREVIEWER = {Hansklaus\ Rummler},}

@article {CGMS02,
    AUTHOR = {Cieliebak, Kai and Gaio, A. Rita and Mundet i Riera, Ignasi
              and Salamon, Dietmar A.},
     TITLE = {The symplectic vortex equations and invariants of
              {H}amiltonian group actions},
   JOURNAL = {J. Symplectic Geom.},
  FJOURNAL = {The Journal of Symplectic Geometry},
    VOLUME = {1},
      YEAR = {2002},
    NUMBER = {3},
     PAGES = {543--645},
      ISSN = {1527-5256,1540-2347},
   MRCLASS = {53D45 (53D20 57R57 57S25)},
  MRNUMBER = {1959059},
MRREVIEWER = {Klaus\ Mohnke},
       URL = {http://projecteuclid.org/euclid.jsg/1092403032},}

@article {CM,
  	AUTHOR = {Cieliebak, K. and Mohnke, K.},
  	TITLE = {Symplectic hypersurfaces and transversality in
  	{G}romov-{W}itten theory},
  	JOURNAL = {J. Symplectic Geom.},
  	FJOURNAL = {The Journal of Symplectic Geometry},
  	VOLUME = {5},
  	YEAR = {2007},
  	NUMBER = {3},
  	PAGES = {281--356},
  	ISSN = {1527-5256},
  	MRCLASS = {53D45},
  	MRNUMBER = {2399678},
  	MRREVIEWER = {Michael J. Usher},
  	URL = {http://projecteuclid.org/euclid.jsg/1210083200},}

@article{CT10,
  	title={Virtual manifolds and localization},
  	author={Chen, Bohui and Tian, Gang},
  	journal={Acta Mathematica Sinica, English Series},
  	volume={26},
  	number={1},
  	pages={1--24},
  	year={2010},
  	publisher={Springer}
  }

@article {Del88,
	AUTHOR = {Delzant, Thomas},
	TITLE = {Hamiltoniens p\'{e}riodiques et images convexes de
	l'application moment},
	JOURNAL = {Bull. Soc. Math. France},
	FJOURNAL = {Bulletin de la Soci\'{e}t\'{e} Math\'{e}matique de France},
	VOLUME = {116},
	YEAR = {1988},
	NUMBER = {3},
	PAGES = {315--339},
	ISSN = {0037-9484},
	MRCLASS = {58F05},
	MRNUMBER = {984900},
	MRREVIEWER = {J.\ J.\ Duistermaat},
	URL = {http://www.numdam.org/item?id=BSMF_1988__116_3_315_0},}

@article{Du03,
  title={The degree of a {S}chubert variety},
  author={Duan, Haibao},
  journal={Advances in Mathematics},
  volume={180},
  number={1},
  pages={112--133},
  year={2003},
  publisher={Elsevier}}

@incollection {Dub14,
	AUTHOR = {Dubrovin, B.},
	TITLE = {Gromov-{W}itten invariants and integrable hierarchies of
	topological type},
	BOOKTITLE = {Topology, geometry, integrable systems, and mathematical
	physics},
	SERIES = {Amer. Math. Soc. Transl. Ser. 2},
	VOLUME = {234},
	PAGES = {141--171},
	PUBLISHER = {Amer. Math. Soc., Providence, RI},
	YEAR = {2014},
	ISBN = {978-1-4704-1871-7; 1-4704-1871-1},
	MRCLASS = {53D45 (14N35 37K10)},
	MRNUMBER = {3307147},
	MRREVIEWER = {Amin\ Gholampour},
	DOI = {10.1090/trans2/234/08},
	URL = {https://doi.org/10.1090/trans2/234/08},
}

@incollection {EGH00,
	AUTHOR = {Eliashberg, Y. and Givental, A. and Hofer, H.},
	TITLE = {Introduction to symplectic field theory},
	NOTE = {GAFA 2000 (Tel Aviv, 1999)},
	JOURNAL = {Geom. Funct. Anal.},
	FJOURNAL = {Geometric and Functional Analysis},
	YEAR = {2000},
	NUMBER = {Special Volume, Part II},
	PAGES = {560--673},
	ISSN = {1016-443X},
	MRCLASS = {53D45 (53D40)},
	MRNUMBER = {1826267},
	MRREVIEWER = {Kai Cieliebak},}

@incollection {FP00,
	AUTHOR = {Faber, C. and Pandharipande, R.},
	TITLE = {Logarithmic series and {H}odge integrals in the tautological
	ring},
	NOTE = {With an appendix by Don Zagier,
	Dedicated to William Fulton on the occasion of his 60th
	birthday},
	JOURNAL = {Michigan Math. J.},
	FJOURNAL = {Michigan Mathematical Journal},
	VOLUME = {48},
	YEAR = {2000},
	PAGES = {215--252},
	ISSN = {0026-2285,1945-2365},
	MRCLASS = {14H10 (11B73 14C15 32G15)},
	MRNUMBER = {1786488},
	MRREVIEWER = {Gary\ P.\ Kennedy},
	DOI = {10.1307/mmj/1030132716},
	URL = {https://doi.org/10.1307/mmj/1030132716},}

@article{Fl88,
	title={Morse theory for Lagrangian intersections},
	author={Floer, Andreas},
	journal={Journal of differential geometry},
	volume={28},
	number={3},
	pages={513--547},
	year={1988},
	publisher={Lehigh University}}

@book {Iv86,
	AUTHOR = {Iversen, B.},
	TITLE = {Cohomology of sheaves},
	SERIES = {Universitext},
	PUBLISHER = {Springer-Verlag, Berlin},
	YEAR = {1986},
	PAGES = {xii+464},
	ISBN = {3-540-16389-1},
	MRCLASS = {14F05 (14-01 18-01 54-01)},
	MRNUMBER = {842190},
	MRREVIEWER = {G. Horrocks},
	DOI = {10.1007/978-3-642-82783-9},
	URL = {https://doi-org.ezp.lib.cam.ac.uk/10.1007/978-3-642-82783-9},}

@article {FaP00,
    AUTHOR = {Faber, C. and Pandharipande, R.},
     TITLE = {Hodge integrals and {G}romov-{W}itten theory},
   JOURNAL = {Invent. Math.},
  FJOURNAL = {Inventiones Mathematicae},
    VOLUME = {139},
      YEAR = {2000},
    NUMBER = {1},
     PAGES = {173--199},
      ISSN = {0020-9910,1432-1297},
   MRCLASS = {14N35},
  MRNUMBER = {1728879},
MRREVIEWER = {Jim\ A.\ Bryan},
       DOI = {10.1007/s002229900028},
       URL = {https://doi.org/10.1007/s002229900028},}

@article{Ger13,
	title={Geometric transversality in higher genus {G}romov-{W}itten theory},
	author={Gerstenberger, Andreas},
	journal={arXiv preprint arXiv:1309.1426},
	year={2013}}

@article {Giv96,
    AUTHOR = {Givental, Alexander B.},
     TITLE = {Equivariant {G}romov-{W}itten invariants},
   JOURNAL = {Internat. Math. Res. Notices},
  FJOURNAL = {International Mathematics Research Notices},
      YEAR = {1996},
    NUMBER = {13},
     PAGES = {613--663},
      ISSN = {1073-7928,1687-0247},
   MRCLASS = {14D07 (14D05 14J32 14N10 32G20)},
  MRNUMBER = {1408320},
MRREVIEWER = {Claire\ Voisin},
       DOI = {10.1155/S1073792896000414},
       URL = {https://doi.org/10.1155/S1073792896000414},}

@incollection {Giv01,
    AUTHOR = {Givental, Alexander B.},
     TITLE = {Gromov-{W}itten invariants and quantization of quadratic
              {H}amiltonians},
      NOTE = {Dedicated to the memory of I. G. Petrovskii on the occasion of
              his 100th anniversary},
   JOURNAL = {Mosc. Math. J.},
  FJOURNAL = {Moscow Mathematical Journal},
    VOLUME = {1},
      YEAR = {2001},
    NUMBER = {4},
     PAGES = {551--568, 645},
      ISSN = {1609-3321,1609-4514},
   MRCLASS = {53D45 (14N35)},
  MRNUMBER = {1901075},
MRREVIEWER = {Domenico\ Fiorenza},
       DOI = {10.17323/1609-4514-2001-1-4-551-568},
       URL = {https://doi.org/10.17323/1609-4514-2001-1-4-551-568},}

@article {Giv01b,
    AUTHOR = {Givental, Alexander B.},
     TITLE = {Semisimple {F}robenius structures at higher genus},
   JOURNAL = {Internat. Math. Res. Notices},
  FJOURNAL = {International Mathematics Research Notices},
      YEAR = {2001},
    NUMBER = {23},
     PAGES = {1265--1286},
      ISSN = {1073-7928,1687-0247},
   MRCLASS = {53D45 (14N35)},
  MRNUMBER = {1866444},
MRREVIEWER = {Gilberto\ Bini},
       DOI = {10.1155/S1073792801000605},
       URL = {https://doi.org/10.1155/S1073792801000605},}

@article {Gr85,
    AUTHOR = {Gromov, M.},
     TITLE = {Pseudo holomorphic curves in symplectic manifolds},
   JOURNAL = {Invent. Math.},
  FJOURNAL = {Inventiones Mathematicae},
    VOLUME = {82},
      YEAR = {1985},
    NUMBER = {2},
     PAGES = {307--347},
      ISSN = {0020-9910,1432-1297},
   MRCLASS = {53C15 (32F25 53C57 57R15)},
  MRNUMBER = {809718},
MRREVIEWER = {Yakov\ Eliashberg},
       DOI = {10.1007/BF01388806},
       URL = {https://doi.org/10.1007/BF01388806},}

@incollection {GoWo22,
    AUTHOR = {Gonz\'{a}lez, E. and Woodward, C.},
     TITLE = {Quantum {K}irwan for quantum {K}-theory},
 BOOKTITLE = {Facets of algebraic geometry. {V}ol. {I}},
    SERIES = {London Math. Soc. Lecture Note Ser.},
    VOLUME = {472},
     PAGES = {265--332},
 PUBLISHER = {Cambridge Univ. Press, Cambridge},
      YEAR = {2022},
      ISBN = {978-1-108-79250-9; 978-1-108-87006-1},
   MRCLASS = {14C35 (14L24 19M99)},
  MRNUMBER = {4381905},
MRREVIEWER = {Cheolgyu\ Lee},}

@article{GKZ22b,
	title={Low-dimensional GKM theory},
	author={Goertsches, Oliver and Konstantis, Panagiotis and Zoller, Leopold},
	journal={arXiv preprint arXiv:2210.06234},
	year={2022}}

@article {GKZ20b,
	AUTHOR = {Goertsches, Oliver and Konstantis, Panagiotis and Zoller,
	Leopold},
	TITLE = {Symplectic and {K}\"{a}hler structures on biquotients},
	JOURNAL = {J. Symplectic Geom.},
	FJOURNAL = {The Journal of Symplectic Geometry},
	VOLUME = {18},
	YEAR = {2020},
	NUMBER = {3},
	PAGES = {791--813},
	ISSN = {1527-5256,1540-2347},
	MRCLASS = {53D05 (22F05 57S15 57S20)},
	MRNUMBER = {4142487},
	MRREVIEWER = {Lorenz\ J.\ Schwachh\"{o}fer},
	DOI = {10.4310/JSG.2020.v18.n3.a6},
	URL = {https://doi.org/10.4310/JSG.2020.v18.n3.a6},}

@article {GZ01,
	AUTHOR = {Guillemin, V. and Zara, C.},
	TITLE = {1-skeleta, {B}etti numbers, and equivariant cohomology},
	JOURNAL = {Duke Math. J.},
	FJOURNAL = {Duke Mathematical Journal},
	VOLUME = {107},
	YEAR = {2001},
	NUMBER = {2},
	PAGES = {283--349},
	ISSN = {0012-7094,1547-7398},
	MRCLASS = {53D20 (55N91 57S15)},
	MRNUMBER = {1823050},
	MRREVIEWER = {Thomas\ Delzant},
	DOI = {10.1215/S0012-7094-01-10724-2},
	URL = {https://doi.org/10.1215/S0012-7094-01-10724-2},
}

@incollection {GZ00,
	AUTHOR = {Guillemin, V. and Zara, C.},
	TITLE = {Equivariant de {R}ham theory and graphs [{MR}1701922
	(2001g:58033)]},
	BOOKTITLE = {Surveys in differential geometry},
	SERIES = {Surv. Differ. Geom.},
	VOLUME = {7},
	PAGES = {221--257},
	PUBLISHER = {Int. Press, Somerville, MA},
	YEAR = {2000},
	ISBN = {1-57146-069-1},
	MRCLASS = {58J20 (53D20 55N91)},
	MRNUMBER = {1919427},
	DOI = {10.4310/SDG.2002.v7.n1.a8},
	URL = {https://doi.org/10.4310/SDG.2002.v7.n1.a8},}

@misc{HW24,
	title={On Donaldson's 4-6 question}, 
	author={Amanda Hirschi and Luya Wang},
	year={2024},
	eprint={2309.07041},
	archivePrefix={arXiv},
	primaryClass={math.SG}}

@incollection {HZ99,
	AUTHOR = {Hofer, H. and Zehnder, E.},
	TITLE = {Pseudoholomorphic curves and dynamics},
	BOOKTITLE = {The {A}rnoldfest ({T}oronto, {ON}, 1997)},
	SERIES = {Fields Inst. Commun.},
	VOLUME = {24},
	PAGES = {225--239},
	PUBLISHER = {Amer. Math. Soc., Providence, RI},
	YEAR = {1999},
	ISBN = {0-8218-0945-8},
	MRCLASS = {53D10 (32Q65 37J45 53D05)},
	MRNUMBER = {1733579},
	MRREVIEWER = {Yi-Jen\ Lee},
	DOI = {10.1007/bf01475788},
	URL = {https://doi.org/10.1007/bf01475788},
}

@article {HWZ17,
	AUTHOR = {Hofer, H. and Wysocki, K. and Zehnder, E.},
	TITLE = {Applications of polyfold theory {I}: {T}he polyfolds of
	{G}romov-{W}itten theory},
	JOURNAL = {Mem. Amer. Math. Soc.},
	FJOURNAL = {Memoirs of the American Mathematical Society},
	VOLUME = {248},
	YEAR = {2017},
	NUMBER = {1179},
	PAGES = {v+218},
	ISSN = {0065-9266},
	ISBN = {978-1-4704-2203-5; 978-1-4704-4060-2},
	MRCLASS = {53D45 (14N35 57R17)},
	MRNUMBER = {3683060},
	MRREVIEWER = {William Liu},
	DOI = {10.1090/memo/1179},
	URL = {https://doi.org/10.1090/memo/1179},}

@book {HKK03,
    AUTHOR = {Hori, Kentaro and Katz, Sheldon and Klemm, Albrecht and
              Pandharipande, Rahul and Thomas, Richard and Vafa, Cumrun and
              Vakil, Ravi and Zaslow, Eric},
     TITLE = {Mirror symmetry},
    SERIES = {Clay Mathematics Monographs},
    VOLUME = {1},
      NOTE = {With a preface by Vafa},
 PUBLISHER = {American Mathematical Society, Providence, RI; Clay
              Mathematics Institute, Cambridge, MA},
      YEAR = {2003},
     PAGES = {xx+929},
      ISBN = {0-8218-2955-6},
   MRCLASS = {14J32 (14N35 32Q25 81T30 81T60)},
  MRNUMBER = {2003030},
MRREVIEWER = {Marcos\ Mari\~{n}o},}

@article{IP13,
	title={A natural {G}romov-{W}itten virtual fundamental class},
	author={Ionel, Eleny-Nicoleta and Parker, Thomas H},
	journal={arXiv preprint arXiv:1302.3472},
	year={2013}}

@book {KS94,
   	AUTHOR = {Kashiwara, Masaki and Schapira, Pierre},
   	TITLE = {Sheaves on manifolds},
   	SERIES = {Grundlehren der mathematischen Wissenschaften [Fundamental
   	Principles of Mathematical Sciences]},
   	VOLUME = {292},
   	PUBLISHER = {Springer-Verlag, Berlin},
   	YEAR = {1994},
   	PAGES = {x+512},
   	ISBN = {3-540-51861-4},
   	MRCLASS = {58G07 (18F20 32C38 35A27)},
   	MRNUMBER = {1299726},
   }

@article {GiKi95,
	AUTHOR = {Givental, Alexander and Kim, Bumsig},
	TITLE = {Quantum cohomology of flag manifolds and {T}oda lattices},
	JOURNAL = {Comm. Math. Phys.},
	FJOURNAL = {Communications in Mathematical Physics},
	VOLUME = {168},
	YEAR = {1995},
	NUMBER = {3},
	PAGES = {609--641},
	ISSN = {0010-3616,1432-0916},
	MRCLASS = {58D10 (14M15 57R57 58F05)},
	MRNUMBER = {1328256},
	MRREVIEWER = {Bruce\ Hunt},
	URL = {http://projecteuclid.org/euclid.cmp/1104272492},
}

@incollection {Kon95,
    AUTHOR = {Kontsevich, Maxim},
     TITLE = {Enumeration of rational curves via torus actions},
 BOOKTITLE = {The moduli space of curves ({T}exel {I}sland, 1994)},
    SERIES = {Progr. Math.},
    VOLUME = {129},
     PAGES = {335--368},
 PUBLISHER = {Birkh\"{a}user Boston, Boston, MA},
      YEAR = {1995},
      ISBN = {0-8176-3784-2},
   MRCLASS = {14N10 (14D22 14L30)},
  MRNUMBER = {1363062},
MRREVIEWER = {Anatoly\ Libgober},
       DOI = {10.1007/978-1-4612-4264-2\_12},
       URL = {https://doi.org/10.1007/978-1-4612-4264-2_12},}

@article {KM94,
	AUTHOR = {Kontsevich, M. and Manin, Yu.},
	TITLE = {Gromov-{W}itten classes, quantum cohomology, and enumerative
	geometry},
	JOURNAL = {Comm. Math. Phys.},
	FJOURNAL = {Communications in Mathematical Physics},
	VOLUME = {164},
	YEAR = {1994},
	NUMBER = {3},
	PAGES = {525--562},
	ISSN = {0010-3616},
	MRCLASS = {14N10 (53C15 58D10 58F05)},
	MRNUMBER = {1291244},
	MRREVIEWER = {Dietmar A. Salamon},
	URL = {http://projecteuclid.org.ezp.lib.cam.ac.uk/euclid.cmp/1104270948},}

@article {KM98,
    AUTHOR = {Kontsevich, M. and Manin, Yu.},
     TITLE = {Relations between the correlators of the topological
              sigma-model coupled to gravity},
   JOURNAL = {Comm. Math. Phys.},
  FJOURNAL = {Communications in Mathematical Physics},
    VOLUME = {196},
      YEAR = {1998},
    NUMBER = {2},
     PAGES = {385--398},
      ISSN = {0010-3616,1432-0916},
   MRCLASS = {14H10 (14D15 14D20 14N10 58D10)},
  MRNUMBER = {1645019},
MRREVIEWER = {Alexandre\ I.\ Kabanov},
       DOI = {10.1007/s002200050426},
       URL = {https://doi.org/10.1007/s002200050426},}

@incollection {La79,
	AUTHOR = {Lashof, Richard},
	TITLE = {Stable {$G$}-smoothing},
	BOOKTITLE = {Algebraic topology, {W}aterloo, 1978 ({P}roc. {C}onf., {U}niv.
	{W}aterloo, {W}aterloo, {O}nt., 1978)},
	SERIES = {Lecture Notes in Math.},
	VOLUME = {741},
	PAGES = {283--306},
	PUBLISHER = {Springer, Berlin},
	YEAR = {1979},
	ISBN = {3-540-09545-4},
	MRCLASS = {57R10 (57S10)},
	MRNUMBER = {557173},
	MRREVIEWER = {Michael\ C.\ Bix},
}

@article{IP19b,
  title={Relating VFCs on thin compactifications},
  author={Ionel, Eleny-Nicoleta and Parker, Thomas H},
  journal={Mathematische Annalen},
  volume={375},
  pages={845--893},
  year={2019},
  publisher={Springer}}

@book {Kir84,
    AUTHOR = {Kirwan, Frances Clare},
     TITLE = {Cohomology of quotients in symplectic and algebraic geometry},
    SERIES = {Mathematical Notes},
    VOLUME = {31},
 PUBLISHER = {Princeton University Press, Princeton, NJ},
      YEAR = {1984},
     PAGES = {i+211},
      ISBN = {0-691-08370-3},
   MRCLASS = {58F05 (14D25 14L30)},
  MRNUMBER = {766741},
MRREVIEWER = {Martin\ A.\ Guest},
       DOI = {10.2307/j.ctv10vm2m8},
       URL = {https://doi.org/10.2307/j.ctv10vm2m8},}

@article {LMP99,
	AUTHOR = {Lalonde, Fran\c cois and McDuff, Dusa and Polterovich, Leonid},
	TITLE = {Topological rigidity of {H}amiltonian loops and quantum
	homology},
	JOURNAL = {Invent. Math.},
	FJOURNAL = {Inventiones Mathematicae},
	VOLUME = {135},
	YEAR = {1999},
	NUMBER = {2},
	PAGES = {369--385},
	ISSN = {0020-9910,1432-1297},
	MRCLASS = {53D45 (53D35 55N35 55Q52 57R52 57S05)},
	MRNUMBER = {1666763},
	MRREVIEWER = {V\'eronique\ Lizan},
	DOI = {10.1007/s002220050289},
	URL = {https://doi.org/10.1007/s002220050289},}

@article {LT99,
	AUTHOR = {Li, Jun and Tian, Gang},
	TITLE = {Comparison of algebraic and symplectic {G}romov-{W}itten
	invariants},
	JOURNAL = {Asian J. Math.},
	FJOURNAL = {Asian Journal of Mathematics},
	VOLUME = {3},
	YEAR = {1999},
	NUMBER = {3},
	PAGES = {689--728},
	ISSN = {1093-6106,1945-0036},
	MRCLASS = {53D45 (14N35 57R17)},
	MRNUMBER = {1793677},
	MRREVIEWER = {David\ E.\ Hurtubise},
	DOI = {10.4310/AJM.1999.v3.n3.a7},
	URL = {https://doi.org/10.4310/AJM.1999.v3.n3.a7},}

@incollection {Liu13,
    AUTHOR = {Liu, Chiu-Chu Melissa},
     TITLE = {Localization in {G}romov-{W}itten theory and orbifold
              {G}romov-{W}itten theory},
 BOOKTITLE = {Handbook of moduli. {V}ol. {II}},
    SERIES = {Adv. Lect. Math. (ALM)},
    VOLUME = {25},
     PAGES = {353--425},
 PUBLISHER = {Int. Press, Somerville, MA},
      YEAR = {2013},
      ISBN = {978-1-57146-258-9},
   MRCLASS = {14N35 (14D23 14H10 14M25)},
  MRNUMBER = {3184181},
MRREVIEWER = {Hao\ Xu},}

@article {LiSh17,
	AUTHOR = {Liu, Chiu-Chu Melissa and Sheshmani, Artan},
	TITLE = {Equivariant {G}romov-{W}itten invariants of algebraic {GKM}
	manifolds},
	JOURNAL = {SIGMA Symmetry Integrability Geom. Methods Appl.},
	FJOURNAL = {SIGMA. Symmetry, Integrability and Geometry. Methods and
	Applications},
	VOLUME = {13},
	YEAR = {2017},
	PAGES = {Paper No. 048, 21},
	ISSN = {1815-0659},
	MRCLASS = {14N35 (14D20 14H10)},
	MRNUMBER = {3667222},
	MRREVIEWER = {Jie\ Zhou},
	DOI = {10.3842/SIGMA.2017.048},
	URL = {https://doi.org/10.3842/SIGMA.2017.048},}

@book {Ma99,
    AUTHOR = {Manin, Yuri I.},
     TITLE = {Frobenius manifolds, quantum cohomology, and moduli spaces},
    SERIES = {American Mathematical Society Colloquium Publications},
    VOLUME = {47},
 PUBLISHER = {American Mathematical Society, Providence, RI},
      YEAR = {1999},
     PAGES = {xiv+303},
      ISBN = {0-8218-1917-8},
   MRCLASS = {53D45 (14H10 14N35 18D50 32G34)},
  MRNUMBER = {1702284},
MRREVIEWER = {Alexandre\ I.\ Kabanov},
       DOI = {10.1090/coll/047},
       URL = {https://doi.org/10.1090/coll/047},
}

@book {MS12,
    AUTHOR = {McDuff, D. and Salamon, D.},
     TITLE = {{$J$}-holomorphic curves and symplectic topology},
    SERIES = {American Mathematical Society Colloquium Publications},
    VOLUME = {52},
   EDITION = {Second},
 PUBLISHER = {American Mathematical Society, Providence, RI},
      YEAR = {2012},
     PAGES = {xiv+726},
      ISBN = {978-0-8218-8746-2},
   MRCLASS = {53D45 (32Q65 53D35)},
  MRNUMBER = {2954391},
MRREVIEWER = {Mark Alan Branson},
}

@article {McD09,
	AUTHOR = {McDuff, Dusa},
	TITLE = {Hamiltonian {$S^1$}-manifolds are uniruled},
	JOURNAL = {Duke Math. J.},
	FJOURNAL = {Duke Mathematical Journal},
	VOLUME = {146},
	YEAR = {2009},
	NUMBER = {3},
	PAGES = {449--507},
	ISSN = {0012-7094,1547-7398},
	MRCLASS = {53D45 (14E08 53D35)},
	MRNUMBER = {2484280},
	MRREVIEWER = {Timothy\ Perutz},
	DOI = {10.1215/00127094-2009-003},
	URL = {https://doi.org/10.1215/00127094-2009-003},}

@article{Mu83,
  title={Towards an enumerative geometry of the moduli space of curves},
  author={Mumford, David},
  journal={Arithmetic and Geometry: Papers Dedicated to IR Shafarevich on the Occasion of His Sixtieth Birthday. Volume II: Geometry},
  pages={271--328},
  year={1983},
  publisher={Springer}}

@article {P16,
	AUTHOR = {Pardon, J.},
	TITLE = {An algebraic approach to virtual fundamental cycles on moduli
	spaces of pseudo-holomorphic curves},
	JOURNAL = {Geom. Topol.},
	FJOURNAL = {Geometry \& Topology},
	VOLUME = {20},
	YEAR = {2016},
	NUMBER = {2},
	PAGES = {779--1034},
	ISSN = {1465-3060},
	MRCLASS = {53D35 (37J10 53D37 53D40 53D42 53D45 54B40 57R17)},
	MRNUMBER = {3493097},
	MRREVIEWER = {Sonja Hohloch},
	DOI = {10.2140/gt.2016.20.779},
	URL = {https://doi.org/10.2140/gt.2016.20.779},}

@book {Pix13,
	AUTHOR = {Pixton, Aaron},
	TITLE = {The tautological ring of the moduli space of curves},
	NOTE = {Thesis (Ph.D.)--Princeton University},
	PUBLISHER = {ProQuest LLC, Ann Arbor, MI},
	YEAR = {2013},
	PAGES = {133},
	ISBN = {978-1303-09766-9},
	MRCLASS = {99-05},
	MRNUMBER = {3153424},
	URL =
	{http://gateway.proquest.com/openurl?url_ver=Z39.88-2004&rft_val_fmt=info:ofi/fmt:kev:mtx:dissertation&res_dat=xri:pqm&rft_dat=xri:pqdiss:3562218},}

@article {RRS08,
	AUTHOR = {Robbin, J. W. and Ruan, Y. and Salamon, D. A.},
	TITLE = {The moduli space of regular stable maps},
	JOURNAL = {Math. Z.},
	FJOURNAL = {Mathematische Zeitschrift},
	VOLUME = {259},
	YEAR = {2008},
	NUMBER = {3},
	PAGES = {525--574},
	ISSN = {0025-5874},
	MRCLASS = {58D15 (53D45 58B25 58D29)},
	MRNUMBER = {2395126},
	MRREVIEWER = {David E. Hurtubise},
	DOI = {10.1007/s00209-007-0237-x},
	URL = {https://doi.org/10.1007/s00209-007-0237-x},
}

@article {Swa21,
	AUTHOR = {Swaminathan, M.},
	TITLE = {Rel-{$C^\infty$} structures on {G}romov-{W}itten moduli
	spaces},
	JOURNAL = {J. Symplectic Geom.},
	FJOURNAL = {The Journal of Symplectic Geometry},
	VOLUME = {19},
	YEAR = {2021},
	NUMBER = {2},
	PAGES = {413--473},
	ISSN = {1527-5256},
	MRCLASS = {53D30 (53D45)},
	MRNUMBER = {4325409},
	DOI = {10.4310/JSG.2021.v19.n2.a4},
	URL = {https://doi.org/10.4310/JSG.2021.v19.n2.a4},}

@article {RT95,
	AUTHOR = {Ruan, Y. and Tian, G.},
	TITLE = {A mathematical theory of quantum cohomology},
	JOURNAL = {J. Differential Geom.},
	FJOURNAL = {Journal of Differential Geometry},
	VOLUME = {42},
	YEAR = {1995},
	NUMBER = {2},
	PAGES = {259--367},
	ISSN = {0022-040X},
	MRCLASS = {58D29 (14C05 14N10 58D10)},
	MRNUMBER = {1366548},
	MRREVIEWER = {Bernd Siebert},
	URL = {http://projecteuclid.org.ezp.lib.cam.ac.uk/euclid.jdg/1214457234},
}

@article{KKP03,
  title={Functoriality in intersection theory and a conjecture of Cox, Katz, and Lee},
  author={Kim, Bumsig and Kresch, Andrew and Pantev, Tony},
  journal={Journal of Pure and Applied Algebra},
  volume={179},
  number={1-2},
  pages={127--136},
  year={2003},
  publisher={Elsevier}
}

@misc{Schm23,
      title={Pseudocycle {G}romov-{W}itten invariants are a strict subset of polyfold {G}romov-{W}itten invariants}, 
      author={Wolfgang Schmaltz},
      year={2023},
      note={arXiv:2308.14204},
      archivePrefix={arXiv},
      primaryClass={math.SG}
}

@article {FO99,
	AUTHOR = {Fukaya, Kenji and Ono, Kaoru},
	TITLE = {Arnold conjecture and {G}romov-{W}itten invariant},
	JOURNAL = {Topology},
	FJOURNAL = {Topology. An International Journal of Mathematics},
	VOLUME = {38},
	YEAR = {1999},
	NUMBER = {5},
	PAGES = {933--1048},
	ISSN = {0040-9383},
	MRCLASS = {53D45 (37J10 37J45 53D40 57R17 57R58)},
	MRNUMBER = {1688434},
	MRREVIEWER = {David E. Hurtubise},
	DOI = {10.1016/S0040-9383(98)00042-1},
	URL = {https://doi-org.ezp.lib.cam.ac.uk/10.1016/S0040-9383(98)00042-1},}

@article {GKM98,
	AUTHOR = {Goresky, Mark and Kottwitz, Robert and MacPherson, Robert},
	TITLE = {Equivariant cohomology, {K}oszul duality, and the localization
	theorem},
	JOURNAL = {Invent. Math.},
	FJOURNAL = {Inventiones Mathematicae},
	VOLUME = {131},
	YEAR = {1998},
	NUMBER = {1},
	PAGES = {25--83},
	ISSN = {0020-9910,1432-1297},
	MRCLASS = {55N91 (14F25 14F32 16E99 18G10 55N33)},
	MRNUMBER = {1489894},
	MRREVIEWER = {Roy\ Joshua},
	DOI = {10.1007/s002220050197},
	URL = {https://doi.org/10.1007/s002220050197},}

@article {GKZ20,
    AUTHOR = {Goertsches, Oliver and Konstantis, Panagiotis and Zoller,
              Leopold},
     TITLE = {G{KM} theory and {H}amiltonian non-{K}\"{a}hler actions in
              dimension 6},
   JOURNAL = {Adv. Math.},
  FJOURNAL = {Advances in Mathematics},
    VOLUME = {368},
      YEAR = {2020},
     PAGES = {107141, 17},
      ISSN = {0001-8708,1090-2082},
   MRCLASS = {57S15 (55N91 57R91)},
  MRNUMBER = {4088417},
MRREVIEWER = {Andrea\ Spiro},
       DOI = {10.1016/j.aim.2020.107141},
       URL = {https://doi.org/10.1016/j.aim.2020.107141},
}

@article {GKZ23,
	AUTHOR = {Goertsches, Oliver and Konstantis, Panagiotis and Zoller,
	Leopold},
	TITLE = {Realization of {GKM} fibrations and new examples of
	{H}amiltonian non-{K}\"{a}hler actions},
	JOURNAL = {Compos. Math.},
	FJOURNAL = {Compositio Mathematica},
	VOLUME = {159},
	YEAR = {2023},
	NUMBER = {10},
	PAGES = {2149--2190},
	ISSN = {0010-437X,1570-5846},
	MRCLASS = {57R91 (32Q15 53D05 55N91)},
	MRNUMBER = {4634089},
	DOI = {10.1112/s0010437x2300742x},
	URL = {https://doi.org/10.1112/s0010437x2300742x},}

@article {GKZ22,
	AUTHOR = {Goertsches, Oliver and Konstantis, Panagiotis and Zoller,
	Leopold},
	TITLE = {G{KM} manifolds are not rigid},
	JOURNAL = {Algebr. Geom. Topol.},
	FJOURNAL = {Algebraic \& Geometric Topology},
	VOLUME = {22},
	YEAR = {2022},
	NUMBER = {7},
	PAGES = {3511--3532},
	ISSN = {1472-2747,1472-2739},
	MRCLASS = {57R91 (53D20 55N91)},
	MRNUMBER = {4545925},
	DOI = {10.2140/agt.2022.22.3511},
	URL = {https://doi.org/10.2140/agt.2022.22.3511},}

@article {GP99,
    AUTHOR = {Graber, T. and Pandharipande, R.},
     TITLE = {Localization of virtual classes},
   JOURNAL = {Invent. Math.},
  FJOURNAL = {Inventiones Mathematicae},
    VOLUME = {135},
      YEAR = {1999},
    NUMBER = {2},
     PAGES = {487--518},
      ISSN = {0020-9910},
   MRCLASS = {14C17 (14D20 14N10 14N35)},
  MRNUMBER = {1666787},
MRREVIEWER = {Paolo Aluffi},
       DOI = {10.1007/s002220050293},
       URL = {https://doi-org.ezp.lib.cam.ac.uk/10.1007/s002220050293},
}

@article {IP19,
	AUTHOR = {Ionel, Eleny-Nicoleta and Parker, Thomas H.},
	TITLE = {Thin compactifications and relative fundamental classes},
	JOURNAL = {J. Symplectic Geom.},
	FJOURNAL = {The Journal of Symplectic Geometry},
	VOLUME = {17},
	YEAR = {2019},
	NUMBER = {3},
	PAGES = {703--752},
	ISSN = {1527-5256},
	MRCLASS = {53D30 (53D45)},
	MRNUMBER = {4022212},
	MRREVIEWER = {David E. Hurtubise},
	DOI = {10.4310/JSG.2019.v17.n3.a4},
	URL = {https://doi-org.ezp.lib.cam.ac.uk/10.4310/JSG.2019.v17.n3.a4},}

@article{Mir03,
	title = {Hamiltonian {G}romov–{W}itten invariants},
	journal = {Topology},
	volume = {42},
	number = {3},
	pages = {525-553},
	year = {2003},
	issn = {0040-9383},
	doi = {https://doi.org/10.1016/S0040-9383(02)00023-X},
	url = {https://www.sciencedirect.com/science/article/pii/S004093830200023X},
	author = {Ignasi {Mundet i Riera}},}

@article{Joy17,
	title={Kuranishi spaces and {S}ymplectic {G}eometry},
	author={Joyce, Dominic},
	journal={Volume II},
	year={2017}
}

@article{Kur09,
	title={Introduction to GKM theory},
	author={Kuroki, Shintar{\^o}},
	journal={Science and Technology},
	volume={2009},
	number={0063179},
	pages={0063179},
	year={2009}
}

@incollection{LT98,
	AUTHOR = {Li, Jun and Tian, Gang},
	TITLE = {Virtual moduli cycles and {G}romov-{W}itten invariants of
	general symplectic manifolds},
	BOOKTITLE = {Topics in symplectic {$4$}-manifolds ({I}rvine, {CA}, 1996)},
	SERIES = {First Int. Press Lect. Ser., I},
	PAGES = {47--83},
	PUBLISHER = {Int. Press, Cambridge, MA},
	YEAR = {1998},
	MRCLASS = {53D45 (53D35 57R17 57R57)},
	MRNUMBER = {1635695},
}

@article {LiPa21,
    AUTHOR = {Lindsay, Nicholas and Panov, Dmitri},
     TITLE = {Symplectic and {K}\"{a}hler structures on {$\Bbb
              CP^1$}-bundles over {$\Bbb CP^2$}},
   JOURNAL = {Selecta Math. (N.S.)},
  FJOURNAL = {Selecta Mathematica. New Series},
    VOLUME = {27},
      YEAR = {2021},
    NUMBER = {5},
     PAGES = {Paper No. 93, 24},
      ISSN = {1022-1824,1420-9020},
   MRCLASS = {53D20},
  MRNUMBER = {4320842},
MRREVIEWER = {Antonella\ Nannicini},
       DOI = {10.1007/s00029-021-00705-7},
       URL = {https://doi.org/10.1007/s00029-021-00705-7},}

@article {MW17b,
AUTHOR = {McDuff, Dusa and Wehrheim, Katrin},
TITLE = {Smooth {K}uranishi atlases with isotropy},
JOURNAL = {Geom. Topol.},
FJOURNAL = {Geometry \& Topology},
VOLUME = {21},
YEAR = {2017},
NUMBER = {5},
PAGES = {2725--2809},
ISSN = {1465-3060},
MRCLASS = {53D35 (53D45 54B15 57R17 57R95)},
MRNUMBER = {3687107},
MRREVIEWER = {Michael J. Usher},
DOI = {10.2140/gt.2017.21.2725},
URL = {https://doi-org.ezp.lib.cam.ac.uk/10.2140/gt.2017.21.2725},}

@incollection {NWZ14,
	AUTHOR = {Nguyen, Khoa Lu and Woodward, Chris and Ziltener, Fabian},
	TITLE = {Morphisms of {C}oh{FT} algebras and quantization of the
	{K}irwan map},
	BOOKTITLE = {Symplectic, {P}oisson, and noncommutative geometry},
	SERIES = {Math. Sci. Res. Inst. Publ.},
	VOLUME = {62},
	PAGES = {131--170},
	PUBLISHER = {Cambridge Univ. Press, New York},
	YEAR = {2014},
	ISBN = {978-1-107-05641-1},
	MRCLASS = {14H10 (14N35 53D45)},
	MRNUMBER = {3380675},
	MRREVIEWER = {Hsian-Hua\ Tseng},}

@article{MiR01,
  title={Lifts of smooth group actions to line bundles},
  author={Riera, Ignasi Mundet I},
  journal={Bulletin of the London Mathematical Society},
  volume={33},
  number={3},
  pages={351--361},
  year={2001},
  publisher={Cambridge University Press}
}

@article{Par22, 
title={Enough vector bundles on orbispaces}, 
volume={158}, 
DOI={10.1112/S0010437X22007783}, 
number={11}, 
journal={Compositio Mathematica}, 
publisher={London Mathematical Society}, 
author={Pardon, John}, 
year={2022}, 
pages={2046–2081}}

@article {RT97,
	AUTHOR = {Ruan, Yongbin and Tian, Gang},
	TITLE = {Higher genus symplectic invariants and sigma models coupled
	with gravity},
	JOURNAL = {Invent. Math.},
	FJOURNAL = {Inventiones Mathematicae},
	VOLUME = {130},
	YEAR = {1997},
	NUMBER = {3},
	PAGES = {455--516},
	ISSN = {0020-9910},
	MRCLASS = {58D10 (14N10 57R57 58D29)},
	MRNUMBER = {1483992},
	MRREVIEWER = {Bernd Siebert},
	DOI = {10.1007/s002220050192},
	URL = {https://doi.org/10.1007/s002220050192},
}

@article {HS22,
	AUTHOR = {Hirschi, A. and Swaminathan, M.},
	TITLE = {Global {K}uranishi charts and a product formula in symplectic
	{G}romov-{W}itten theory},
	JOURNAL = {Selecta Math. (N.S.)},
	FJOURNAL = {Selecta Mathematica. New Series},
	VOLUME = {30},
	YEAR = {2024},
	NUMBER = {5},
	PAGES = {Paper No. 87, 74},
	ISSN = {1022-1824,1420-9020},
	MRCLASS = {53D45 (14N35)},
	MRNUMBER = {4807086},
	DOI = {10.1007/s00029-024-00982-y},
	URL = {https://doi.org/10.1007/s00029-024-00982-y},}

@article{Sa99,
	title={Lectures on Floer homology},
	author={Salamon, Dietmar},
	journal={Symplectic geometry and topology (Park City, UT, 1997)},
	volume={7},
	pages={143--229},
	year={1999},
	publisher={Citeseer}}

@article {Spi99,
    AUTHOR = {Spielberg, Holger},
     TITLE = {The {G}romov-{W}itten invariants of symplectic toric
              manifolds, and their quantum cohomology ring},
   JOURNAL = {C. R. Acad. Sci. Paris S\'{e}r. I Math.},
  FJOURNAL = {Comptes Rendus de l'Acad\'{e}mie des Sciences. S\'{e}rie I.
              Math\'{e}matique},
    VOLUME = {329},
      YEAR = {1999},
    NUMBER = {8},
     PAGES = {699--704},
      ISSN = {0764-4442},
   MRCLASS = {14N35 (14M25 53D45)},
  MRNUMBER = {1724149},
MRREVIEWER = {Ignasi\ Mundet-Riera},
       DOI = {10.1016/S0764-4442(00)88220-8},
       URL = {https://doi.org/10.1016/S0764-4442(00)88220-8},}

@article {Te12,
	AUTHOR = {Teleman, Constantin},
	TITLE = {The structure of 2{D} semi-simple field theories},
	JOURNAL = {Invent. Math.},
	FJOURNAL = {Inventiones Mathematicae},
	VOLUME = {188},
	YEAR = {2012},
	NUMBER = {3},
	PAGES = {525--588},
	ISSN = {0020-9910,1432-1297},
	MRCLASS = {57R56 (18D10 53D45)},
	MRNUMBER = {2917177},
	MRREVIEWER = {Julia\ Bergner},
	DOI = {10.1007/s00222-011-0352-5},
	URL = {https://doi.org/10.1007/s00222-011-0352-5},}

@inproceedings {TX17,
	AUTHOR = {Tian, Gang and Xu, Guangbo},
	TITLE = {The symplectic approach of gauged linear {$\sigma$}-model},
	BOOKTITLE = {Proceedings of the {G}\"{o}kova {G}eometry-{T}opology
	{C}onference 2016},
	PAGES = {86--111},
	PUBLISHER = {G\"{o}kova Geometry/Topology Conference (GGT), G\"{o}kova},
	YEAR = {2017},
	ISBN = {978-1-57146-340-1},
	MRCLASS = {53D45},
	MRNUMBER = {3676084},
	MRREVIEWER = {Eduardo\ A.\ Gonzalez},
}

@article {Tol98,
    AUTHOR = {Tolman, Susan},
     TITLE = {Examples of non-{K}\"{a}hler {H}amiltonian torus actions},
   JOURNAL = {Invent. Math.},
  FJOURNAL = {Inventiones Mathematicae},
    VOLUME = {131},
      YEAR = {1998},
    NUMBER = {2},
     PAGES = {299--310},
      ISSN = {0020-9910,1432-1297},
   MRCLASS = {53D20},
  MRNUMBER = {1608575},
MRREVIEWER = {Stephanie\ F.\ Singer},
       DOI = {10.1007/s002220050205},
       URL = {https://doi.org/10.1007/s002220050205},}

@article {Ush11,
	AUTHOR = {Usher, Michael},
	TITLE = {Deformed {H}amiltonian {F}loer theory, capacity estimates and
	{C}alabi quasimorphisms},
	JOURNAL = {Geom. Topol.},
	FJOURNAL = {Geometry \& Topology},
	VOLUME = {15},
	YEAR = {2011},
	NUMBER = {3},
	PAGES = {1313--1417},
	ISSN = {1465-3060,1364-0380},
	MRCLASS = {53D45 (53D40)},
	MRNUMBER = {2825315},
	MRREVIEWER = {David\ E.\ Hurtubise},
	DOI = {10.2140/gt.2011.15.1313},
	URL = {https://doi.org/10.2140/gt.2011.15.1313},}

@article {Wo98,
AUTHOR = {Woodward, Chris},
TITLE = {Multiplicity-free {H}amiltonian actions need not be
{K}\"{a}hler},
JOURNAL = {Invent. Math.},
FJOURNAL = {Inventiones Mathematicae},
VOLUME = {131},
YEAR = {1998},
NUMBER = {2},
PAGES = {311--319},
ISSN = {0020-9910,1432-1297},
MRCLASS = {53D20},
MRNUMBER = {1608579},
MRREVIEWER = {Stephanie\ F.\ Singer},
DOI = {10.1007/s002220050206},
URL = {https://doi.org/10.1007/s002220050206},}

@incollection {Sei99,
	AUTHOR = {Seidel, Paul},
	TITLE = {On the group of symplectic automorphisms of {$\bold C{\rm
	P}^m\times\bold C{\rm P}^n$}},
	BOOKTITLE = {Northern {C}alifornia {S}ymplectic {G}eometry {S}eminar},
	SERIES = {Amer. Math. Soc. Transl. Ser. 2},
	VOLUME = {196},
	PAGES = {237--250},
	PUBLISHER = {Amer. Math. Soc., Providence, RI},
	YEAR = {1999},
	ISBN = {0-8218-2075-3},
	MRCLASS = {53D45 (57R17)},
	MRNUMBER = {1736220},
	MRREVIEWER = {Miguel\ T.\ Abreu},
	DOI = {10.1090/trans2/196/11},
	URL = {https://doi.org/10.1090/trans2/196/11},
}

@book {Hsi75,
AUTHOR = {Hsiang, Wu-yi},
TITLE = {Cohomology theory of topological transformation groups},
SERIES = {Ergebnisse der Mathematik und ihrer Grenzgebiete [Results in
Mathematics and Related Areas]},
VOLUME = {Band 85},
PUBLISHER = {Springer-Verlag, New York-Heidelberg},
YEAR = {1975},
PAGES = {x+164},
MRCLASS = {57EXX},
MRNUMBER = {423384},
MRREVIEWER = {Theodore\ Chang},}

@incollection {Wi91,
	AUTHOR = {Witten, Edward},
	TITLE = {Two-dimensional gravity and intersection theory on moduli
	space},
	BOOKTITLE = {Surveys in differential geometry ({C}ambridge, {MA}, 1990)},
	PAGES = {243--310},
	PUBLISHER = {Lehigh Univ., Bethlehem, PA},
	YEAR = {1991},
	ISBN = {0-8218-0168-6},
	MRCLASS = {32G15 (14C17 14H15 32G81 58F07 81T40)},
	MRNUMBER = {1144529},
	MRREVIEWER = {Steven\ Rosenberg},}

@misc{Xu24,
	title={Quantum Kirwan map and quantum Steenrod operation}, 
	author={Guangbo Xu},
	year={2024},
	eprint={2405.12902},
	archivePrefix={arXiv},
	primaryClass={math.SG}}

@article{Zi08,
	title={Pseudocycles and integral homology},
	author={Zinger, Aleksey},
	journal={Transactions of the American Mathematical Society},
	volume={360},
	number={5},
	pages={2741--2765},
	year={2008}}

@misc{Zi17,
	title={Real {R}uan-{T}ian {P}erturbations}, 
	author={A. Zinger},
	year={2017},
	note={arXiv:1701.01420},
	archivePrefix={arXiv},
	primaryClass={math.SG}}

@article {Mei98,
	AUTHOR = {Meinrenken, Eckhard},
	TITLE = {Symplectic surgery and the {${\rm Spin}^c$}-{D}irac operator},
	JOURNAL = {Adv. Math.},
	FJOURNAL = {Advances in Mathematics},
	VOLUME = {134},
	YEAR = {1998},
	NUMBER = {2},
	PAGES = {240--277},
	ISSN = {0001-8708,1090-2082},
	MRCLASS = {58G10 (58F06)},
	MRNUMBER = {1617809},
	MRREVIEWER = {Christopher\ T.\ Woodward},
	DOI = {10.1006/aima.1997.1701},
	URL = {https://doi.org/10.1006/aima.1997.1701},}

\Addresses

\end{document}